\documentclass[10pt]{article}

\usepackage{url}
\usepackage{mathtools}
\usepackage{amssymb}
\usepackage{amsthm}
\usepackage{empheq}
\usepackage{latexsym}
\usepackage{enumitem}
\usepackage{eurosym}
\usepackage{dsfont}
\usepackage{appendix}
\usepackage{color} 
\usepackage[unicode]{hyperref}
\usepackage{frcursive}
\usepackage[utf8]{inputenc}
\usepackage[T1]{fontenc}
\usepackage{geometry}
\usepackage{multirow}
\usepackage{todonotes}
\usepackage{lmodern}
\usepackage{anyfontsize}
\usepackage{stmaryrd}
\usepackage{natbib}
\usepackage{cleveref}

\setcitestyle{numbers,open={[},close={]}}

\definecolor{red}{rgb}{0.7,0.15,0.15}
\definecolor{green}{rgb}{0,0.5,0}
\definecolor{blue}{rgb}{0,0,0.7}
\hypersetup{colorlinks, linkcolor={red},citecolor={green}, urlcolor={blue}}
			
\makeatletter \@addtoreset{equation}{section}

\newtheorem{theorem}{Theorem}[section]
\newtheorem{assumption}[theorem]{Assumption}
\newtheorem{corollary}[theorem]{Corollary}

\newtheorem{lemma}[theorem]{Lemma}
\newtheorem{proposition}[theorem]{Proposition}

\newtheorem{definition}[theorem]{Definition}
\newtheorem{remark}[theorem]{Remark}

\DeclareUnicodeCharacter{014D}{\=o}
\def \E{\mathbb{E}}
\def \F{\mathbb{F}}
\def \G{\mathbb{G}}

\def \M{\mathbb{M}}
\def \N{\mathbb{N}}

\def \P{\mathbb{P}}
\def \Q{\mathbb{Q}}
\def \R{\mathbb{R}}
\def \S{\mathbb{S}}

\def \Z{\mathbb{Z}}

\def \Pr{\mathrm{P}}

\def\Ac{{\cal A}}
\def\Bc{{\cal B}}
\def\Cc{{\cal C}}

\def\Fc{{\cal F}}
\def\Gc{{\cal G}}

\def\Lc{{\cal L}}

\def\Pc{{\cal P}}

\def\Sc{{\cal S}}

\def\Wc{{\cal W}}

\def\Fb{{\bar F}}

\def\Gcb{\overline \Gc}

\def\Pb{{\overline \P}}

\def\x{\times}
\def\eps{\varepsilon}

\def\Om{\Omega}
\def\Omt{\widetilde{\Omega}}
\def\Omh{\widehat{\Omega}}
\def\om{\omega}

\def\Omb{\overline{\Om}}
\def\omb{\bar \om}
\def\Fb{\overline{\F}}
\def\Gcb{\overline{\Gc}}
\def\Fcb{\overline{\Fc}}

\def\Fct{\widetilde{\Fc}}
\def\Ft{\widetilde{\F}}

\def\Wt{\widetilde{W}}
\def\Xt{\widetilde{X}}

\def\0{\mathbf{0}}

\def \bb{\mathbf{b}}
\def \mb{\mathbf{m}}
\def \nb{\mathbf{n}}

\def \qb{\mathbf{q}}
\def \xb{\mathbf{x}}

\def \zb{\mathbf{z}}

\def \mub{\overline{\mu}}

\def \muh{\widehat{\mu}}

\def \nub{\bar{\nu}}

\def\normeL2#1{\left\|{#1}\right\|_{L^2}}

\def\Pcb{\overline \Pc}

\def \Prod{\displaystyle\prod}

\def \Lim{\displaystyle\lim}

\def \Limsup{\displaystyle\limsup}

\def \Xbb{\mathbf{X}}

\def \Wbb{\mathbf{W}}

\title{Extended mean field control problem: a propagation of chaos result \footnote{The author is grateful to Dylan Possama\"{i} and Xiaolu Tan for helpful comments and suggestions.}}

\author{
	Mao Fabrice {\sc Djete}\footnote{Universit\'e Paris--Dauphine, PSL University, CNRS, CEREMADE, 75016 Paris,France, djete@ceremade.dauphine.fr. This work benefited from support of the r\'egion \^Ile--de--France.}
}
	\date{\today}

\begin{document}

\maketitle

\begin{abstract}
	In this paper, we study the $extended$ mean field control problem, which is a class of McKean--Vlasov stochastic control problem where the state dynamics and the reward functions depend upon the joint (conditional) distribution of the controlled state and the control process. By considering an appropriate controlled Fokker--Planck equation, we can formulate an optimization problem over a space of measure--valued processes and, under suitable assumptions, prove the equivalence between this optimization problem and the $extended$ mean--field control problem. Moreover, with the help of this new optimization problem, we establish the associated limit theory i.e. the $extended$ mean field control problem is the limit of a large population control problem where the interactions are achieved via the empirical distribution of state and control processes.
\end{abstract}

\section{Introduction}
The aim of this paper is to provide a rigorous connection between two stochastic control problems: the stochastic control problem of large population (or particles) interacting through the empirical distribution of their states and controls on the one hand, and the other hand the problem of control of stochastic dynamics depending upon the joint (conditional) distribution of the controlled state and the control, also called $extended$ mean field control problem.

\medskip
To fix the ideas, let us briefly described the problems. The large population stochastic control problem can be formulated as follows (see Section \ref{subsec:N-agents} for more details). Consider $N$--interacting controlled state processes $\Xbb:=(\Xbb^1,...,\Xbb^N)$  governed by the following system of stochastic differential equations:
\begin{align*}
        \mathrm{d}\Xbb^{i}_t
        &= b\big(t,\;\Xbb^{i}_{t},\;\big(\varphi^{N,\Xbb}_{s} \big)_{s \in [0,t]},\;\varphi^{N}_{t} ,\;\alpha^i_t \big) \mathrm{d}t 
        +
        \sigma \big(t,\;\Xbb^{i}_{t},\;\big(\varphi^{N,\Xbb}_{s} \big)_{s \in [0,t]},\;\varphi^{N}_{t},\;\alpha^i_t \big) \mathrm{d}\Wbb^i_t
        +
        \sigma_0 \mathrm{d}B_t,~~t \in [0,T],
        \\
        \varphi^{N}_{t}&:= \frac{1}{N}\sum_{i=1}^N \delta_{\big(\Xbb^{i}_{t},\;\alpha^i_t \big)}~\mbox{and}~\varphi^{N,\Xbb}_{t}:= \frac{1}{N}\sum_{i=1}^N \delta_{\Xbb^{i}_{t}}.
    \end{align*}
    Here $T>0$ is a fixed time horizon, $(B,\Wbb^1,...,\Wbb^N)$ are independent Brownian motions, $B$ is called the common noise and $(\alpha^1,..,\alpha^N)$ are some admissible controls chosen by a global planner. In this stochastic control problem, the global planner aims to maximise the average reward value given by
    \begin{align*}
        \frac{1}{N}\sum_{i=1}^N
        \E \bigg[
        \int_0^T L\big(t,\;\Xbb^{i}_{t},\;\big(\varphi^{N,\Xbb}_{s} \big)_{s \in [0,t]},\;\varphi^{N}_{t} ,\;\alpha^i_t \big) \mathrm{d}t 
        + 
        g \big( \Xbb^{i}_{T},\; \big(\varphi^{N,\Xbb}_{s} \big)_{s \in [0,T]} \big)
        \bigg].
    \end{align*}
    
    When $N$ goes to infinity, the expectation is that this problem $``converges"$ towards the $extended$ mean field control problem. Loosely speaking (see Section \ref{subsec:ExtMFC} for more details), in the $extended$ mean field control problem the objective is to control via $\alpha$ the state process $X$ which follows the stochastic differential equation of McKean--Vlasov type
    \begin{align*}
		\mathrm{d}X_t
		= 
		b \big(t,\;X_t,\;(\Lc(X_s | B ))_{s \in [0,t]},\;\Lc(X_t, \alpha_t | B ),\; \alpha_t \big) \mathrm{d}t
		+
	    \sigma\big(t,\;X_t,\;(\Lc(X_s | B ))_{s \in [0,t]},\;\Lc(X_t, \alpha_t | B ),\; \alpha_t \big) \mathrm{d} W_t
		+ 
		\sigma_0 \mathrm{d}B_t,
	\end{align*}
	in order to maximise the quantity
	 \begin{align*}
			\E \bigg[
				\int_0^T L(t,\;X_t,\;(\Lc(X_s | B ))_{s \in [0,t]},\;\Lc(X_t, \alpha_t | B ),\;\alpha_t) \mathrm{d}t 
				+ 
				g(X_T,\;(\Lc(X_s | B ))_{s \in [0,T]})
			\bigg],
		\end{align*}
	where $\Lc(X_t, \alpha_t | B )$ (resp $\Lc(X_t | B )$) denote the conditional distribution of the couple $(X_t, \alpha_t)$ (resp the state $X_t$) given the common noise $B.$ 
	
	\medskip
	The connection we are investigating, i.e. that the stochastic control problem of large population $converges$ towards the mean field control problem, is often called $limit\;theory$ or (controlled) $propagation\;of\;chaos.$ In contrast with the classical framework of McKean--Vlasov stochastic control problem which only considers the conditional distribution of $X_t,$ here, there is in addition the presence of the conditional distribution of $(X_t, \alpha_t).$ 
	Indeed, when there is no law of control i.e. no $\Lc(X_t, \alpha_t | B )$ but only $\Lc(X_t | B )$ in $(b,\sigma,L,g)$, these problems have been studied in the literature.
    Let us mention the work of \citeauthor*{snitzman1991topics} \cite{snitzman1991topics} which shows, for particular coefficients $(b,\sigma)$ in the absence of control (and the law of control), via some compactness arguments, a connection of this type. See also the papers of \citeauthor*{oelschlager1984martingale} \cite{oelschlager1984martingale} and \citeauthor*{gartner1988mckean} \cite{gartner1988mckean}, with no control and no law of control as well, which use martingale problem in the sense of \citeauthor*{stroock2007multidimensional} \cite{stroock2007multidimensional} adapted in the context of Mckean--Vlasov equation to prove similar results under minimal assumptions.

    In the controlled dynamic case but no $extended$ type, that is to say when the dynamic depends on the control but not its law, \citeauthor*{fischer2016continuous} \cite{fischer2016continuous} get a connection between the large population stochastic control problem and the $(extended)$ mean field control problem for the study of a mean--variance problem arising in finance. Another interesting work is that of \citeauthor*{budhiraja2012large} \cite{budhiraja2012large}, where they study the behavior of empirical measures of controlled interacting diffusion in order to prove a large deviation principle in a McKean--Vlasov framework. 
    Still without touching the case with law of control, the first papers that deal with the case with control under general assumptions are \citeauthor*{lacker2017limit} \cite{lacker2017limit} and \citeauthor*{djete2019general} \cite{djete2019general}.
    Thanks to an (extension of) martingale problem of \cite{stroock2007multidimensional}, as well as relaxed controls initiated by \citeauthor*{fleming1984stochastic} \cite{fleming1984stochastic}, and developed by \citeauthor*{el1987compactification} \cite{el1987compactification}, combined with compactness arguments adapted to the McKean--Vlasov setting, \cite{lacker2017limit} proves the connection between the two problems under general conditions on $(b,\sigma,L,g)$ without common noise. {\color{black} Indeed, the idea of using relaxed controls, i.e. control seen as probability measure of type $\delta_{\alpha_t}(\mathrm{d}u)\mathrm{d}t$ helps to find some compactness properties necessary for proving these types of results.}
    Following upon these ideas, \cite{djete2019general} develops a general overview of McKean--Vlasov or mean field control problem, and treats the case with common noise, which turns out to be a non trivial extension.
    
\medskip
    In the presence of the law of control, this propagation of chaos result is a natural expectation. In spite of appearances, this is not an easy extension. The aforementioned techniques do not work in this context. Two main reasons can explain the unsuitable aspect of the techniques mentioned above. Firstly, the continuity of the application $t \mapsto \Lc(X_t | B )$ (or $t \mapsto \varphi^{N,\Xbb}_{t}$) plays a crucial role. Indeed, the classical idea is to put this application in a canonical space, which is here the space $C([0,T];\Pc(\R^n))$ of continuous functions from $[0,T]$ into the space of probability measures on $\R^n,$ and via compactness arguments and martingale problem get this connection (see \cite{lacker2017limit}, and \cite{djete2019general} for the non--Markovian case with common noise). In our situation, this type of continuity is lost because we must take into account the application $t \mapsto \Lc(X_t, \alpha_t | B )$ (or $t \mapsto \varphi^{N}_{t}$) which does not have this property since the presence of control $\alpha$ can generate some discontinuities.
    Secondly, as highlighted in \cite{djete2019general}, proving a result of $propagation\;of\;chaos$ is extremely related to the search of the closure of
    the set of all probabilities that are the image measure of the controlled state process, the control and the conditional distribution of the controlled state process and control, i.e. {\color{black}$\Lc\big(X, \delta_{\alpha_t}(\mathrm{d}u)\mathrm{d}t, \Lc(X,\delta_{\alpha_t}(\mathrm{d}u)\mathrm{d}t | B ) \big).$ } Unfortunately, the natural space that one might think to answer this question is not a closed set due to another problem of continuity (see Remark \ref{remark_possible-relaxed-control} for a more thorough discussion). 
    
\medskip
    There are not many papers in the literature which study the mean field control problem with law of control and its connection with a large population stochastic control problem. To the best of our knowledge, only the recent papers of \citeauthor*{M_Lauriere-Tangpi} \cite{M_Lauriere-Tangpi} (with strong assumptions) and \citeauthor*{Mmotte-Pham_2019} \cite{Mmotte-Pham_2019} (for mean field Markov decision processes) treat the limit theory question. Most papers focus on the questions of existence and uniqueness of optimal control. \citeauthor*{Bacciaio} \cite{Bacciaio}, with the help of Pontryagin's maximum principle, obtain necessary and sufficient conditions to characterize the optimum with strong assumptions on the coefficients in a no common noise framework. \citeauthor*{pham2016dynamic} \cite{pham2016dynamic} (without common noise, with closed loop controls) and  \citeauthor*{djete2019mckean} \cite{djete2019mckean} establish the Dynamic Programming Principle (DPP for short) and give a Hamilton--Jacobi equation on a space of probability measures verified by the value function (heuristically proved in \cite{djete2019mckean}). Let us  also mention \citeauthor*{LackerCarmona-Ext} \cite{LackerCarmona-Ext},  \citeauthor*{ElieMastroliaPossamai-2018} \cite{ElieMastroliaPossamai-2018}, \citeauthor*{Lehalle-card} \cite{Lehalle-card}, \citeauthor*{alasseur-taher-matoussi-2020} \cite{alasseur-taher-matoussi-2020}, \citeauthor*{Casgrain_jaimungal_2018}  \cite{Casgrain_jaimungal_2018}, \citeauthor*{Lacker_Soret_2019}  \cite{Lacker_Soret_2019}, \citeauthor*{Tankov_Tinsi_Feron_2020} \cite{Tankov_Tinsi_Feron_2020} and \cite{M_Lauriere-Tangpi} who study similar problem in the mean field game framework called mean field game of controls or $extended$ mean field game, as well as our companion paper \citeauthor*{MFD-2020_MFG} \cite{MFD-2020_MFG} adapts the arguments of this paper to the context of mean field game of controls.
    
\medskip    
    In this article, our goal is to give some properties on the $extended$ mean field control problem and to show its connection with the large population stochastic control problem under general assumptions on $(b,\sigma,L,g)$ (see Theorem \ref{theo:approximation-N_strong} and Theorem \ref{theo:equality_strong-relaxed}). To bypass the difficulties highlighted above, we follow the idea mentioned in \cite{djete2019general} which is to introduce a new optimization problem by considering a suitable set of controls. This set must be the closure of some set of probability measures. In this framework, the appropriate space is the closure of all the probabilities that are the distributions of the conditional distribution of the state controlled process and the conditional distribution of the state controlled process and the control {\color{blue},} i.e. $\Lc \big( \Lc(X_t | B ))_{t \in [0,T]}, \delta_{\Lc(X_t, \alpha_t | B )}(\mathrm{d}m)\mathrm{d}t \big)$ (for more details see Section \ref{subsec-relaxed_controls}). Taking into account this type of probability turns out to be the key to solve the main difficulties. The characterization of its closure is possible by the appropriate use of (controlled) Fokker--Planck equation. Inspired by the techniques developed in the proofs of \citeauthor*{gyongy1986mimicking} \cite{gyongy1986mimicking}, especially \cite[Lemma 2.1]{gyongy1986mimicking} (an adaptation of \citeauthor*{Krylov-mimick} \cite{Krylov-mimick}) and \cite[Proposition 4.3]{gyongy1986mimicking} which are regularization results, we can determine the desired set thanks to a Fokker--Planck equation.  
    The conditions used on the coefficients are general, except the non--degeneracy of the volatility $\sigma.$ This assumption is capital to prove our main results. Apart from this assumption, our result appears to be one of the first to establish some general properties on $extended$ mean field control problem and to show its connection with the large population stochastic problem. 
    \citeauthor*{Lacker-closedloop} \cite{Lacker-closedloop} used similar techniques in the context of convergence of closed loop Nash equilibria, but his analysis focuses mainly on an adequate manipulation of \cite[Theorem 4.6]{gyongy1986mimicking}, while ours focuses on the techniques used for the proofs. Also, let us mention \citeauthor*{Lacker-Shkolnikov-Zhang_2020} \cite{Lacker-Shkolnikov-Zhang_2020} which establish a correspondence between Fokker--Planck equations and solutions of SDE in a McKean--Vlasov framework with common noise.
    

\medskip    
    The rest of the paper is structured as follows.  After introducing the notations and the probabilistic structure to give an adequate definition of the tools that are used throughout the paper, Section \ref{sec-McKVl} states all the main assumptions and carefully formulates first the large population stochastic control problem, then the strong formulation of the extended mean field control problem and finally the stochastic control of measure--valued processes. Next, in Section \ref{sec:main-results}, we present the main results of this paper: the equivalence between the strong formulation of extended mean field control problem and the stochastic control of measure--valued processes, and the propagation of chaos result i.e. the extended mean field control problem is, when $N$ goes to infinity, the limit of the large population stochastic control problem in presence of interactions through the empirical distribution of state and control processes.  Finally, \Cref{sec:proofs} is devoted to the proof of our main results and \Cref{sec:approx} provides some approximation results related to the Fokker--Planck equation needed in our proofs.
    
\vspace{0.5em}

{\bf \large Notations}.
	$(i)$
	Given a Polish space $(E,\Delta)$, $p \ge 1,$ we denote by $\Pc(E)$ the collection of all Borel probability measures on $E$,
	and by $\Pc_p(E)$ the subset of Borel probability measures $\mu$ 
	such that $\int_E \Delta(e, e_0)^p  \mu(\mathrm{d}e) < \infty$ for some $e_0 \in E$.
	We equip $\Pc_p(E)$ with the Wasserstein metric $\Wc_p$ defined by
	\[
		\Wc_p(\mu , \mu') 
		~:=~
		\bigg(
			\inf_{\lambda \in \Lambda(\mu, \mu')}  \int_{E \x E} \Delta(e, e')^p ~\lambda( \mathrm{d}e, \mathrm{d}e') 
		\bigg)^{1/p},
	\]
	where $\Lambda(\mu, \mu')$ denotes the collection of all probability measures $\lambda$ on $E \x E$ 
	such that $\lambda( \mathrm{d}e, E) = \mu(\mathrm{d}e)$ and $\lambda(E,  \mathrm{d}e') = \mu'( \mathrm{d}e')$. Equipped with $\Wc_p,$ $\Pc_p(E)$ is a Polish space (see \cite[Theorem 6.18]{villani2008optimal}). For any  $\mu \in \Pc(E)$ and $\mu$--integrable function $\varphi: E \to \R,$ we write
	\begin{align} \label{eq:def-cro}
	    \langle \varphi, \mu \rangle
	    =
	    \langle \mu, \varphi \rangle
	    :=
	    \int_E \varphi(e) \mu(\mathrm{d}e),
	\end{align}
	and for another metric space $(E^\prime,\Delta^\prime)$, we denote by $\mu \otimes \mu^\prime \in \Pc(E \x E')$ the product probability of any $(\mu,\mu^\prime) \in \Pc(E) \x \Pc(E^\prime)$.
	
	Given a probability space $(\Om, \Fc, \P)$ supporting a sub--$\sigma$--algebra $\Gc \subset \Fc$ then for a Polish space $E$ and any random variable $\xi: \Om \longrightarrow E$, both the notations $\Lc^{\P}( \xi | \Gc)(\om)$ and $\P^{\Gc}_{\om} \circ (\xi)^{-1}$ are used to denote the conditional distribution of $\xi$ knowing $\Gc$ under $\P$.
	
	
	
	
\medskip
	\noindent $(ii)$	
	For any $(E,\Delta)$ and $(E',\Delta')$ two Polish spaces, we shall refer to $C_b(E,E')$ to designate the set of continuous functions $f$ from $E$ into $E'$ such that $\sup_{e \in E} \Delta'(f(e),e'_0) < \infty$ for some $e'_0 \in E'.$
	Let {\color{black} $\N$ be the set of non--negative integers and} $\N^*$ be the notation of the set of positive integers, {\color{black} i.e. $\N^*:=\N\setminus\{0\}$.} Given non--negative integers $m$ and $n$, we denote by $\S^{m \x n}$ the collection of all $m \x n$--dimensional matrices with real entries, equipped with the standard Euclidean norm, which we denote by $|\cdot|$ regardless of the dimensions, for notational simplicity. 
	We also denote $\S^n:=\S^{n \times n}$, and denote by $0_{m \times n}$ the element in $\S^{m \times n}$ whose entries are all $0$, and by $\mathrm{I}_n$ the identity matrix in $\S^n$. For any matrix $a \in \S^{n}$ which is symmetric positive semi--definite, we write $a^{1/2}$ the unique symmetric positive semi--definite square root of the matrix $a.$
	Let $k$ be a positive integer, we denote by $C^k_b(\R^n;\R)$ the set of bounded maps $f: \R^n \longrightarrow \R$ with bounded continuous derivatives of order up to and including $k$. 
	Let $f: \R^n \longrightarrow \R$ be twice differentiable, we denote by $\nabla f$ and $\nabla^2f$ 
	the gradient and Hessian of $f$.


\medskip
    \noindent $(iii)$
	Let $T > 0$, and $(\Sigma,\rho)$ be a Polish space, we denote by $C([0,T], \Sigma)$ the space of all continuous functions on $[0,T]$ taking values in $\Sigma$.
	Then $C([0,T], \Sigma)$ is a Polish space under the uniform convergence topology, and we denote by $\|\cdot\|$ the uniform norm. 
	When $\Sigma=\R^k$ for some $k\in\N$, we simply write $\Cc^k := C([0,T], \R^k),$ also we shall denote by $\Cc^{k}_{\Wc}:=C([0,T], \Pc(\R^k)),$ and for $p \ge 1,$ $\Cc^{k,p}_{\Wc}:=C([0,T], \Pc_p(\R^k)).$ 

\medskip
	With a Polish space $E$, we denote by $\M(E)$ the space of all Borel measures $q( \mathrm{d}t,  \mathrm{d}e)$ on $[0,T] \x E$, 
	whose marginal distribution on $[0,T]$ is the Lebesgue measure $ \mathrm{d}t$, 
	that is to say $q( \mathrm{d}t, \mathrm{d}e)=q(t,  \mathrm{d}e) \mathrm{d}t$ for a family $(q(t,  \mathrm{d}e))_{t \in [0,T]}$ of Borel probability measures on $E$. We also consider the subset $\M_0(E) \subset \M(E)$ which is the collection of all $q \in \M(E)$ such that $q(\mathrm{d}t,\mathrm{d}e)=\delta_{\psi(t)}(\mathrm{d}e)\mathrm{d}t$ for some Borel measurable function $\psi:[0,T] \to E.$
	For any  $q \in \M(E)$, we define
	\begin{equation}\label{eq:lambda}
	q_{t \wedge \cdot}(\mathrm{d}s, \mathrm{d}e) :=  q(\mathrm{d}s, \mathrm{d}e) \big|_{ [0,t] \x E} + \delta_{e_0}(\mathrm{d}e) \mathrm{d}s \big|_{(t,T] \x E},\; \text{for some fixed $e_0 \in E$.}
	\end{equation}


\section{Extended mean field control problem} \label{sec-McKVl}

\medskip
    Let $(\ell, n) \in \N\times\N^\star,$ $(U,\rho)$ be a nonempty Polish space and $\Pc^n_U$ denote the space of all Borel probability measures on $\R^n \x U$ i.e. $\Pc^n_U:=\Pc(\R^n \x U).$ We give ourselves the following Borel measurable functions
	\[
		\big[b, \sigma,L \big]:[0,T] \x \R^n \x \Cc^n_{\Wc}  \x \Pc^n_U \x U \longrightarrow \R^n \x \S^{n \x n} \x \R\;\mbox{and}\; 
		g: \R^n \x \Cc^n_{\Wc} \longrightarrow \R.
	\]


    \begin{assumption} \label{assum:main1}  
    	The functions $[b,\sigma,L]$ are non--anticipative in the sense that, for all $(t, x, \pi,m, u) \in [0,T] \x \R^n \x \Cc^n_{\Wc} \x \Pc^n_U \x U$
	\[
		\big[b, \sigma, L \big] (t, x, \pi, m, u) 
		=
		\big[b, \sigma, L \big] (t, x, \pi_{t \wedge \cdot},m, u).
	\]
	
		Moreover, there exist positive constants $C$ and $p$ such that $p \ge 2 $ and
		
		$(i)$ $U$ is a compact space;
		
		$(ii)$ $b$ and $\sigma$ are continuous bounded functions, and $\sigma_0 \in \S^{n \x \ell}$ is constant;
		
		$(iii)$ one has for all $(t, x,x',\pi,\pi',m,m',u) \in [0,T] \x (\R^n)^2 \x (\Cc^n_{\Wc})^2  \x (\Pc^n_U)^2 \x U$
		\begin{align*}
			\big| [b,\sigma](t, x, \pi,m, u)
			-
			[b,\sigma](t, x', \pi',m', u)
			\big |
			~\le~
			C \big( | x-x' | + \sup_{s \in [0,T]}\Wc_p(\pi_s,\pi'_s) + \Wc_p (m,m')\big);
		\end{align*}		
	
 
		$(iv)$ for some constant $\theta >0$, one has, for all $(t,x,\pi,m,u) \in [0,T] \x \R^n \x \Cc^n_{\Wc} \x \Pc^n_U \x U$, 
		\begin{align*}
		    \theta \mathrm{I}_{n } \le ~\sigma \sigma^\top(t, x, \pi,m, u);
		\end{align*}

		$(v)$ the reward functions  $L$ and $g$ are continuous,
		and for all $(t,x,\pi,m,u) \in [0,T] \x \R^n \x \Cc^n_{\Wc} \x \Pc^n_U \x U$, one has
		\[  
			\big|L(t,x,\pi,m,u) \big| + |g(x,\pi)|
			\le  
			C \bigg [ 1+ | x |^{p} + \sup_{s \in [0,T]} \Wc_p(\pi_s,\delta_0)^{p} + \int_{\R^n} | x'|^{p} m(\mathrm{d}x',U)   \bigg ].
		\]
	\end{assumption}
\begin{remark}
	    These assumptions are standard and in the same spirit as those used in {\rm \cite{lacker2017limit}} and {\rm \cite{djete2019general}}, but with some specific modifications adapted to the context of this article. They ensure the well--posedness of the objects used throughout this paper. Due to the technical aspect of our paper, the point $(i)$ is considered essentially to simplify $($the presentation of$)$ the proofs. But, using the classical uniform integrability condition as in {\rm \cite{lacker2017limit}} and {\rm \cite{djete2019general}}, it is possible to work with $U$ a non--bounded set of $\R^n$ for instance.
	   The point $(iv)$ is the least classical assumption in the study in this problem. This is an important assumption for the proofs of our results, in particular to deal with the Fokker--Planck equations and the different SDEs considered in the proofs $($see {\rm \Cref{sec:approx}}$)$.
	\end{remark}

\subsection{The large population stochastic control problem} \label{subsec:N-agents}
    \medskip 
    In this section, we present the $N$--agent stochastic control problem or large population control problem. The study of this control problem when $N$ goes to infinity is one of the main objective of this paper.
	
	\medskip
	For a fixed $(\nu^1,\dots,\nu^N) \in \Pc_p(\R^n)^N,$ let
	\[ 
		\Om^N 
		~:=~
		(\R^n)^N \x (\Cc^{n})^N \x \Cc^{\ell}
	\]
	be the canonical space,
	with canonical variable $\Xbb_0 = (\Xbb^1_0,\dots,\Xbb^N_0),$ canonical processes $\Wbb = (\Wbb^1_s,\dots,\Wbb^N_s)_{0 \le s \le T}$ and $B = (B_s)_{0 \le s \le T}$, and probability measure $\P^N_{\nu}$ under which $\Xbb_0 \sim \nu_N:=\nu^1 \otimes \dots \otimes \nu^N$ and $(\Wbb, B)$ are standard Brownian motions independent of $\Xbb_0$.
	Let $\F^N = (\Fc^N_s)_{0 \le s \le T}$ be defined by
	$$
		\Fc^N_s 
		:=
		\sigma \big\{\Xbb_{0}, \Wbb_r, B_r, ~r \in [0,s] \big\},\;s \in [0,T].
	$$
	Let us denote by {\color{red}$\Ac_N(\nu_N)$} the collection of all $U$--valued $\F^N$--predictable processes.
	Then, given $\alpha:=(\alpha^1,\dots,\alpha^N) \in (\Ac_N(\nu_N))^N$, denote by $\Xbb^{\alpha}:=(\Xbb^{\alpha,1}_{\cdot},\dots,\Xbb^{\alpha,N}_{\cdot})$ the unique strong solution of the following system of SDEs, for each $i \in \{1,\dots,N\},$ $\E^{\P^N_{\nu}}\big[\|\Xbb^{\alpha,i}\|^p \big]< \infty,$
    \begin{align} \label{eq:N-agents_StrongMV_CommonNoise}
        \Xbb^{\alpha,i}_t
        =
        \Xbb^i_0
        +
        \int_0^t b\big(r,\Xbb^{\alpha,i}_{r},\varphi^{N,\Xbb}_{r \wedge \cdot},\varphi^{N}_{r} ,\alpha^i_r \big) \mathrm{d}r 
        +
        \int_0^t\sigma \big(r,\Xbb^{\alpha,i}_{r},\varphi^{N,\Xbb}_{r \wedge \cdot},\varphi^{N}_{r} ,\alpha^i_r \big) \mathrm{d}\Wbb^i_r
        +
        \sigma_0 B_t,\;\mbox{for all}\;t \in [0,T],
    \end{align}
	with 
	\[
	    \varphi^{N,\Xbb}_{t}(\mathrm{d} x) := \frac{1}{N}\sum_{i=1}^N \delta_{\big(\Xbb^{\alpha,i}_{t} \big)}(\mathrm{d} x)\;\;\mbox{and}\;\;
	    \varphi^{N}_{t}(\mathrm{d} x, \mathrm{d} u) := \frac{1}{N}\sum_{i=1}^N \delta_{\big(\Xbb^{\alpha,i}_{t},\;\alpha^i_t \big)}(\mathrm{d} x, \mathrm{d} u)
	    ,~
	    \mbox{for all}
	    ~~
	    t \in [0,T].
	\] 
    
    The value function $V^{N}_S(\nu^1,\dots,\nu^N)$ is defined by
    \begin{align} \label{eq:N-Value_StrongForm}
        V^{N}_S(\nu^1,\dots,\nu^N)
        :=
        \sup_{(\alpha^1,\dots,\alpha^N)} J^N(\alpha)\;\mbox{where}\;J^N(\alpha):=
        \frac{1}{N}\sum_{i=1}^N
        \E^{\P^N_{\nu}} \bigg[
        \int_0^T L\big(t,\Xbb^{\alpha,i}_{t},\varphi^{N,\Xbb}_{t \wedge \cdot},\varphi^{N}_{t} ,\alpha^i_t \big) \mathrm{d}t 
        + 
        g \big( \Xbb^{\alpha,i}_{T}, \varphi^{N,\Xbb}_{T \wedge \cdot} \big)
        \bigg],
    \end{align}
    which is well--posed under \Cref{assum:main1}.

\begin{remark} \label{remark:def-proba}
    {
    $(i)$ Our formulation allows for coefficients depending on the path of the empirical distribution of $\Xbb^\alpha,$ but can only accommodate a Markovian dependence with respect to $\Xbb^\alpha$ itself.  
    In some sense, we work on a non--Markovian framework w.r.t. the empirical distribution of $\Xbb^\alpha.$ Indeed, as we will see in {\rm \Cref{subsec-relaxed_controls}}, our point of view is to write the entire problem as an optimization involving mainly the empirical distribution of $\Xbb^\alpha$ i.e. $\varphi^{N,\Xbb}.$  Therefore our key variable is $\varphi^{N,\Xbb}$ $($not $\Xbb^\alpha$ $)$ and we can deal with its path, hence the non--Markovian aspect.    

\medskip    
    $(ii)$ Sometimes, the probability on $\Cc^n_{\Wc} \x \M(\Pc^n_U) \x \Cc^\ell$
        \begin{align} \label{def:proba-Nagents}
            \P(\alpha^1,...,\alpha^N)
            :=
            \P^N_\nu \circ \Big( (\varphi^{N,\Xbb}_t)_{t \in [0,T]}, \delta_{( \varphi^N_s )}(\mathrm{d}m) \mathrm{d}s, (B_t)_{t \in [0,T]} \Big)^{-1}
        \end{align}
        will be used to refer to $(\alpha^1,\dots,\alpha^N) \in (\Ac_N(\nu_N))^N$. 
        The notation $\Pc_S^N(\nu^1,\dots,\nu^N)$ will designate all probabilities of this type. The need for this space will become clearer in the following.
       
    }
    \end{remark}

\subsection{The extended mean field control problem} \label{subsec:ExtMFC}
On a fixed probability space, we formulate the classical McKean--Vlasov control problem with common noise including the (conditional) law of control.

	\medskip
	For a fixed  $\nu \in \Pc_p(\R^n)$,
	let
	\[ \label{eq:Omt}
		\Om 
		~:=~
		\R^n \x \Cc^{n} \x \Cc^{\ell}
	\]
	be the canonical space,
	with canonical variable $\xi,$ canonical processes $W = (W_t)_{0 \le t \le T}$ and $B = (B_t)_{0 \le t \le T}$, and probability measure $\P_{\nu}$ under which $\xi \sim \nu$ and $(W, B)$ are standard Brownian motions independent of $\xi$.
	Let $\F = (\Fc_s)_{0 \le s \le T}$ and $\G = (\Gc_s)_{0 \le s \le T}$ be defined by: for all $s \in [0,T],$
	$$
		\Fc_s 
		:=
		\sigma 
		\big\{\xi, W_r, B_r, ~r \in [0,s] \big\}
		~~~~~
		\mbox{and}
		~~~~~~~
		\Gc_s 
		:= 
		\sigma \big\{ B_r, ~r \in [0,s] \big\}.
	$$
	Let us denote by {\color{red}$\Ac(\nu)$} the collection of all $U$--valued processes $\alpha = (\alpha_s)_{0 \le s \le T}$ which are $\F$-predictable.
	Then, given $\alpha \in \Ac(\nu)$, let $X^{\alpha}$ be the unique strong solution of the SDE (see \cite[Theorem A.3]{djete2019mckean}): $\E^{\P_{\nu}} \big[\|X^\alpha\|^p \big]< \infty,$
	$X^{\alpha}_0= \xi$, and for $t \in [0,T]$,
	\begin{align} \label{eq:MKV_strong}
		X^{\alpha}_t
		= 
		X^{\alpha}_0 
		+
		\int_0^t b \big(r, X^{\alpha}_r,\mu^\alpha_{r \wedge \cdot}, \mub^{\alpha}_r, \alpha_r \big) \mathrm{d}r
		+
		\int_0^t \sigma\big(r, X^{\alpha}_r, \mu^\alpha_{r \wedge \cdot}, \mub^{\alpha}_r, \alpha_r \big) \mathrm{d} W_r
		+ 
		\sigma_0 B_t,
	\end{align}
	with $\mu^\alpha_r:=\Lc^{\P_{\nu}}\big(X^{\alpha}_r \big| \Gc_r \big)$ and $\mub^{\alpha}_r:=\Lc^{\P_{\nu}}\big(X^{\alpha}_r, \alpha_r \big| \Gc_r \big),$ for all $r \in [0,T].$

\medskip	
	Let us now introduce the following McKean--Vlasov control problem by
	    \begin{align} \label{eq:strong_Value}
	    	V_S(\nu)
			~:=~
			\sup_{\alpha \in \Ac(\nu)} 
			\Phi(\alpha)\;\mbox{where}	~~\Phi(\alpha):=\E^{\P_{\nu}} \bigg[
				\int_0^T L(t, X^{\alpha}_t,\mu^\alpha_{t \wedge \cdot}, \mub^{\alpha}_t, \alpha_t) \mathrm{d}t 
				+ 
				g(X^{\alpha}_T, \mu^{\alpha}) 
			\bigg].
		\end{align}

\begin{remark}
    Similarly to {\rm \cite{djete2019general}}, notice that, this formulation takes into account the case without common noise. Indeed, when $\ell =0,$ the space $\Cc^\ell$ and $\S^{n \x \ell}$ degenerate and become $\{0\}.$ Then, $B=0$ and, the filtration $\G$ is constant equal to the trivial $\sigma$--algebra $\{\emptyset,\Om \}.$ Therefore, there is no conditional distribution anymore.
\end{remark}

{\color{black}
\begin{remark}[Discussion on a possible relaxed extended mean field control problem] \label{remark_possible-relaxed-control}
    An adequate way to study the properties of $V_S$ and/or to give a limit theory is to find the closure $\overline{\Sc}(\nu)$ of some particular space $\Sc(\nu)$ for the Wasserstein topology. To simplify, let us take $\ell=0$ $($without common noise$)$, according to the classical ideas of relaxed controls, $\Sc(\nu):=\big\{ \P_{\nu} \circ \big(X^\alpha,\;\delta_{\alpha_t}(\mathrm{d}u)\mathrm{d}t\big)^{-1},\; \alpha \in \Ac(\nu) \big\}$ $($see discussion {\rm \citeauthor*{djete2019general} \cite{djete2019general}} and also  {\rm \citeauthor*{lacker2017limit} \cite{lacker2017limit}}$)$. 

\medskip
Following {\rm \cite{lacker2017limit}} and {\rm \cite{djete2019general}}, let us give an example to see why the $``$natural$"$ expected relaxed controls is not a $``$good$"$ set. Let $n=1,$ $U=[1,2],$ $\nu=\delta_0,$ $\sigma(t,x,\pi,m,u):=\big|\int_U u'\; m(\R^n,\mathrm{d}u')\big|$ and $b=0.$
Notice that $\Sc(\nu) \subset \Pc\big(\Cc^n \x \M(U) \big),$ then the canonical space is $\Omb_R:=\Cc^n \x \M(U).$ Denote $(X,\Lambda_t(\mathrm{d}u)\mathrm{d}t)$ the canonical process and $\Fb:=(\Fcb_t)_{t \in [0,T]}$ the canonical filtration. A naive relaxed controls is $\Pc_R(\nu) \subset \Pc(\Cc^n \x \M(U))$ defined by
\begin{align*}
    \Pc_R(\nu)
    :=
    \Big\{
        \Pb:\;\Pb(X_0=0)=1,\; (M^{\Pb,f}_t )_{t \in [0,T]}\;\mbox{is a }(\Pb,\Fb)\mbox{--martingale }\forall f \in C^2_b(\R)
    \Big\},
\end{align*}
where $M_t^{\Pb,f}:=f(X_t)-\frac{1}{2}\int_0^t \nabla^2 f(X_s) \E^{\Pb} \big[\int_U u\; \Lambda_s(\mathrm{d}u) \big]^2\mathrm{d}s.$

\medskip
But, $\Pc_R(\nu)$ defined in this way is not a closed set. Indeed the map $q \in \M(U) \to q_t \in \Pc(U)$ is not continuous for the Wasserstein topology. Therefore $\Pc_R(\nu)$ can not be the closure of $\Sc(\nu).$ More generally, as long as the coefficients $(b,\sigma)$ are non--linear w.r.t $m,$ this kind discontinuity will appear. Due to this type of lack of continuity, this approach cannot work. We need then to change the framework.
\end{remark}    
}

\subsection{Stochastic control of measure--valued processes}
\label{subsec-relaxed_controls}
    
    \medskip
    As previously mentioned, the classical approach of relaxed controls is not appropriate. To bypass the difficulty generated by the (conditional) distribution of control in this study, especially to prove the limit theory result or (controlled) propagation of chaos, we introduce a new stochastic control problem. Motivated by the Fokker--Planck equation verified by the couple $(\mu^\alpha,\mub^\alpha)$ from \eqref{eq:MKV_strong}, we give in this part an equivalent formulation of the extended mean field control problem which is less $``$rigid$"$.


\subsubsection{Measure--valued rules}    
    
    \medskip
    Recall that $\M := \M \big(\Pc^n_U \big)$ denotes the collection of all finite (Borel) measures $q(\mathrm{d}t, \mathrm{d}m)$ on $[0,T] \x \Pc^n_U$, 
	whose marginal distribution on $[0,T]$ is the Lebesgue measure $\mathrm{d}s$, 
	i.e. $q(\mathrm{d}s,\mathrm{d}m)=q(s,\mathrm{d}m)\mathrm{d}s$ for a measurable family $(q(s, \mathrm{d}m))_{s \in [0,T]}$ of Borel probability measures on $\Pc_U^n$.
	Let $\Lambda$ be the canonical element on $\M$.
	We then introduce a canonical filtration $\F^\Lambda = (\Fc^\Lambda_t)_{0 \le t \le T}$ on $\M$ by
	$$
		\Fc^\Lambda_t:= \sigma \big\{ \Lambda(C \x [0,s]): \forall s \le t, C \in \Bc(\Pc^n_U) \big\}.
	$$
	For each $q \in \M$, one has a disintegration property: $q(\mathrm{d}t, \mathrm{d}m) = q(t, \mathrm{d}m) \mathrm{d}t$, and there is a version of disintegration
	such that $(t, q) \mapsto q(t, \mathrm{d}m)$ is $\F^\Lambda$--predictable.
	

    \vspace{4mm}
    \noindent
    We denote by $(\mu,\Lambda,B)$ the canonical element on $\Omb:=\Cc^n_{\Wc} \x \M \x \Cc^\ell.$ The canonical filtration $\Fb = (\Fcb_t)_{t \in [0,T]}$ is then defined by: for all $t \in [0,T]$ 
    $$ 
        \Fcb_t
        :=
        \sigma 
        \big\{\mu_{t \wedge \cdot}, \Lambda_{t \wedge \cdot},B_{t \wedge \cdot}
        \big\}, 
    $$
    where $\Lambda_{t \wedge \cdot}$ denotes the restriction of $\Lambda$ on $\Pc^n_U \x [0,t]$ (see notation \ref{eq:lambda}). Notice that, we can choose a version of disintegration $\Lambda(\mathrm{d}m,\mathrm{d}t)=\Lambda_t(\mathrm{d}m)\mathrm{d}t$ with $(\Lambda_t)_{t \in [0,T]}$ a $\Pc(\Pc^n_U)$--valued $\Fb$--predictable process.

\medskip    
    Let us consider $\Lc$ the following generator: for all $(t,x,\pi,m,u) \in [0,T] \x \R^n \x \Cc^n_{\Wc} \x \Pc^n_U \x U $ and any $ \varphi \in C^{2}(\R^n)$
    \begin{align*}
        \Lc_t\varphi(x,\pi,m,u) 
        &:= 
        \frac{1}{2}  \text{Tr}\big[\sigma \sigma^\top(t,x,\pi,m,u) \nabla^2 \varphi(x) 
        \big] 
        + b(t,x,\pi,m,u)^\top \nabla \varphi(x),
    \end{align*}
    also we introduce, for every $f \in C^{2}(\R^n),$ $N_t(f)$:
    \begin{align} \label{equation:FP-characteristic}
        N_t(f)
        :=
        \langle f(\cdot-\sigma_0 B_t),\mu_t \rangle
	    -
	    \langle f,\mu_0 \rangle
	    -
	    \int_0^t \int_{\Pc^n_U}  \int_{\R^n \x U} \Lc_r[ f(\cdot-\sigma_0 B_r)]\big(x,\mu,m,u \big) m(\mathrm{d}x,\mathrm{d}u)  \Lambda_r(\mathrm{d}m) \mathrm{d}r,
    \end{align}
recall that $\langle \cdot, \cdot \rangle$ is defined in \eqref{eq:def-cro}.    
Notice that, under \Cref{assum:main1}, the integral in the definition $N(f)$ is well--posedness.
For each $\pi \in \Pc(\R^n),$ one considers the Borel set $\Z_{\pi}$ which is the set of probability measures $m$ on $\R^n \x U$ with marginal on $\R^n$ equal to $\pi$ i.e.
    \begin{align*}
        \Z_\pi:=\Big\{ m \in \Pc^n_U: m(\mathrm{d}x,U)=\pi(\mathrm{d}x) \Big\}.
    \end{align*}    
    
\begin{definition} \label{def:RelaxedCcontrol}
    For every $\nu \in \Pc(\R^n)$, $\Pr \in \Pc(\Omb)$ is a measure--valued rule if:
    \begin{itemize}
        \item $\Pr \big(\mu_0=\nu \big)=1$.
        
        \item $(B_t)_{t \in [0,T]}$ is a $(\Pr,\Fb)$ Wiener process starting at zero and for $\Pr$--almost every $\om \in \Omb$, $N_t(f)=0$ for all $f \in C^{2}_b(\R^n)$ and every $t \in [0,T]$ .
        
        \item For $\mathrm{d}\Pr \otimes \mathrm{d}t$ almost every $(t,\om) \in [0,T] \x \Om$, $ {\color{black}\Lambda_t\big(\Z_{\mu_t} \big)}=1.$
    \end{itemize}

\end{definition}
    
    We shall denote by $\Pcb_V(\nu)$ the set of all measure--valued rules with initial value $\nu.$

\subsubsection{Optimization problem} 
    Let us define, for all  $(\pi, q) \in \Cc^n_{\Wc} \x \M(\Pc^n_U),$
	    \begin{align*}
	        J(\pi, q)
	        :=
            \int_0^T \int_{\Pc^n_U} \int_{\R^n \x U} L\big(t,x,\pi,m,u \big) m(\mathrm{d}x,\mathrm{d}u) q_t(\mathrm{d}m)\mathrm{d}t 
            +
            \int_{\R^n} g\big(x,\pi \big)\pi_T(\mathrm{d}x).
	    \end{align*}
	Notice that, under \Cref{assum:main1}, the map $ J: \Cc^{n,p}_{\Wc} \x \M_p(\Pc^n_U) \to \R$ is continuous (see for instance \Cref{lemma:continuity}). 
    We can now define the measure--valued control problem: for each $\nu \in \Pc(\R^n),$ 
    \begin{align} \label{eq:optimization-FP}
        V_V(\nu) 
        :=
        \sup_{\Pr \in \Pcb_{V}(\nu)}
        \E^{\Pr}
        \big[
        J(\mu,\Lambda)
        \big].
    \end{align}	
	
	\begin{remark}
	    $(i)$ {\rm \Cref{def:RelaxedCcontrol}} is partly inspired by the Fokker--Planck equation verified by $(\mu^\alpha_t,\mub^\alpha_t)_{t \in [0,T]}$ $($see \eqref{eq:MKV_strong} and {\rm \Cref{prop-link-weakstrong}}$)$, in particular the last two points characterize this Fokker--Planck aspect. Indeed, $(\mu,\Lambda)$ satisfy: for all $(t,f)$
	    \begin{align*}
	        \langle f(\cdot-\sigma_0 B_t),\mu_t \rangle
	        =
	        \langle f,\mu_0 \rangle
	        +
	        \int_0^t \int_{\Pc^n_U}  \int_{\R^n \x U} \Lc_r[ f(\cdot-\sigma_0 B_r)]\big(x,\mu,m,u \big) m^x(\mathrm{d}u)\mu_r(\mathrm{d}x)  \Lambda_r(\mathrm{d}m) \mathrm{d}r,
	    \end{align*}
	    where for each $m \in \Pc^n_U,$ the Borel measurable function $\R^n \ni x \to m^x \in \Pc(U)$ verifies $m^x(\mathrm{d}u)m(\mathrm{d}x,U)=m(\mathrm{d}x,\mathrm{d}u).$ This kind of control turns out to be less $``$rigid$"$. Especially, $\Pcb_V(\nu)$ is a compact set for the Wasserstein topology $($see {\rm \Cref{theo:equality_strong-relaxed}}$)$.
\medskip	    
	    
	    $(ii)$ Working with these variables seems to be the key to better understand the problem and solves the principal difficulties. Mainly, to prove a limit theory result in this context, we make an approximation of the distribution of $(\mu,\Lambda)$ thanks to the distribution of variables of type $(\mu^\alpha,\delta_{\mub^{\alpha}_t}(\mathrm{d}m)\mathrm{d}t)$ and not thanks to the approximation of the law of $X.$ This approximation is achieved by using Fokker--Planck equations. To the best of our knowledge, looking at this kind of variable or $``$control$"$ has never been studied in the literature $($except in {\rm \cite{djete2019general}}, only for technical reasons$)$. 
	    
	\end{remark}

    \paragraph*{SDE formulation of measure--valued rules} Instead of presenting what we call measure--valued rules as solutions of Fokker--Planck equation, it is possible to formulate the measure--valued rules through solution of SDEs. Indeed, using an equivalence between Fokker--Planck equations and SDEs, there is an alternative way to formulate the measure--valued rules. In order to give more insights about the measure--valued rules, let us describe the SDEs formulation. For this purpose, we introduce the notion of extended relaxed control rules. We say that the tuple
    \begin{align*}
        (\Om,\Fc,\F,\P,W,B,X,\mu,\Lambda)
    \end{align*}
    is an extended relaxed control rule if
    
    \begin{enumerate}
        \item[$(i)$] $(\Om,\Fc,\F,\P)$ is a filtered probability space. on  $(\Om,\Fc,\F,\P),$ $(W,B)$ is a $\R^n \x \R^\ell$--valued $\F$--Brownian motion, $(X,\mu)$ is a $\R^n \x \Pc(\R^n)$--valued $\F$--adapted continuous process and $\Lambda$ is a $\Pc(\Pc^n_U)$--valued $\F$--predictable process.
        
        \item[$(ii)$] $X_0,$ $W$ and $(\mu,\Lambda,B)$ are independent.
        
        \item[$(iii)$]The process $\mu$ verifies $\mu_t=\Lc^\P(X_t|\mu_{t \wedge \cdot},\Lambda_{t \wedge \cdot},B_{t \wedge \cdot})=\Lc^\P(X_t|\mu,\Lambda,B)$ for all $t \in [0,T].$ The process $\Lambda$ is s.t. $\Lambda_t(\Z_{\mu_t})=1$ $\mathrm{d}\P \otimes \mathrm{d}t$ a.e. and the process $X$ is solution of: $\Lc^\P(X_0)=\nu$ and
        \begin{align*}
            \mathrm{d}X_t
            =
            \int_{\Pc^n_U} \int_U b(t,X_t,\mu,m,u) m^{X_t}(\mathrm{d}u) \Lambda_t(\mathrm{d}m)\mathrm{d}t
            +
            \Big(\int_{\Pc^n_U} \int_U \sigma \sigma^\top(t,X_t,\mu,m,u) m^{X_t}(\mathrm{d}u) \Lambda_t(\mathrm{d}m) \Big)^{1/2}\mathrm{d}W_t
            +
            \sigma_0\mathrm{d} B_t,
        \end{align*}
        where for each $m \in \Pc^n_U,$ the Borel measurable function $\R^n \ni x \to m^x \in \Pc(U)$ verifies $m^x(\mathrm{d}u)m(\mathrm{d}x,U)=m(\mathrm{d}x,\mathrm{d}u).$
    \end{enumerate}
    
    Using \cite[Theorem 1.3.]{Lacker-Shkolnikov-Zhang_2020} or an easy adaptation of \Cref{prop:weak-appr} or \Cref{prop:approximation-FP_BY_SDE-1}, we have the following equivalence result.
    \begin{proposition}
        $(i)$ For any extended relaxed control rule $(\Om,\Fc,\F,\P,W,B,X,\mu,\Lambda),$ $\P \circ \big(\mu,\Lambda,B \big)^{-1}$ belongs to $\Pcb_V(\nu).$
    
    \medskip    
        $(ii)$ Conversely, for any $\Pr \in \Pcb_V(\nu)$ measure--valued rule, there exists an extended relaxed control rule $(\Om,\Fc,\F,\P,W,B,X,\mu,\Lambda)$ s.t.
        \begin{align*}
            \Pr=\P \circ \big(\mu,\Lambda,B \big)^{-1}.
        \end{align*}
    \end{proposition}
    
\medskip	
	As stated in the preamble of this part, the measure--valued control problem is motivated by the Fokker--Planck equation verified by the couple $(\mu^\alpha,\mub^\alpha)$ of the strong formulation. Therefore, the strong controls i.e. $(\mu^\alpha,\mub^\alpha)_{\alpha \in \Ac(\nu)}$ can be seen as a special case of measure--valued rules. By taking into account the previous equivalence Proposition or by applying It\^o's formula, it is straightforward to deduce the following proposition.  
	\begin{proposition} \label{prop-link-weakstrong}
	    For each $\nu \in \Pc_p(\R^n),$ let us introduce
	\begin{align*}
	    \Pcb_S(\nu)
	    :=
	    \Big \{
	        \P_{\nu} \circ \big( (\mu^{\alpha}_t )_{t \in [0,T]}, \delta_{\mub^\alpha_r}(\mathrm{d}m)\mathrm{d}r, (B_t)_{t \in [0,T]} \big)^{-1},\;\alpha \in \Ac(\nu)
	    \Big \}.
	\end{align*}
	 one has $\Pcb_S(\nu) \subset \Pcb_V(\nu)$ and
	\begin{align*}
	        V_S(\nu)
	        =
	        \sup_{\mathrm{Q} \in \Pcb_{S}(\nu)}
            \E^{\mathrm{Q}}
            \big[
            J\big(\mu, \Lambda \big)
            \big].
	\end{align*}
	\end{proposition}
	
\begin{proof}
    Let $f \in C^2(\R^n)$ and $t \in [0,T],$ denote by $N_t(\mu,\Lambda,B)(f):=N_t(f).$ For any $\alpha \in \Ac(\nu),$ it is obvious that $\P_{\nu}(\mu^\alpha_0=\nu)=1$ and $\delta_{\mub^\alpha_t}\big(\Z_{\mu^\alpha_t} \big)=1$ $\mathrm{d}\P_{\nu} \otimes \mathrm{d}t$ a.e.. After applying  It\^o's formula with the process $X^\alpha_{\cdot}-\sigma_0B_{\cdot},$ and taking the conditional expectation w.r.t. the $\sigma$--field $\Gc_T,$ one has $N_t(\mu^\alpha, \delta_{\mub^\alpha_t}(\mathrm{d}m)\mathrm{d}t,B)(f)=0,$ $\P_{\nu}$--a.e. for all $(t,f).$ Then $\P_{\nu} \circ \big( \mu^\alpha, \delta_{\mub^\alpha_t}(\mathrm{d}m)\mathrm{d}t, B \big)^{-1} \in \Pcb_V(\nu).$ Therefore $\Pcb_S(\nu) \subset \Pcb_V(\nu).$ In addition, notice that
    \begin{align*}
        \Phi(\alpha)
        =
        \E^{\P_{\nu}} \bigg[
				\int_0^T \int_{\Pc^n_U} \langle L(t, \cdot,\mu^\alpha_{t \wedge \cdot}, m, \cdot), m \rangle \delta_{\mub^\alpha_t}(\mathrm{d}m) \mathrm{d}t 
				+ 
				\langle g(\cdot, \mu^{\alpha}), \mu^\alpha_T \rangle 
			\bigg],
    \end{align*}
    consequently $V_S(\nu)=\sup_{\mathrm{Q} \in \Pcb_{S}(\nu)}\E^{\mathrm{Q}}\big[J\big(\mu, \Lambda \big)\big].$
\end{proof}

\section{Main results} \label{sec:main-results}
Now, we formulate the main results of this paper.

\begin{theorem}[Equivalence] \label{theo:equality_strong-relaxed}
	Let \Cref{assum:main1} hold true and $\nu \in \Pc_{p'}(\R^n),$ with $p' > p.$ Then $\Pcb_V(\nu)$ is convex and compact for the Wasserstein metric $\Wc_p.$ Moreover 
	
\medskip	
	\hspace{1em} $(i)$ When $\ell \neq 0,$ for $\Wc_p,$ the set $\Pcb_S(\nu)$ is dense in $\Pcb_V(\nu).$ 

\medskip	
	\hspace{1em}$(ii)$ When $\ell = 0,$ for any $\mathrm{P} \in \Pcb_V(\nu),$ there exists a family $(\mathrm{P}^k_{z})_{(k,z) \in \N^* \x [0,1]} \subset \Pcb_S(\nu)$ such that for each $k \in \N^*,$ $[0,1] \ni z \to \mathrm{P}^k_{z} \in \Pc(\Omb)$ is Borel measurable and one gets $\Lim_{k \to \infty} \Wc_p \bigg( \int_0^1 \Pr^k_z\;\mathrm{d}z,\;\;\Pr \bigg)=0.$
	
\medskip	
	Consequently
	$$
	    V_V(\nu)=V_S(\nu),
	$$
	and there exists $\Pr^{\star} \in \Pcb_V(\nu)$ such that $V_S(\nu)=\E^{\Pr^\star}\big[ J\big(\mu,\Lambda\big)\big].$
	
	
\end{theorem}

\begin{remark}
    $(i)$ As in {\rm \cite{djete2019general}} $($see also {\rm \cite{lacker2016general}} and {\rm \cite{MFD-2020_MFG}} for the mean field game context$)$, there are some specificities when $\ell=0$. Indeed, when $\ell=0,$ $(\mu^\alpha,\mub^\alpha)$ are deterministic, but $(\mu,\Lambda)$ can still be random, therefore, except in particular situation, it is not possible to approximate the non atomic measure $\Pr$ by a sequence of atomic measure of type $\delta_{(\mu^\alpha, \delta_{\mub^\alpha_s}(\mathrm{d}m)\mathrm{d}s)}.$ However, a randomisation is possible as mentioned in $(ii)$ of {\rm \Cref{theo:equality_strong-relaxed}}.
    
    \medskip
    $(ii)$ {\rm \Cref{theo:equality_strong-relaxed}} and the following {\rm \Cref{theo:approximation-N_strong}} are in the same spirit that {\rm Theorem 3.1} and {\rm Theorem 3.6} of {\rm \cite{djete2019general}}. The main difference is the presence of the distribution of controlled state and control, and this particularity turns out to be a non trivial extension $($see discussion in {\rm \Cref{subsec:ExtMFC}}$)$. 
\end{remark}

\begin{theorem}[Propagation of chaos] \label{theo:approximation-N_strong}
        
     Let {\rm \Cref{assum:main1}} hold true, $p' > p$ and $(\nu^i)_{i \in \N^*} \subset \Pc_{p'}(\R^n)$ satisfying $\sup_{N \ge 1} \frac{1}{N}\sum_{i=1}^N \int_{\R^n}| x'|^{p'} \nu^i(\mathrm{d} x') < \infty.$
    Then
    $$
        \lim_{N \to \infty} \bigg|
        V_S^N\big(\nu^1,\dots,\nu^N \big)- V_S\Big(\frac{1}{N}\sum_{i=1}^N \nu^i \Big)\bigg|=0 .
    $$
\end{theorem}

\medskip
    Finally, we provide some properties of optimal control of our problem. For any $\nu \in \Pc(\R^n),$ denote by $\Pcb^\star_V(\nu)$ the set of optimal control i.e. $\Pr^\star \in \Pcb^\star_V(\nu)$ if $\Pr^\star \in \Pcb_V(\nu)$ and $V_V(\nu)=\E^{\Pr^\star}\big[ J\big(\mu,\Lambda\big)\big].$
    
    \begin{proposition} \label{proposition:optimal-control-convergence}
        Suppose that the conditions of {\rm \Cref{theo:approximation-N_strong}} hold. Let $\lim_{N \to \infty} \Wc_{p} \big(\frac{1}{N}\sum_{i=1}^N \nu^i,\nu \big)=0$ with $\nu \in \Pc_p(\R^n).$

\medskip    
        \hspace{1em} $(i)$ For any sequence of non negative numbers $(\varepsilon_N)_{N \in \N^*}$ verifying $\Lim_{N \to \infty} \varepsilon_N=0,$ if $(\Pr^N)_{N \in \N^*}$ is the sequence satisfying $\Pr^N:=\P(\alpha^{1},\dots,\alpha^{N})$ $($see {\rm \eqref{def:proba-Nagents}}$)$ with 
        \begin{align} \label{cond-optimality}
           \mbox{for each }N \in \N^*,\;\alpha^{i} \in \Ac_N(\nu_N)\; \forall i \in [\![ 1,N]\!]\;\mbox{and}\; V_S^N(\nu^1,\dots,\nu^N) - \varepsilon_N \le \E^{\Pr^N}\big[J\big(\mu, \Lambda \big)\big],
        \end{align} then 
        $$
            \Lim_{N \to \infty} \inf_{\Pr^\star \in \Pcb^\star_V(\nu)} \Wc_{p} \big( \Pr^N, \Pr^\star \big)=0.
        $$ 
        
        \medskip
        \hspace{1em} $(ii)$ Moreover, for each $\Pr^\star \in \Pcb^\star_V(\nu),$ there exist $(\varepsilon_N)_{N \in \N^*} \subset (0,\infty)$ verifying $\Lim_{N \to \infty} \varepsilon_N=0$ and a sequence $(\Pr^{\star,N})_{N \in \N^*}$ satisfying $\Pr^{\star,N}:=\P(\alpha^{\star,1},\dots,\alpha^{\star,N})$ and {\rm condition \ref{cond-optimality}} s.t. $\Lim_{N \to \infty} \Wc_p(\Pr^{\star,N},\Pr^\star)=0.$
    \end{proposition}
    
    \begin{remark}
        $(i)$ The previous proposition shows that any $\varepsilon_N$--optimal control of the large population stochastic control problem converges towards an optimal control of the McKean--Vlasov stochastic control problem in distribution sense. In particular when there exists a unique strong optimal control of the McKean--Vlasov control problem, any $\varepsilon_N$--optimal control of the large population control problem converges towards this control. 
        
        \medskip
        $(ii)$ To the best of our knowledge, {\rm \Cref{theo:approximation-N_strong}} and {\rm \Cref{proposition:optimal-control-convergence}} seem to be the first result under these general assumptions to provide these types of convergence results. As mentioned in the introduction, other authors treat these questions but in a particular framework. For instance, while dealing with the convergence of Nash equilibria, {\rm \cite{M_Lauriere-Tangpi}} gives a limit theory result for the extended mean field control problem. The framework of {\rm \cite{M_Lauriere-Tangpi}} is less general than ours, in particular, they consider a situation without common noise $(\sigma_0=0)$, with volatility $\sigma$ constant. Besides, they need assumptions over $(b,g,L)$ via the Hamiltonian which lead to the uniqueness of the optimum and, these assumptions are sometimes quite difficult to verify in practice. However, it should be mentioned that the results of {\rm \cite{M_Lauriere-Tangpi}} include a rate of convergence that we do not provide. Let us also mention {\rm \cite{Mmotte-Pham_2019}} which treats these questions of convergence but for Markov decision processes in discrete time.
    \end{remark}

\medskip
    The next corollary is just a combination of Theorem \ref{theo:approximation-N_strong} and \cite[Proposition 4.15]{djete2019general}. It states that if a strong control is close enough to the optimum value of the mean field control problem, from this control, we can construct $N$ agents which are close to the optimum of the large population stochastic control problem.

\begin{corollary}
    Let {\rm \Cref{assum:main1}} hold true. Let $\nu \in \Pc_{p'}(\R^n),$ with $p' > p,$ $(\varepsilon_N)_{N \in \N^*}$ be a sequence of non negative real such that $\Lim_{N \to \infty} \varepsilon_N=0.$ Also, for each $N \in \N^*,$ let $\alpha^N \in \Ac(\nu)$ satisfying $\alpha^N_t=\phi^N(t,\xi,W_{t \wedge \cdot}, B_{t \wedge \cdot})$ $\P_{\nu}$ a.e. for all $t \in [0,T]$ with a Borel function $\phi^N: [0,T] \x \R^n \x \Cc^d \x \Cc^\ell \to U,$ and 
    \begin{align*}
        V_S(\nu) - \varepsilon_N 
        \le 
         \Phi(\alpha^N).
    \end{align*}
    Then, there exists $(\delta_N)_{N \in \N^*} \subset (0,\infty)$ s.t. $\Lim_{N \to \infty} \delta_N=0$ and $(\alpha^{1,N},\dots,\alpha^{N,N}) \in \Ac_N(\nu_N)^N$ with $\nu_N:=\nu \otimes \dots \otimes \nu$ satisfying
    \[
        \alpha^{i,N}_t
        =
        \phi^N (t,\Xbb^i_0,\Wbb^i_{t \wedge \cdot}, B_{t \wedge \cdot}),\;\P^N_{\nu}\;\mbox{a.e.}\;\mbox{for all}\;t \in [0,T]\;\;\mbox{and}\;\;V^N_S(\nu,\dots,\nu) - \delta_N 
        \le 
        J^N(\alpha^{1,N},\dots,\alpha^{N,N}).
    \]
\end{corollary}


\section{Proofs of the main results} \label{sec:proofs}

In this part, we will present the proof of the main results of this paper namely \Cref{theo:equality_strong-relaxed} and \Cref{theo:approximation-N_strong}. Some proofs use the results from \Cref{sec:approx} which will be proven just after.

\subsection{Equivalence result}

This section is devoted to the proof of \Cref{theo:equality_strong-relaxed}. To achieve this proof, we provide an approximation of measure--valued rule by McKean--Vlasov processes. Before starting the proofs, by shifting some probabilities, let us give a reformulation of measure--valued rules. For all $(t,\bb,\pi,m) \in [0,T] \x \Cc^\ell \x \Cc^n_{\Wc} \x \Pc^n_U,$
\begin{align} \label{eq:shift-proba-initial}
    \pi_t[\bb](\mathrm{d}y):= \int_{\R^n} \delta_{\big(y'+\sigma_0 \bb_t \big)}(\mathrm{d}y) \pi_t(\mathrm{d}y'),\;\; m[\bb_t](\mathrm{d}u,\mathrm{d}y):=\int_{\R^n \x U} \delta_{(y'+\sigma_0 \bb_t)}(\mathrm{d}y)m(\mathrm{d}u,\mathrm{d}y')
\end{align}
and any $q \in \M,$
\begin{align} \label{eq:shift-proba-M}
        q_t[\bb](\mathrm{d}m):=\int_{\Pc^n_U} \delta_{\big(m^{'}[\bb_t] \big)}(\mathrm{d}m) q_t(\mathrm{d}m').
    \end{align}
In the same way, let us consider the $``shifted"$ generator $\widehat{\Lc},$
\begin{align} \label{eq:shift-generator}
        \widehat{\Lc}_t[ \varphi](y,\bb,\pi,m,u):= \frac{1}{2}  \mathrm{Tr}\big[\sigma \sigma^\top(t,y+\sigma_0\bb_t,\pi_t[\bb_t],m[\bb_t],u) \nabla^2 \varphi(y) \big] +b(t,y+\sigma_0\bb_t,\pi_t[\bb_t],m[\bb_t],u)^\top \nabla \varphi(y).
    \end{align}

Next, on the canonical filtered space $(\Omb,\Fb)$ (see Section \ref{subsec-relaxed_controls}), let $(\vartheta_t)_{t \in [0,T]}$ be the $\Pc(\R^n)$--valued $\Fb$--adapted continuous process and $(\Theta_t)_{t \in [0,T]}$ be the $\Pc^n_U$--valued $\Fb$--predictable process defined by
    \begin{align} \label{eq:shift-proba}
        \vartheta_t(\omb):=\mu_t(\omb)[-B(\omb)]\;\;\mbox{and}\;\;\Theta_t(\omb)(\mathrm{d}m):=\Lambda_t(\omb)[-B(\omb)](\mathrm{d}m),\;\mbox{for all}\;(t,\omb)\in [0,T] \x \Omb.
    \end{align}    
    
The next result follows immediately, so we omit the proof.

\begin{lemma} \label{lemma:reformulation}
    Let $\Pr \in \Pcb_V(\nu).$ Then, $\Theta_t(\Z_{\vartheta_t})=1,$ $\mathrm{d}\Pr \otimes \mathrm{d}t,$ a.e. $(t,\omb) \in [0,T] \x \Omb,$ and $\Pr$--a.e. $\omb \in \Omb,$ for all $(f,t) \in C^{2}_b(\R^n) \x [0,T],$
    \begin{align*}
        N_t(f)
        =
        \langle f,\vartheta_t \rangle
	    -
	    \langle f,\nu \rangle
	    -
	    \int_0^t \int_{\Pc^n_U} \int_{\R^n \x U} \widehat{\Lc}_r f(y,B,\vartheta,m,u) m(\mathrm{d}u,\mathrm{d}y) \Theta_r(\mathrm{d}m) \mathrm{d}r.
    \end{align*}
\end{lemma}

Next, let us provide some estimates for the different controls. The first result is standard, the second is just an application of \Cref{lemm:reguralization_FP} (see also \Cref{remark:estimates}) combined with \Cref{lemma:reformulation}.

\begin{lemma}[Estimates]\label{lemma:estimates}
        Under {\rm \Cref{assum:main1}}, for any $(\nu,\nu^1,\dots,\nu^N) \in \Pc_{p'}(\R^n)^{N+1}$ with $p'>p,$ there exists $K>0,$ depending only of coefficients $(b,\sigma)$ and $p',$ such that: for every $(\alpha^1,\dots,\alpha^N) \in (\Ac_N(\nu_N))^N$ one has
        $$
            \E^{\Pr^N}\bigg[ 
            \sup_{t \in [0,T]} \int_{\R^n}|x|^{p'} \mu_t(\mathrm{d}x)
            \bigg]  
            \le K 
            \bigg[ 1+\int_{\R^n}|x'|^{p'}\frac{1}{N}\sum_{i=1}^N\nu^i(\mathrm{d}x')
            \bigg],
        $$
        where $\Pr^N:= \P(\alpha^1,...,\alpha^N) \in \Pc(\Omb)$ $($see definition \eqref{def:proba-Nagents}$)$, and for each $\Pr \in \Pcb_V(\nu)$ or $\alpha \in \Ac(\nu)$ with $\Pr= \P_{\nu} \circ \big( \mu^\alpha, \delta_{\mub^\alpha_t}(\mathrm{d}m)\mathrm{d}t, B \big)^{-1}$
        $$
            \sup_{t \in [0,T]} \int_{\R^n}|x|^{p'} \vartheta_t(\om)(\mathrm{d}x)
            +
            \E^{\Pr}\bigg[
            \sup_{t \in [0,T]} \int_{\R^n}|x|^{p'} \mu_t(\mathrm{d}x)
            \bigg]
            \le 
            K 
            \bigg[ 1+\int_{\R^n}|x'|^{p'}\nu(\mathrm{d}x') \bigg],\;\Pr\mbox{-a.e.}\;\om \in \Omb.
        $$
        In addition
        \begin{align*}
            \Wc_p \big(\vartheta_s(\om), \vartheta_t(\om) \big)^p \le K|t-s|,~\mbox{for all}~(t,s) \in [0,T] \x [0,T],\;\Pr\mbox{-a.e.} \;\om \in \Omb,
        \end{align*}
        where $\vartheta$ is the process given in equation \eqref{eq:shift-proba}.
    \end{lemma}

\subsubsection{Technical lemmas}

In this part, from a measure--valued rule, we will build a sequence of processes that approximate the measure--valued rule and that are close enough to strong control rules. This part is the fundamental part for the proof of Theorem \ref{theo:equality_strong-relaxed}.

\medskip
let $\nu \in \Pc_{p'}(\R^n),$ $\Pr \in \Pcb_V(\nu),$ and  $(\widetilde{\Om},\widetilde{\F},\widetilde{\Fc},\widetilde{\P})$ be a filtered probability space supporting $W$ $\R^n$--valued $\widetilde{\F}$--Brownian motion and   let $\xi$ be a $\widetilde{\Fc}_0$--random variable s.t. $\Lc^{\widetilde{\P}}(\xi)=\nu.$ 
We define the filtered probability space $(\Omh,\widehat{\F},\widehat{\Fc},\widehat{\P})$ which is an extension of the canonical space $(\Omb,\Fb,\Pr)$: $\Omh:=\widetilde{\Om} \x \Omb,$ $\widehat{\F}:=(\widetilde{\Fc}_t \otimes \Fcb_t)_{t \in [0,T]}$ and $\widehat{\P}:=\widetilde{\P} \otimes \Pr.$ The variables $(\xi,W)$ of $\widetilde{\Om}$ and $(B,\mu,\Lambda)$ of $\Omb$ are naturally extended on the space $\widehat{\Om}$ while keeping the same notation $(\xi,W,B,\mu,\Lambda)$ for simplicity. Also, let us consider the filtration $(\widehat{\Gc}_t)_{t \in [0,T]}$ defined by
\begin{align*}
    \widehat{\Gc}_t
    :=
    \sigma \big\{ B_{t \wedge \cdot},\mu_{t \wedge \cdot},\Lambda_{t \wedge \cdot} \big\},\;\mbox{for all}\;t \in [0,T].
\end{align*}

\begin{proposition} \label{prop:approximation_weak}
    Under {\rm \Cref{assum:main1}}, for any $[0,1]$--valued uniform variable $Z$ $\widehat{\P}$--independent of $(\xi,W,B,\mu,\Lambda),$ there exists a sequence of $\widehat{\F}$--predictable processes $(\alpha^k)_{k \in \N^*}$ satisfying: for each $k \in \N^*,$
    \begin{align*}
        \alpha^k_t:=G^k(t,\xi,\mu_{t \wedge \cdot},\Lambda_{t \wedge \cdot},W_{t \wedge}, B_{t \wedge},Z),\;\widehat{\P}\mbox{--a.e.},\;\mbox{for all}\;t \in [0,T],
    \end{align*}
    with a Borel function $G^k:[0,T] \x \R^n \x \Cc^n_{\Wc} \x \M(\Pc^n_U) \x \Cc^n \x \Cc^\ell \x [0,1] \to U$ such that if we let $\widehat{X}^k$ be the unique strong solution of: $\E^{\widehat{\P}}[\|\widehat{X}^{k}\|^{p'}]< \infty,$ for all $t \in [0,T]$
    \begin{align*}
        \widehat{X}^{k}_t
        =
        \xi
        &+
        \int_0^t
        b(r,\widehat{X}^{k}_r,\mu^k,\mub^k_r,\alpha^{k}_{r}) \mathrm{d}r
        +
        \int_0^t
        \sigma(r,\widehat{X}^k_r,\mu^k,\mub^k_r,\alpha^{k}_{r}) \mathrm{d}W_r
        +
        \sigma_0 B_t,\;\widehat{\P}\mbox{--a.e.}
    \end{align*}
    where $\mu^k_t:=\Lc^{\P}(\widehat{X}^{k}_t|\widehat{\Gc}_t)$ and $\mub^k_t:=\Lc^{\widehat{\P}}(\widehat{X}^{k}_t,\alpha^{k}_{t}|\widehat{\Gc}_t)$  then
    \begin{align} \label{result:a.s.}
        \Lim_{k \to \infty} \bigg[ \Wc_p \Big(\delta_{\mub^k_s}(\mathrm{d}m) \mathrm{d}s, \Lambda_s(\mathrm{d}m) \mathrm{d}s \Big) +  \sup_{t \in [0,T]} \Wc_p(\mu^k_t,\mu_t) \bigg] = 0,\;\widehat{\P}\mbox{--a.e.}.
    \end{align}
Therefore
    \begin{align*} 
        \Lim_{k \to \infty} \Lc^{\widehat{\P}} \Big( (\mu^k_t)_{t \in [0,T]}, \delta_{\mub^k_s}(\mathrm{d}m) \mathrm{d}s, (B_t)_{t \in [0,T]} \Big)=\Pr,\;\mbox{for the Wasserstein metric}\;\Wc_p.
    \end{align*}

\end{proposition}


\begin{proof}
    As $\Pr \in \Pcb_V(\nu),$ by definition, $\Pr$ a.e. $\om \in \Omb,$ $N_t(f)=0$ for all $f \in C^{2}_b(\R^n)$ and $t \in [0,T].$ By \Cref{lemma:reformulation}, by taking into account the extension of all variables on $\widehat{\Om},$ recall that $(\vartheta_t)_{t \in [0,T]}$ and $(\Theta_t)_{t \in [0,T]}$ are defined in \eqref{eq:shift-proba}, one has $\Theta_t(\Z_{\vartheta_t})=1,$ $\mathrm{d}\widehat{\P} \otimes \mathrm{d}t$ a.e. $(t,\om) \in [0,T] \x \Omh,$ and $\widehat{\P}$--a.e. $\om \in \Omh,$ for all $(f,t) \in C^{2}_b(\R^n) \x [0,T],$
    \begin{align*}
        N_t(f)
        =
        \langle f,\vartheta_t \rangle
	    -
	    \langle f,\nu \rangle
	    -
	    \int_0^t \int_{\Pc^n_U} \int_{\R^n \x U} \widehat{\Lc}_r f(y,B,\vartheta,m,u) m(\mathrm{d}u,\mathrm{d}y) \Theta_r(\mathrm{d}m) \mathrm{d}r.
    \end{align*}
    
\medskip
Define
\begin{align*}
    \Gamma
    :=
    \Big\{
        m \in \Pc^n_U:\;\; \int_{\R^n} |y|^{p'} m(\mathrm{d}y,U) \le \hat K
    \Big\},
\end{align*}
where $\hat K>0$ is such that $\hat K>K \bigg[ 1+\int_{\R^n}|x'|^{p'}\nu(\mathrm{d}x') \bigg],$ with $K$ is a constant used in Lemma \ref{lemma:estimates}. Notice that $\Gamma$ is a compact set of $\Pc_p(\R^n \x U)$ and by \Cref{lemma:estimates}, one has that $\Theta_t(\Gamma)=1,$ $\mathrm{d}\widehat{\P} \otimes \mathrm{d}t,$ a.e. $(t,\om) \in [0,T] \x \widehat{\Om}.$
As $\Gamma$ is a compact set of $\Pc_p(\R^n \x U),$ there exists a family of measurable functions $(h^k)_{k \in \N^*}$ with $h^k: [0,T] \x \M \to \Pc^n_U,$ s.t.   
\begin{align*} 
    \Lim_{k \to \infty}
    \delta_{h^k(t,\Theta_{t \wedge \cdot})}(\mathrm{d}m)\mathrm{d}t
    =
   \Theta_t(\mathrm{d}m)\mathrm{d}t,\;\widehat{\P}\;\mbox{--a.e.}\;\;\mbox{then}\;\;\Lim_{k \to \infty} \Lc^{\widehat{\P}} \big(\vartheta,\delta_{h^k(t,\Theta_{t \wedge \cdot})}(\mathrm{d}m)\mathrm{d}t, B \big)= \Lc^{\widehat{\P}}(\vartheta,\Theta,B),\;\;\mbox{in}\;\Wc_p.
\end{align*}

In the same spirit of notations \eqref{eq:shift-generator}, we introduce 
    \begin{align} \label{eq:def-coeff}
        [\hat b,\hat \sigma](t,y,\bb,\pi,m,u):=[b,\sigma](t,y+\sigma_0\bb_t,\pi[\bb],m[\bb_t],u),
    \end{align}
    notice that $[\hat b,\hat \sigma]: [0,T] \x \R^n \x \Cc^\ell \x \Cc^n_{\Wc} \x \Pc^n_U \x U \to \R^n \x \S^{n \x n}$ is continuous and for $\bb \in \Cc^\ell,$ $[\hat b,\hat \sigma](\cdot,\cdot,\bb,\cdot,\cdot,\cdot)$ verify the Assumption \ref{assum:main1} with constant $C$ and $\theta$ independent of $\bb$ (see Assumption \ref{assum:main1}).

Now, let us apply \Cref{prop:weak-appr} (see also \Cref{prop:approximation-FP_BY_SDE-2}). 
As $\big(\vartheta,\delta_{h^k(t,\Theta_{t \wedge \cdot})}(\mathrm{d}m)\mathrm{d}t, B \big)_{k \in \N^*}$ is $\widehat{\P}$ independent of $(\xi,W)$  and
\begin{align*}
    \Lim_{k \to \infty} \Lc^{\widehat{\P}} \Big(\vartheta,\delta_{h^k(s,\Theta_{s \wedge \cdot}) }(\mathrm{d}m)\mathrm{d}s, B \Big)= \Lc^{\widehat{\P}}\big(\vartheta,\Theta_s(\mathrm{d}m)\mathrm{d}s,B \big),\;\mbox{in}\;\Wc_p,
\end{align*}
by \Cref{prop:weak-appr}, there exists $G^k: [0,T] \x \R^n \x \M \x \Cc^n_{\Wc} \x \Cc^n \x \Cc^\ell \x [0,1] \to U$ a Borel function such that if $X^k$ is the unique strong solution of: for all $t \in [0,T]$
\begin{align} \label{eq:general-weak_McK}
    X^k_t
    =
    \xi
    &+
    \int_0^t \hat b\big(r,X^k_r,B,\vartheta^k, \overline{\vartheta}^k_r,\alpha^{k}_{r}\big) \mathrm{d}r
    +
    \int_0^t \hat \sigma \big(r,X^k_r,B,\vartheta^k, \overline{\vartheta}^k_r,\alpha^{k}_{r}\big)
    \mathrm{d}W_r,\;\widehat{\P}\mbox{--a.e.},
\end{align}
where
\begin{align*}
    \alpha^{k}_t:= G^k\big(t,\xi,\Theta^k_{t \wedge \cdot},\vartheta_{t \wedge \cdot},W_{t \wedge \cdot}, B_{t \wedge \cdot},Z \big),\;\;\overline{\vartheta}^k_t:=\Lc^{\widehat{\P}}\big(X^k_t,\alpha^{k}_t\big|\Gc^k_t\big)\;\;\mbox{and}\;\;\vartheta^k_t:=\Lc^{\widehat{\P}}\big(X^k_t\big|\Gc^k_t\big),
\end{align*}
with $\Theta^k_t(\mathrm{d}m)\mathrm{d}t:=\delta_{\big(h^k(t,\Theta_{t \wedge \cdot}) \big)}(\mathrm{d}m)\mathrm{d}t,$ and $\G^k:=(\Gc^k_s)_{s \in [0,T]}:=(\sigma \{ \vartheta_{s \wedge \cdot},\Theta^k_{s \wedge \cdot}, B_{s \wedge \cdot} \})_{s \in [0,T]},$ then
    \begin{align*}
        \Lim_{j \to \infty} \E^{\widehat{\P}} \bigg[\int_0^T \Wc_p \big(\overline{\vartheta}^{k_j}_t,\mb^{k_j}_t \big)^p \mathrm{d}t + \sup_{t \in [0,T]} \Wc_p(\vartheta^{k_j}_t,\vartheta_t) \bigg]=0\;\;\mbox{and}\;\;
        \Lim_{j \to \infty} \Lc^{\widehat{\P}} \big(\vartheta^{k_j},\Theta^{k_j},B \big)
        =
        \Lc^{\widehat{\P}} \big(\vartheta,\Theta,B \big),\;\;\mbox{in}\;\Wc_p,
    \end{align*}
where $\mb_t^k:=h^k(t,\Theta_{t \wedge \cdot})$ and $(k_j)_{j \in \N^*} \subset N^*$ is a sub--sequence.
Notice that, as $\G^k \subset \widehat{\G},$ and $(\xi,W,Z)$ are $\widehat{\P}$ independent of $\widehat{\G},$ one has $\Lc^{\widehat{\P}}\big(X^k_t,\alpha^{k}_t\big|\Gc^k_t\big)=\Lc^{\widehat{\P}}\big(X^k_t,\alpha^{k}_t\big|\widehat{\Gc}_t\big),$ $\widehat{\P}$--a.e. for all $t \in [0,T].$ Using equation \eqref{eq:def-coeff}, we rewrite $X^k$ by
\begin{align*} 
    X^k_t
    =
    \xi
    &+
    \int_0^t b\big(r,X^k_r+\sigma_0B_r,(\Lc^{\widehat{\P}}(X^k_s+\sigma_0B_s|\widehat{\Gc}_s))_{s \in [0,T]}, \Lc^{\widehat{\P}}(X^k_r+\sigma_0B_r,\alpha^{k}_r|\widehat{\Gc}_r),\alpha^{k}_{r}\big) \mathrm{d}r
    \\
    &~~~~+
    \int_0^t \sigma \big(r,X^k_r+\sigma_0B_r,(\Lc^{\widehat{\P}}(X^k_s+\sigma_0B_s|\widehat{\Gc}_s))_{s \in [0,T]}, \Lc^{\widehat{\P}}(X^k_r+\sigma_0B_r,\alpha^{k}_r|\widehat{\Gc}_r),\alpha^{k}_{r}\big)
    \mathrm{d}W_r,
    ~
    \mbox{for all}
    ~
    t \in [0,T],\;\widehat{\P}\mbox{--a.e.}
\end{align*}
Denote by $\widehat{X}^k:=X^k+\sigma_0B,$ one finds
\begin{align*} 
    \widehat{X}^k_t
    =
    \xi
    &+
    \int_0^t b\big(r,\widehat{X}^k_r,(\Lc^{\widehat{\P}}(\widehat{X}^k_s|\widehat{\Gc}_s))_{s \in [0,T]}, \Lc^{\widehat{\P}}(\widehat{X}_r,\alpha^{k}_r|\widehat{\Gc}_r),\alpha^{k}_{r}\big) \mathrm{d}r
    \\
    &~~~~+
    \int_0^t \sigma \big(r,\widehat{X}^k_r,(\Lc^{\widehat{\P}}(\widehat{X}^k_s|\widehat{\Gc}_s))_{s \in [0,T]}, \Lc^{\widehat{\P}}(\widehat{X}^k_r,\alpha^{k}_r|\widehat{\Gc}_r),\alpha^{k}_{r}\big)
    \mathrm{d}W_r+\sigma_0B_t,
    ~
    \mbox{for all}
    ~
    t \in [0,T],\;\widehat{\P}\mbox{--a.e.}.
\end{align*}

With the notation introduced in \eqref{eq:shift-proba-initial} and \eqref{eq:shift-proba-M}, it is straightforward to check that the map
\begin{align*}
    (\pi,q,\bb) \in \Cc^n_{\Wc} \x \M \x \Cc^\ell \to \big(\pi[\bb],q_t[\bb](\mathrm{d}m)\mathrm{d}t,\bb \big) \in \Cc^n_{\Wc} \x \M \x \Cc^\ell
\end{align*}
is continuous. Consequently, one has
\begin{align*}
    \Lim_{j \to \infty} \E^{\widehat{\P}} \bigg[\int_0^T \Wc_p \Big( \overline{\vartheta}^{k_j}_t[B_t],\mb^{k_j}_t[B_t] \Big)^p \mathrm{d}t + \sup_{t \in [0,T]} \Wc_p(\vartheta^{k_j}_t[B],\vartheta_t[B]) \bigg]=0,
\end{align*}
therefore, in $\Wc_p,$
    \begin{align*}
        \Lim_{j \to \infty} \Lc^{\widehat{\P}} \Big( (\Lc^{\widehat{\P}}(\widehat{X}^{k_j}_t|\widehat{\Gc}_t))_{t \in [0,T]}, \delta_{(\Lc^{\widehat{\P}}(\widehat{X}^{k_l}_s,\alpha^{k_l}_s|\widehat{\Gc}_s))} (\mathrm{d}m)\mathrm{d}s, B  \Big)
        &=
        \Lim_{j \to \infty} \Lc^{\widehat{\P}} \big(\vartheta^{k_j}[B],\Theta^{k_j}_t[B](\mathrm{d}m)\mathrm{d}t,B \big)
        \\
        &=
        \Lc^{\widehat{\P}} \big(\vartheta[B],\Theta_t[B](\mathrm{d}m)\mathrm{d}t,B \big)
    \end{align*}
After simple calculations, $(\vartheta[B],\Theta_t[B](\mathrm{d}m)\mathrm{d}t,B)=(\mu,\Lambda,B),$ $\widehat{\P}$--a.e.  Then 
\begin{align*}
    \Lim_{j \to \infty} \E^{\widehat{\P}} \bigg[\int_0^T \Wc_p \Big(\Lc^{\widehat{\P}}(\widehat{X}^{k_j}_t,\alpha^{k_j}_t|\widehat{\Gc}_t),\mb^{k_j}_t[B_t] \Big)^p \mathrm{d}t + \sup_{t \in [0,T]} \Wc_p\big(\Lc^{\widehat{\P}}(\widehat{X}^{k_j}_t|\widehat{\Gc}_t),\mu_t \big) \bigg]=0,
\end{align*}
and hence
\begin{align*}
        \Lim_j \Lc^{\widehat{\P}} \Big( (\Lc^{\widehat{\P}}(\widehat{X}^{k_j}_t|\widehat{\Gc}_t))_{t \in [0,T]}, \delta_{(\Lc^{\widehat{\P}}(\widehat{X}^{k_j}_s,\alpha^{k_l}_s|\widehat{\Gc}_s))} (\mathrm{d}m)\mathrm{d}s, B  \Big)
        &=
        \Lc^{\widehat{\P}} (\mu,\Lambda,B)=\Pr,\;\mbox{in}\;\Wc_p.
    \end{align*}
After extraction from $(\widehat{X}^{k_j},\alpha^{k_j})_{j \in \N^*},$ one has also the $\widehat{\P}$--a.e. convergence \eqref{result:a.s.}.
    
\end{proof}

\subsubsection{Proof of Theorem \ref{theo:equality_strong-relaxed}} 

    First, for $\nu \in \Pc_{p'}(\R^n),$ under \Cref{assum:main1}, let us prove that $\Pcb_V(\nu)$ is a compact set for the Wasserstein topology $\Wc_p.$ Let $(\Pr_k)_{k \in \N^*} \subset \Pcb_V(\nu),$ by \Cref{Propostion:convergence}, $(\Pr_k)_{k \in \N^*}$ is relatively compact for the Wassertein topology $\Wc_p$ and any limit $\Pr_{\infty}$ of any sub--sequence belongs to $\Pcb_V(\nu).$ Therefore $\Pcb_V(\nu)$ is compact. By similar techniques used in  \cite[Theorem 3.1]{djete2019general}, it is straightforward to show that $\Pcb_V(\nu)$ is convex.
    
\medskip    
    Next, we prove the items $(i)$ and $(ii)$ of \Cref{theo:equality_strong-relaxed}.
    By applying \Cref{prop:approximation_weak}, with the same notations, for any $[0,1]$--valued uniform variable $Z$ $\widehat{\P}$--independent of $(\xi,W,B,\mu,\Lambda),$ there exists a sequence of $\widehat{\F}$--predictable processes $(\alpha^k)_{k \in \N^*}$ satisfying: for each $k \in \N^*,$
    \begin{align*}
        \alpha^k_t:=G^k(t,\xi,\mu_{t \wedge \cdot},\Lambda_{t \wedge \cdot},W_{t \wedge}, B_{t \wedge},Z),\;\widehat{\P}\mbox{--a.e.},\;\mbox{for all}\;t \in [0,T],
    \end{align*}
    with $G^k:[0,T] \x \R^n \x \Cc^n_{\Wc} \x \M(\Pc^n_U) \x \Cc^n \x \Cc^\ell \x [0,1] \to U$ is a Borel function such that if $\widehat{X}^k$ is the unique strong solution of: for all $t \in [0,T]$
    \begin{align*}
        \widehat{X}^{k}_t
        =
        \xi
        &+
        \int_0^t
        b(r,\widehat{X}^{k}_r,\mu^k,\mub^k_r,\alpha^{k}_{r}) \mathrm{d}r
        +
        \int_0^t
        \sigma(r,\widehat{X}^k_r,\mu^k,\mub^k_r,\alpha^{k}_{r}) \mathrm{d}W_r
        +
        \sigma_0 B_t,\;\widehat{\P}\mbox{--a.e.}
    \end{align*}
    where $\mu^k_t:=\Lc^{\P}(\widehat{X}^{k}_t|\widehat{\Gc}_t)$ and $\mub^k_t:=\Lc^{\widehat{\P}}(\widehat{X}^{k}_t,\alpha^{k}_{t}|\widehat{\Gc}_t)$ then 
    \begin{align*}
        \Lim_{k \to \infty} \Lc^{\widehat{\P}} \Big( (\mu^k_t)_{t \in [0,T]}, \delta_{\mub^k_s}(\mathrm{d}m) \mathrm{d}s, (B_t)_{t \in [0,T]} \Big)=\Pr,\;\mbox{for the Wasserstein metric}\;\Wc_p.
    \end{align*}
    
    For each $k \in \N^*,$ $\widehat{X}^k_t=H^k_t(\xi,W_{t \wedge \cdot},\mu_{t \wedge \cdot},\Lambda_{t \wedge \cdot},B_{t \wedge \cdot},Z),$ for all $t \in [0,T],$ $\widehat{\P}$--a.e. with $H^k: \R^n \x \Cc^n \x \Cc^n_{\Wc} \x \M \x \Cc^\ell \x [0,1] \to \Cc^n$ a Borel function. Then, as $(\xi,W,Z)$ are $\widehat{\P}$--independent of $(\mu,\Lambda,B),$ one gets that for all $t \in [0,T],$ $\Lc^{\widehat{\P}} (\widehat{X}^k_{t \wedge \cdot},\alpha^k_t|\widehat{\Gc}_t)=\Lc^{\widehat{\P}} (\widehat{X}^k_{t \wedge \cdot},\alpha^k_t|\widehat{\Gc}_T),$ $\widehat{\P}$--a.e.. Let us introduce the process $(\muh^k_t)_{t \in [0,T]},$
    \begin{align*}
        \muh^k_t
        :=
        \Lc^{\widehat{\P}} (\widehat{X}^k_{t \wedge \cdot},\widehat{X}^k_{t \wedge \cdot}-\sigma_0B_{t \wedge \cdot},W,\Lambda^k_{t \wedge \cdot}|\widehat{\Gc}_t),\;\mbox{for all}\;t \in [0,T]\;\mbox{with}\;\Lambda^k_t(\mathrm{d}u)\mathrm{d}t:=\delta_{\alpha^k_t}(\mathrm{d}u)\mathrm{d}t.
    \end{align*}
    For each $k \in \N^*,$ $\muh^k_t \in \Pc(\Cc^n \x \Cc^n \x \Cc^n \x \M(U)),$ for all $t \in [0,T]$ and if $(\Xt,\widetilde{Y},\Wt,\widetilde{\Lambda})$ is the canonical process on $\Cc^n \x \Cc^n \x \Cc^n \x \M(U),$ one has $\mu^k_t=\Lc^{\hat \mu^k_t}(\Xt_t),$ $\widehat{\P}$--a.e., and $\Lc^{\widehat{\P}}(\widehat{X}^{k}_t,\alpha^{k}_{t}|\widehat{\Gc}_t)(\mathrm{d}x,\mathrm{d}u)=\E^{\hat \mu^k_t}[\delta_{\Xt_t}(\mathrm{d}x)\widetilde{\Lambda}_t(\mathrm{d}u)],$ $\widehat{\P}$--a.e. for all $t \in [0,T].$ It is straightforward to see that $\muh^k_t=\Lc^{\widehat{\P}} (\widehat{X}^k_{t \wedge \cdot},\widehat{X}^k_{t \wedge \cdot}-\sigma_0B_{t \wedge \cdot},W,\Lambda^k_{t \wedge \cdot}|\widehat{\Gc}_T),$ for each $k \in \N^*,$ then
    \begin{align*}
        \muh^k_t
        =
        \Lc^{\widehat{\P}} (\widehat{X}^k_{t \wedge \cdot},\widehat{X}^k_{t \wedge \cdot}-\sigma_0B_{t \wedge \cdot},W,\Lambda^k_{t \wedge \cdot}|B_{t \wedge \cdot},\muh^k_{t \wedge \cdot})
        =\Lc^{\widehat{\P}} (\widehat{X}^k_{t \wedge \cdot},\widehat{X}^k_{t \wedge \cdot}-\sigma_0B_{t \wedge \cdot},W,\Lambda^k_{t \wedge \cdot}|B,\muh^k),\;\widehat{\P}\mbox{--a.e.},\;\mbox{for all}\;t \in [0,T],
    \end{align*}
    and $(B,\muh^k)$ are $\widehat{\P}$--independent of $(\xi,W).$ For all $k \in \N^*,$ denote
    \begin{align*}
        \overline{\Q}^k
        :=
        \widehat{\P} \circ \big(\widehat{X}^k, \widehat{X}^k-\sigma_0B, \Lambda^k,W,B,\muh^k \big)^{-1} \in \Pc\Big(\Cc^n \x \Cc^n \x \M(U) \x \Cc^n \x \Cc^\ell \x \Pc(\Cc^n \x \Cc^n \x \Cc^n \x \M(U)) \Big),
    \end{align*}
    then $\overline{\Q}^k$ is a weak control according to \cite[Definition 2.9]{djete2019general}.
    Then by (a slight extension of) \cite[Proposition 4.5]{djete2019general},

    $(1)$ when $\ell \neq 0,$ there exists $\alpha^{j,k} \in \Ac(\nu),$ and $X^{\alpha^{j,k}}$ the strong solution of \eqref{eq:MKV_strong} with control $\alpha^{j,k}$ such that 
\begin{align*}
    \lim_{j \to \infty}
    \P_{\nu} \circ \Big( X^{\alpha^{j,k}},W,B, \delta_{(\mub^{\alpha^{j,k}}_s,\;\alpha^{j,k}_s )}(\mathrm{d}m,\mathrm{d}u)\mathrm{d}s \Big)^{-1}
    =
    \widehat{\P} \circ \Big( \widehat{X}^k, W, B, \delta_{(\mub^k_s,\;\alpha^k_s)}(\mathrm{d}m,\mathrm{d}u)\mathrm{d}s \Big)^{-1},\;\mbox{in}~\Wc_p.
\end{align*}

    $(2)$ When  $\ell=0,$ there exists a family of Borel functions $(\kappa^k_j)_{k,j}$ with $\kappa^k_j: [0,T] \x \R^n \x \Cc^n \x [0,1] \to U,$ such that if $\alpha^{j,k}_t[z]:=\kappa^k_j(t,\xi,W_{t \wedge \cdot},z),$ for $z \in [0,1],$ one gets $(\alpha^{j,k}_t[z])_{t \in [0,T]} \in \Ac(\nu)$ and 
\begin{align*}
    \lim_{j \to \infty}
    \int_0^1 \P_{\nu} \circ \Big( X^{\alpha^{j,k}[z]},W,B, \delta_{(\mub^{\alpha^{j,k}[z]}_s,\;\alpha^{j,k}_s[z] )}(\mathrm{d}m,\mathrm{d}u)\mathrm{d}s \Big)^{-1} \mathrm{d}z
    =
    \widehat{\P} \circ \Big( \widehat{X}^k, W, B, \delta_{(\mub^k_s,\;\alpha^k_s)}(\mathrm{d}m,\mathrm{d}u)\mathrm{d}s \Big)^{-1},\;\mbox{in}~\Wc_p.
\end{align*}

All these results are enough to deduce the items $(i)$ and $(ii)$ of Theorem \ref{theo:equality_strong-relaxed}, and conclude that: for $\nu \in \Pc_{p'}(\R^n),$ $V_S(\nu)=V_V(\nu)$ and there exists $\Pr^\star \in \Pcb_V(\nu)$ such that $V_V(\nu)=\E^{\Pr^*}\big[J\big(\mu,\Lambda\big)\big].$

\subsection{Propagation of chaos}

With the help of Theorem \ref{theo:equality_strong-relaxed}, in this section we provide one of the main objective of this paper, which is to prove the limit theory result or (controlled) propagation of chaos.

\subsubsection{Technical results: study of the behavior of processes when $N$ goes to infinity}
    
In this part, the properties of some sequences of probability measures on the canonical space $\Omb$ are given. Mainly, the behavior when $N$ goes to infinity of sequences of type $(\P(\alpha^1,...,\alpha^N))_{N \in \N^*}$ construct from the formulation of large population stochastic control problem are studied. (see \Cref{subsec:N-agents} and \Cref{remark:def-proba}). 
    
    {

    \begin{proposition}
    \label{Propostion:convergence}
    Let {\rm \Cref{assum:main1}} hold true and $(\nu^i)_{i \in \N^*}\subset \Pc_{p'}(\R^n).$ Recall that $\nu_N:=\nu^1 \otimes ... \otimes \nu^N,$ for each $N \in \N^*.$
    
       $(i)$  Let $(\Pr^N)_{N \in \N^*}$ be the sequence satisfying $\Pr^N:=\P(\alpha^{1,N},...,\alpha^{N,N})$ $($see definition \eqref{def:proba-Nagents}$)$ with $\alpha^{i,N} \in \Ac_N(\nu_N)$  $\forall i \in [\![ 1,N]\!],$ for each $N \in \N^*.$  
    If  
    $$
        \sup_{N \ge 1} \frac{1}{N}\sum_{i=1}^N \int_{\R^n}| x'|^{p'} \nu^i(\mathrm{d}x') < \infty
    $$
    then $(\Pr^N)_{N \in \N^*}$ is precompact in $\Pc_p(\Omb)$ for the metric $\Wc_{p}$ and for every $\Pr^\infty \in \Pc(\Omb)$ the limit of any sub--sequence $(\Pr^{N_j})_{j \in \N}$,  $\Pr^\infty \in \Pcb_V\big(\lim_{j \to \infty} \frac{1}{N_j}\sum_{i=1}^{N_j}\nu^i \big)$.
    
\medskip
       $(ii)$ Let us consider the sequence $(\Pr_k)_{k \in \N^*}$ of probability measures such that $\Pr_k \in \Pcb_V(\nu^k)$ for each $k \in \N^*.$
        If  
    $$
        \sup_{k \ge 1} \int_{\R^n}| x'|^{p'} \nu^k(\mathrm{d}x') < \infty
    $$
    then $(\Pr_k)_{k \in \N^*}$ is precompact in $\Pc_p(\Omb)$ for the metric $\Wc_{p}$ and for every $\Pr_{\infty} \in \Pc(\Omb)$ the limit of any sub--sequence $(\Pr_{k_j})_{j \in \N^*}$, $\Pr_{\infty} \in \Pcb_V\big(\lim_{j \to \infty} \nu^{k_j} \big).$

    \end{proposition}

    \begin{proof}
    $(i)$ Thanks to  Proposition A.2 or/and Proposition-B.1 of \cite{carmona2014mean}, as $U$ is compact, it is easy to check that $(\Pr^N)_{N \in \N^*}$ is pre--compact on $\Pc_p(\Omb)$ for the metric $\Wc_{p}$.
    Let $\Pr^\infty$ be a limit of a sub--sequence $(\Pr^{N_j})_{j \in \N^*}$. For sake of simplicity, we denote $(\Pr^{N_j})_{j \in \N^*}=(\Pr^{N})_{N \in \N^*}$ and $\nu:=\lim_j \frac{1}{N_j}\sum_{i=1}^{N_j}\nu^i.$

\medskip    
    Now, let us show $\Pr^\infty \in \Pcb_V(\nu).$ Let $f \in C_b^{2}(\R^{n}).$ For each $t \in [0,T],$ denote $N_t(B_{t \wedge \cdot},\Lambda_{t \wedge \cdot},\mu_{t \wedge \cdot})(f)=N_t(f)$ to specify the dependence w.r.t. $(B,\mu,\Lambda)$ (see definition \eqref{equation:FP-characteristic}). Notice that the function $(t,\bb,\pi,q) \in [0,T] \x \Cc^\ell \x \Cc^n_{\Wc} \x \M \to N_t(\bb_{t \wedge \cdot},q_{t \wedge \cdot},\pi_{t \wedge \cdot})(f) \in \R$ is continuous and bounded. It is straightforward to check that: for all $t \in [0,T]$
    \begin{align*}
        N_t\big(B_{t \wedge \cdot},(\delta_{\varphi^N_s}(\mathrm{d}m)\mathrm{d}s)_{t \wedge \cdot},\varphi^{N,\Xbb}_{t \wedge \cdot} \big)(f)
        =
        \frac{1}{N} \sum_{i=1}^N \int_0^t \nabla f(\Xbb^{\alpha,i}_r-\sigma_0 B_r) \sigma(r,\Xbb^{\alpha,i}_r,\varphi^{N,\Xbb}_{r \wedge \cdot},\varphi^N_r,\alpha^i_r) \mathrm{d}\Wbb^i_r,\;\P^N_{\nu}\mbox{--a.e.}.
    \end{align*}
    With the same techniques used in the proof of  \cite[Proposition 5.1]{lacker2017limit} or \cite[Proposition 4.17]{djete2019general},
    one has 
    \begin{align*}
        \E^{\Pr^\infty}
        \Big[ \big|(N_t(f) \big|^2
        \Big] 
        &=
        \E^{\Pr^\infty}
        \Big[ \big|(N_t(B_{t \wedge \cdot},\Lambda_{t \wedge \cdot},\mu_{t \wedge \cdot})(f) \big|^2
        \Big]
        =
        \Lim_N \E^{\Pr^N}
        \Big[ \big|(N_t(B_{t \wedge \cdot},\Lambda_{t \wedge \cdot},\mu_{t \wedge \cdot})(f) \big|^2
        \Big]
        \\
        &=
        \Lim_N \E^{\P^N_{\nu}}
        \Big[ \big|(N_t\big(B_{t \wedge \cdot},(\delta_{\varphi^N_s}(\mathrm{d}m)\mathrm{d}s)_{t \wedge \cdot},\varphi^{N,\Xbb}_{t \wedge \cdot} \big)(f) \big|^2
        \Big]
        \\
        &= \lim_N \E^{\P^N_{\nu}} \bigg[ \Big| \frac{1}{N} \sum_{i=1}^N \int_0^t \nabla f(\Xbb^{\alpha,i}_r-\sigma_0 B_r) \sigma(r,\Xbb^{\alpha,i}_r,\varphi^{N,\Xbb},\varphi^N_r,\alpha^i_r) \mathrm{d}\Wbb^i_r \Big|^2\bigg] 
        \\
        &= \lim_N \frac{1}{N^2} \sum_{i=1}^N \E^{\P^N_{\nu}} \bigg[  \int_0^t \Big|\nabla f(\Xbb^{\alpha,i}_r-\sigma_0 B_r) \sigma(r,\Xbb^{\alpha,i}_r,\varphi^{N,\Xbb},\varphi^N_r,\alpha^i_r) \Big|^2 \mathrm{d}r \bigg]=0.
    \end{align*}
    By taking $(t,f)$ under a countable set of $[0,T] \x C_b^{2}(\R^{n})$ then $\Pr^\infty$ a.e. $\om \in \Omb,$ $N_t(f)=0$ for all $(t,f) \in [0,T] \x C^{2}_b(\R^n).$
    
    For all $h \in C_b(\R^n),$ the map $(q,\pi) \in \M \x \Cc^n_{\Wc} \to \int_0^T \int_{\Pc^n_U} \big| \langle h,m(\mathrm{d}z,U) \rangle - \langle h,\pi_t(\mathrm{d}z) \rangle  \big|^2 q_t(\mathrm{d}m)\mathrm{d}t \in \R$ is bounded and continuous (see for instance \Cref{lemma:continuity}), one finds that
    \begin{align*}
        &\E^{\Pr^\infty} \bigg[ 
        \int_0^T \int_{\Pc^n_U} \big| \langle h,m(\mathrm{d}z,U) \rangle - \langle h,\mu_t(\mathrm{d}z) \rangle  \big|^2 \Lambda_t(\mathrm{d}m)\mathrm{d}t
        \bigg]
        \\
        &=
        \lim_N \E^{\P^N_\nu} \bigg[ 
        \int_0^T \int_{\Pc^n_U} \big| \langle h,m(\mathrm{d}z,U) \rangle - \langle h,\varphi^{N,\Xbb}_t(\mathrm{d}z) \rangle  \big|^2 \delta_{\varphi^N_t}(\mathrm{d}m)\mathrm{d}t
        \bigg]
        =
        \lim_N
        \E^{\P^N_\nu} \bigg[
        \int_0^T \Big| \frac{1}{N} \sum_{i=1}^N [h(\Xbb^{\alpha,i}_t)- h(\Xbb^{\alpha,i}_t)] \Big|^2 \mathrm{d}t
        \bigg]
        =0,
    \end{align*}
    by taking $h$ under a countable set of $C_b(\R^n)$, one concludes $\Lambda_t \big( \Z_{\mu_t} \big)=1$ $\Pr^\infty \otimes \mathrm{d}t$ a.e. .
    It is obvious that $(B_t)_{t \in [0,T]}$ is a $(\Pr^\infty,\Fb)$ Wiener process. Let $Q \in \N^*,$ and $(h^q)_{q \in \{1,..,Q \}}: \R^n \to \R^Q$ be bounded functions, one has
    \begin{align*}
        \E^{\Pr^\infty}
        \bigg[ \prod_{q=1}^Q
        \langle h^q ,\mu_0 \rangle
        \bigg]
        =
        \prod_{q=1}^Q
        \langle h^q ,\nu \rangle.
    \end{align*}
Let us show this result when $Q=2,$ when $Q \in \N^*,$ the proof is similar.
    \begin{align*}
        \E^{\Pr^\infty}
        \big[ 
        \langle h^1 ,\mu_0 \rangle
        \langle h^2 ,\mu_0 \rangle
        \big]
        &=
        \lim_N \frac{1}{N^2}\sum_{i,j=1}^N \E^{\P^N_{\nu}}\big[  h^1(\Xbb^{\alpha,i}_{0}) h^2(\Xbb^{\alpha,j}_{0})\big] 
        =
        \lim_N \frac{1}{N^2}\sum_{i=1}^N\langle h^1,\nu^i \rangle \langle h^2,\nu^i \rangle
        +
        \frac{1}{N^2}\sum_{i \neq j}^N\langle h^1,\nu^i \rangle \langle h^2,\nu^j \rangle
        \\
        &=
        \lim_N
        \langle h^1, \frac{1}{N}\sum_{i=1}^N \nu^i \rangle \langle h^2, \frac{1}{N}\sum_{i=1}^N \nu^i \rangle
        =
        \langle h^1,\nu
        \rangle
        \langle h^2,\nu
        \rangle,
    \end{align*}
    by \cite[Proposition A.3]{djete2019general}, $\Pr^\infty \circ (\mu_0)^{-1}=\delta_{\nu},$ then $\mu_0=\nu,$ $\Pr^\infty$--a.e.. All these results allow to deduce the first statement of this proposition.
    
    \medskip
    $(ii)$ For the second part of this proposition, notice that, thanks to Lemma \ref{lemma:estimates},
    \begin{align*}
        \sup_{k \in \N^*} \E^{\Pr_k}\bigg[
            \sup_{t \in [0,T]} \int_{\R^n}|x|^{p'} \vartheta_t(\mathrm{d}x)
            \bigg]
            \le 
            K 
            \bigg[ 1+\sup_{k \in \N^*} \int_{\R^n}|x'|^{p'}\nu^k(\mathrm{d}x') \bigg]< \infty
    \end{align*}
    and
    \begin{align*}
        \Limsup_{\delta \to 0}\sup_{k \in \N^*} \sup_{\tau} \E^{\Pr^k}\big[\Wc_p \big(\vartheta_{(\tau + \delta) \wedge T}, \vartheta_{\tau} \big) \big]=0,
    \end{align*}
    where $\tau$ is a $[0,T]$--valued $\Fb$--stopping time, and recall that $(\vartheta)_{t \in [0,T]}$ is the $\Pc(\R^n)$--valued $\F$--adapted continuous process defined in equation \eqref{eq:shift-proba}.
    Then by Aldous' criterion \cite[Lemma 16.12]{kallenberg2002foundations} (see also proof of \cite[Proposition-B.1]{carmona2014mean} ), $\big(\Pr_k \circ \big( (\vartheta_t)_{t \in [0,T]} \big)^{-1}\big)_{k \in \N^*}$ is relatively compact for the metric $\Wc_p.$ Then, using the fact that $\Pr_k \in \Pcb_V(\nu^k)$  for each $k \in \N^*$ and the relation between $(\vartheta,\Theta)$ and the canonical processes $(\mu,\Lambda)$ (see equation \eqref{eq:shift-proba}), we deduce that $(\Pr_k)_{k \in \N^*}=\big(\Pr_k \circ \big(\mu, \Lambda, B \big)^{-1} \big)_{k \in \N^*}$ is relatively compact in $\Wc_p.$
    The rest of the proof is similar to the previous proof.

    \end{proof}
    }
    
\begin{proposition} \label{Proposition:Continuity}
		Let {\rm \Cref{assum:main1}} hold true, $\nu \in \Pc_{p'}(\R^n)$ with $p' > p$ and $(\nu^i)_{i \in \N} \subset \Pc_{p'}(\R^n)$ such that
		\[
			\sup_{i\in\N} \int_{\R^n}|x^\prime|^{p'} \nu^i(\mathrm{d}x^\prime) < \infty
			\; \mbox{and}\;
			\nu^i \overset{\Wc_p}{\underset{i\rightarrow\infty}{\longrightarrow}} \nu,\;\mbox{then}\;\lim_{i \to \infty} V_S(\nu^i)
			=
			V_S \big( \nu \big).
		\]
		In particular, the map $V_S:\Pc_{p'}(\R^n) \longrightarrow \R$ is continuous.
	\end{proposition}
	
	\begin{proof}
	    By Theorem \ref{theo:equality_strong-relaxed}, one has $V_S(\nu)=V_V(\nu),$ thanks to this result,  the proof is similar to the proof of \cite[Proposition 3.7.]{djete2019general}.
	    Let $(\delta^k)_{k \in \N^*} \subset \N^*$ with $\lim_{k \to \infty} \delta^k=0$ and $(\Pr^k)_{k \in \N^*}$ be a sequence such that $\Pr^k \in \Pcb_V(\nu^k)$ and $V_V(\nu^k)-\delta^k \le \E^{\Pr^k}[J(\mu,\Lambda)]$. By Proposition \ref{Propostion:convergence}, $(\Pr^k)_{k \in \N}$ is relatively compact on $(\Pc_{p}(\Omb),\Wc_{p})$ and if $\Pr \in \Pc(\Omb)$ is the limit of a sub--sequence $(\Pr^{k_j})_{j \in \N^*}$ then $\Pr \in \Pcb_V(\nu)$. Using \Cref{assum:main1}, by convergence of $(\Pr^{k_j})_{j \in \N^*},$ one has $\lim_j |\E^{\Pr^{k_j}}[J(\mu,\Lambda)]-\E^{\Pr}[J(\mu,\Lambda)]|=0.$
    Therefore, one gets 
    \[
        \limsup_k V_V(\nu^{k}) \le \lim_j \E^{\Pr^{k_j}}[J(\mu,\Lambda)]=\E^{\Pr}[J(\mu,\Lambda)] \le V_V(\nu)=V_S(\nu).
    \]
   By \cite[Proposition 4.15]{djete2019general}, $V_S (\nu) \le \liminf_j V_S(\nu^{k_j}),$ this is enough to conclude that $\Lim_k V_S(\nu^{k})=V_S(\nu),$ and deduce the result.
	\end{proof}

\subsubsection{Proof of Theorem \ref{theo:approximation-N_strong}}

By combining \Cref{theo:equality_strong-relaxed}, \Cref{Propostion:convergence} and  \Cref{Proposition:Continuity}, this proof turns to be the same used in the proof of \cite[Theorem 3.6]{djete2019general}. For the sake of completeness, we repeat the proof.

        $(i)$ By \Cref{Propostion:convergence} (with the same notations), if the sequence $(\Pr^N)_{N \in \N^*}$ is such that: $V_S^N(\nu^1,...,\nu^N)-\varepsilon_N \le  \E^{\Pr^N}[J(\mu,\Lambda)],$
        where $(\varepsilon^N)_{N \in \N^*}$ is sequence with $\lim_{N \to \infty} \varepsilon^N=0$, then $(\Pr^N)_{N \in \N^*}$ is relatively compact on $(\Pc_p(\Omb),\Wc_{p})$ and for every $\Pr^\infty \in \Pc(\Omb)$ the limit of the sub--sequence $(\Pr^{N_j})_{j \in \N^*}$, $\Pr^\infty \in \Pcb_V\big(\lim_{j \to \infty} \frac{1}{N_j} \sum_{i=1}^{N_j}\nu_i \big)$, therefore
        \[
            \limsup_{N \to \infty} V_S^N(\nu^1,...,\nu^N) \le \lim_{j \to \infty} \E^{\Pr^{N_j}}[J(\mu,\Lambda)] 
            =
            \E^{\Pr}[J(\mu,\Lambda)] 
            \le V_V \Big(\lim_{j \to \infty} \frac{1}{N_j} \sum_{i=1}^{N_j}\nu_i \Big).
        \]
        Then, as $\lim_{j \to \infty} \frac{1}{N_j} \sum_{i=1}^{N_j}\nu_i \in \Pc_{p'}(\R^n)$ and \Cref{assum:main1} holds true one can deduce that $V_V \Big(\lim_{j \to \infty} \frac{1}{N_j} \sum_{i=1}^{N_j}\nu_i \Big)=V_S \Big(\lim_{j \to \infty} \frac{1}{N_j} \sum_{i=1}^{N_j}\nu_i \Big).$
        By \cite[Proposition 4.15]{djete2019general}, $V_S \Big(\lim_{j \to \infty} \frac{1}{N_j} \sum_{i=1}^{N_j}\nu_i \Big) \le \liminf_{j \to \infty} V_S^{N_j}(\nu^1,...,\nu^{N_j}).$
        To recap
        \begin{align*}
            V_S \Big(\lim_{j \to \infty} \frac{1}{N_j} \sum_{i=1}^{N_j}\nu_i \Big) \le \liminf_{j \to \infty} V_S^{N_j}(\nu^1,...,\nu^{N_j})
            \le
            \limsup_{j \to \infty} V_S^{N_j}(\nu^1,...,\nu^{N_j}) \le V_S \Big(\lim_{j \to \infty} \frac{1}{N_j} \sum_{i=1}^{N_j}\nu_i \Big).
        \end{align*}
        
        $(ii)$ Let  $(N_j)_{j \in \N}$ be the sequence corresponding to :
        \[
            \limsup_{N \to \infty} \Big|V_S^N(\nu^1,...,\nu^N)-V_S\Big(\frac{1}{N}\sum_{i=1}^{N}\nu^i \Big) \Big|
            =
            \lim_{j \to \infty} \Big|V_S^{N_j}(\nu^1,...,\nu^{N_j})-V_S \Big(\frac{1}{N_j}\sum_{i=1}^{N_j}\nu^i \Big) \Big|.
        \]
        
        By the previous proof, $\lim_{j \to \infty} V_S^{N_j}(\nu^1,...,\nu^{N_j})=V_S \Big(\lim_{j \to \infty} \frac{1}{N_j}\sum_{i=1}^{N_j}\nu^i \Big),$
        as $(\frac{1}{N_j}\sum_{i=1}^{N_j}\nu^i)_{j \in \N^*}$ is bounded in $(\Pc_{p'}(\R^n),\Wc_{p'})$ and converges in $(\Pc_p(\R^n),\Wc_{p})$, by \Cref{Proposition:Continuity},
        \[
            \lim_{j \to \infty} V_S \Big(\frac{1}{N_j}\sum_{i=1}^{N_j}\nu^i \Big)=V_S \Big(\lim_{j \to \infty} \frac{1}{N_j}\sum_{i=1}^{N_j}\nu^i \Big),
        \]
        this is enough to conclude the proof.
        \qed

\subsection{Proof of \Cref{proposition:optimal-control-convergence}}

Notice that, for $\nu \in \Pc_{p'}(\R^n),$ by \Cref{theo:equality_strong-relaxed}, $\Pcb^\star_V(\nu)$ is nonempty. Let us define the distance function to the set $\Pcb^\star_V(\nu),$ for each $Q \in \Pc(\Omb),$ $\Psi^\star (Q):= \inf_{\Pr^\star \in \Pcb^\star_V(\nu)} \Wc_p \big(Q, \Pr^\star \big).$ It is well know that, as $\Pcb^\star_V(\nu)$ is nonempty, the function $\Psi^\star: Q \in \Pc_p(\Omb) \to \R$ is continuous. Then by \Cref{Propostion:convergence},  $(\Pr^N)_{N \in \N^*}$ is precompact in $\Pc_p(\Omb)$ for the metric $\Wc_{p}$ and if $\Pr \in \Pc(\Omb)$ is the limit of a sub--sequence $(\Pr^{N_j})_{j \in \N^*}$, one have $\Pr \in \Pcb_V(\nu)$. Under \Cref{assum:main1}, $\lim_{j \to \infty} \E^{\Pr^{N_j}}[J(\mu,\Lambda)]=\E^{\Pr}[J(\mu,\Lambda)].$ Combining \Cref{theo:approximation-N_strong} and \Cref{Proposition:Continuity}, one has that 
$$
    \Lim_{j \to \infty} V^{N_j}_S(\nu^1,...,\nu^{N_j})=V_S \Big(\lim_{j \to \infty} \frac{1}{N_j} \sum_{i=1}^{N_j}\nu_i \Big)=V_S(\nu)=V_V(\nu) \le \E^{\Pr}[J(\mu,\Lambda)],
$$
then $\Pr \in \Pcb^\star_V(\nu).$ Hence each limit of any sub--sequence of $(\Pr^N)_{N \in \N^*}$ belongs to $\Pcb^\star_V(\nu).$ Consequently, if $(\Pr^{N_j})_{j \in \N}$ is the sub--sequence corresponding to $\Limsup_{N \to \infty} \Psi^\star(\Pr^N)=\Lim_{j \to \infty} \Psi^\star(\Pr^{N_j}),$ by continuity of $\Psi^\star$ and the fact that any limit is an optimal control, $\Limsup_{N \to \infty} \Psi^\star(\Pr^N)=0.$ The second part of this proposition is just a combination of \Cref{theo:equality_strong-relaxed}, \cite[Proposition 4.15]{djete2019general} and \Cref{theo:approximation-N_strong}. This is enough to conclude the result.

\section{Approximation of Fokker--Planck equations} \label{sec:approx}

In this section, we give an approximation of a particular Fokker--Planck equation via a sequence of measure--valued processes constructed from classical SDE processes interacting through the empirical distribution of their states and controls. This result is a crucial part for the proof of Theorem \ref{theo:equality_strong-relaxed} and Theorem \ref{theo:approximation-N_strong}. 

{\color{black}

\subsection{Main ideas leading the proof} \label{sec:main_ideas}

Because of the technical aspect of this part, before going into details, let us first explain in a simple situation the main goal of this part and the ideas for the proof. As we said earlier, from a Fokker--Planck equation satisfied by a measure--valued solution $\Pr$ (see \Cref{def:RelaxedCcontrol}), we want to construct a sequence of $``$weak$"$ McKean--Vlasov processes s.t. the limit, in a certain sense, of this sequence will be $\Pr$. Let us be more precise. For simplification, we assume that $n=\ell=1,$ $U=[1,2],$ $b=0,$ $\sigma(t,x,\pi,m,u)=\sigma(m,u):=\sigma(m)u.$ 
Let $\Pr \in \Pcb_V,$ $(\mu,\Lambda,B)$ satisfy: $\Lambda_t(\Z_{\mu_t})$ $\mathrm{d}\widetilde{\Pr} \otimes \mathrm{d}t$ a.e. and for all $(t,f)$
	    \begin{align} \label{eq:idea_FP}
	        \mathrm{d}\langle f(\cdot-\sigma_0 B_t),\mu_t \rangle
	        =
	        \int_{\Pc^n_U}  \int_{\R^n \x U} f''(x - \sigma_0 B_t) \sigma(m)^2 u^2 m^x(\mathrm{d}u)\mu_t(\mathrm{d}x)  \Lambda_t(\mathrm{d}m) \mathrm{d}t.
	    \end{align}
    Using the SDEs formulation, on an extension $(\Omt,\Ft,\widetilde{\Pr})$ of $(\Omb,\Fb,\Pr),$ we can find $X$ satisfying
    \begin{align} \label{eq:idea_MV}
        \mathrm{d}X_t
        =
        \bigg(\int_{\Pc^n_U}  \int_{U} \sigma(m)^2 u^2 m^{X_t}(\mathrm{d}u) \Lambda_t(\mathrm{d}m) \bigg)^{1/2} \mathrm{d}W_t
        +
        \sigma_0 \mathrm{d}B_t,\;\;X_0=\xi\;\;\mbox{with}\;\mu_t=\Lc^{\widetilde{\Pr}}(X_t|\Gcb_t)=\Lc^{\widetilde{\Pr}}(X_t|\Gcb_T),
    \end{align}
    where $W$ is a $\Ft$--Brownian motion, $\xi$ a $\Fct_0$--random variable s.t. $\Lc(\xi)=\nu$ and $(W,\xi)$ is independent of $\Gcb_T.$ The process $(\Lambda_t)_{t \in [0,T]}$ can be seen as a control of the process $X$ or $\mu.$ The goal is to construct a sequence of $\F$--predictable processes $(\alpha^k)_{k \in \N^*}$ s.t. if $X^k$ is the solution of 
    \begin{align*}
        \mathrm{d}X^k_t
        =
        \sigma(\overline{m}^k_t)\alpha^k_t \mathrm{d}W_t
        +
        \sigma_0 \mathrm{d}B_t,\;\;X^k_0=\xi,\;m^k_t:=\Lc(X^k_t| \Gcb_t)\;\mbox{and}\;\overline{m}^k_t:=\Lc(X^k_t,\alpha^k_t| \Gcb_t),
    \end{align*}
    one has that
    \begin{align*}
        \Lim_{k \to \infty} \widetilde{\Pr}\circ \big( m^k,\delta_{\overline{m}^k_t}(\mathrm{d}m)\mathrm{d}t, B\big)^{-1}
        =
        \widetilde{\Pr}\circ \big(\mu,\Lambda,B \big)^{-1}\;\mbox{in}\;\Wc_p.
    \end{align*}
    
    If it was possible for \Cref{eq:idea_FP} or \Cref{eq:idea_MV} to satisfied an appropriate uniqueness result (in law), this kind of approximation would become much simpler to perform. Unfortunately, for a general $\Lambda,$ a uniqueness result can not be expected for this type of equation. Therefore, find the sequence $(\alpha^k)_{k \in \N^*}$ becomes a challenging problem.
    
    \paragraph*{Strategy of proof: 1--regularization} This part is realized in \Cref{sec:F--P_regularization}. The main idea here is to regularize \Cref{eq:idea_FP} or \Cref{eq:idea_MV} in order to recover some uniqueness result. Indeed, in \Cref{sec:F--P_regularization}, we show that: $X^{\varepsilon}$ solution of
    \begin{align} \label{eq:approx_MV}
        \mathrm{d}X^\varepsilon_t
        =
        \sigma^{\varepsilon}(t,\Lambda_t,X^\varepsilon_t) \mathrm{d}W_t
        +
        \sigma_0 \mathrm{d}B_t,\;\;X^\varepsilon_0=\xi,\;\mu^\varepsilon_t:=\Lc^{\widetilde{\Pr}}(X^\varepsilon_t|\Gcb_t)
    \end{align}
    satisfies
    \begin{align*}
        \Lim_{\varepsilon \to 0} \sup_{t \in [0,T]}\Wc_p(\mu_t,\mu^\varepsilon_t)=0,\;\widetilde{\Pr}\;\mbox{a.e.}
    \end{align*}
    where for each $\varepsilon >0,$ we define  $G_\varepsilon(x):= \varepsilon^{-1}G(\varepsilon^{-1}x),$ where  $G \in C^{\infty}(\R^n;\R)$ with compact support satisfying $G \ge 0,$ $G(x)=G(-x)$ for $x \in \R^n,$ and $\int_{\R^n}G(y)\mathrm{d}y=1,$ and (recall that $\Lambda_t(\Z_{\mu_t})=1$)
    \begin{align*}
        \sigma^\varepsilon(t,\Lambda_t,x)^2
        &:=
        \int_{\Pc^n_U}  \int_{U} \sigma(m)^2 u^2 m^{y}(\mathrm{d}u) \frac{G_{\varepsilon}(x-y)}{\int_{\R^n} G_{\varepsilon}(x-z) \mu_t(\mathrm{d}z)}  \mu_t(\mathrm{d}y)\Lambda_t(\mathrm{d}m) 
        \\
        &=
         \int_{\Pc^n_U}  \int_{U} \sigma(m)^2 u^2 m(\mathrm{d}u,\mathrm{d}y) \frac{G_{\varepsilon}(x-y)}{\int_{\R^n} G_{\varepsilon}(x-z) m(\mathrm{d}z,U)}  \Lambda_t(\mathrm{d}m). 
    \end{align*}
    Notice that, now, when $\Lambda$ is given, \Cref{eq:approx_MV} or its associated Fokker--Planck equation satisfies a uniqueness result. Indeed, as $\sigma^\varepsilon$ is smooth in $x,$ \Cref{eq:approx_MV} is uniquely solvable.
    
    Next, we are able to find a sequence of $\Pc^n_U$--valued $(\sigma\{\Lambda_{t \wedge \cdot}\})_{t \in [0,T]}$--predictable processes $(\nub^k)_{k \in \N^*}$ s.t. $\Lim_{k \to \infty} \delta_{\nub^k_t}(\mathrm{d}m)\mathrm{d}t=\Lambda$ $\widetilde{\Pr}$--a.e. If $\mu^{\varepsilon,k}_t=\Lc^{\widetilde{\Pr}}(X^{\varepsilon,k}_t|\Gcb_t)$ is the solution of 
    \begin{align} \label{eq:approx_2_FP}
	        \mathrm{d}\langle f(\cdot-\sigma_0 B_t),\mu^{\varepsilon,k}_t \rangle
	        =
	        \int_{\R^n} f''(x - \sigma_0 B_t) \sigma^\varepsilon(t,\delta_{\nub^k_t}(\mathrm{d}m),x)^2\mu^{\varepsilon,k}_t(\mathrm{d}x) \mathrm{d}t,
	    \end{align}
	    one has, when $\varepsilon > 0$ is fixed, by passing to the limit in \Cref{eq:approx_2_FP} and using uniqueness of \Cref{eq:approx_MV}, we find that $\Lim_{k \to \infty} \mu^{\varepsilon,k}=\mu^\varepsilon$ a.e. Consequently, we can set $k$ and $\varepsilon$ as fixed, and focus on the approximation of \Cref{eq:approx_2_FP} or equivalently of 
	    \begin{align} \label{eq:approx_2_MV}
        \mathrm{d}X^{\varepsilon,k}_t
        =
        \sigma^{\varepsilon}(t,\delta_{\nub^k_t},X^{\varepsilon,k}_t) \mathrm{d}W_t
        +
        \sigma_0 \mathrm{d}B_t,\;\;X^{\varepsilon,k}_0=\xi
    \end{align}

\paragraph*{Strategy of proof: 2--construction of control and discretization} Recall that  $\sigma^{\varepsilon}(t,\delta_{\nub^k_t},x)$ satisfies
 \begin{align*}
        \sigma^\varepsilon(t,\delta_{\nub^k_t},x)^2
        =
        \int_{U} \sigma(\nub^k_t)^2 u^2 \nub^k_t(\mathrm{d}u,\mathrm{d}y) \frac{G_{\varepsilon}(x-y)}{\int_{\R^n} G_{\varepsilon}(x-z) \nub^k_t(\mathrm{d}z,U)}. 
    \end{align*}

Let us assume that it is possible to construct a Borel function $\alpha^{\varepsilon,k}:[0,T] \x U \x \R^n \to U,$ a $\R^n$--valued $\F$--adapted continuous process $\Xt^{\varepsilon,k}$ and a $[0,1]$--valued $\F$--predictable process $F$ satisfying: $F_t$ and $\Xt^{\varepsilon,k}_t$ are conditionally independent given $\Gcb_t,$ 
$$
    \Lc^{\widetilde{\Pr}^{\Gcb_t}}(\alpha^{\varepsilon,k}(t,F_t,\Xt^{\varepsilon,k}_t) | \Xt^{\varepsilon,k}_t=x)=\int_{U} \nub^k_t(\mathrm{d}u,\mathrm{d}y) \frac{G_{\varepsilon}(x-y)}{\int_{\R^n} G_{\varepsilon}(x-z) \nub^k_t(\mathrm{d}z,U)},
$$
and $\Xt^{\varepsilon,k}$ satisfies
\begin{align*}
        \mathrm{d}\Xt^{\varepsilon,k}_t
        =
        \sigma( \nub^k_t) \alpha^{\varepsilon,k}(t,F_t,\Xt^{\varepsilon,k}_t)  \mathrm{d}W_t
        +
        \sigma_0 \mathrm{d}B_t,\;\;\Xt^{\varepsilon,k}_0=\xi.
    \end{align*}
Notice that, by uniqueness of \Cref{eq:approx_2_FP}, $\Lc(\Xt^{\varepsilon,k}_t|\Gcb_t)=\mu^{\varepsilon,k}_t$ a.e. for all $t \in [0,T].$ Given $(\alpha^{\varepsilon,k}, \Xt^{\varepsilon,k}, F),$ our last sequence is then given by: $Y^{\varepsilon,k}$ solution of
\begin{align*}
        \mathrm{d}Y^{\varepsilon,k}_t
        =
        \sigma(\overline{m}^{\varepsilon,k}_t) \alpha^{\varepsilon,k}(t,F_t,\Xt^{\varepsilon,k}_t)  \mathrm{d}W_t
        +
        \sigma_0 \mathrm{d}B_t,\;\mbox{with}\;m^{\varepsilon,k}_t:=\Lc(Y^{\varepsilon,k}_t |\Gcb_t)\;\mbox{and}\;\overline{m}^{\varepsilon,k}_t:=\Lc(Y^{\varepsilon,k}_t,\alpha^{\varepsilon,k}(t,F_t,\Xt^{\varepsilon,k}_t) |\Gcb_t).
    \end{align*}
    
    By using some technical results, proving in \Cref{prop:conservation_to_the_limit} and \Cref{lemm:appr_coef}, we deduce that
    \begin{align*}
        \Lim_{k \to \infty} \Lim_{\varepsilon \to 0} \widetilde{\Pr}\circ \big( m^{\varepsilon,k},\delta_{\overline{m}^{\varepsilon,k}_t}(\mathrm{d}m)\mathrm{d}t, B\big)^{-1}
        =
        \widetilde{\Pr}\circ \big(\mu,\Lambda,B \big)^{-1}\;\mbox{in}\;\Wc_p.
    \end{align*}

\medskip    
    The fact is we are not able to construct the tuple $(\alpha^{\varepsilon,k}, \Xt^{\varepsilon,k}, F)$ as presented below. This construction will be done through approximation by discretization in time in \Cref{sec:F--P_Nagents}. Moreover, the framework that we will consider in the next part will be more general than the presentation we have chosen for the main results. The reason is that the techniques we use can be applied to both mean field game and mean field control problem (see our companion paper \cite{MFD-2020_MFG}). Therefore, we made the choice to have a presentation that allows the results to be used in both contexts.



}

\subsection{Regularization of the Fokker--Planck equation} \label{sec:F--P_regularization}

In this part, with  the help of a regularization by convolution, we show that it can be possible to approximate a particular solution of a Fokker--Planck equation with $``$non--smooth$"$ coefficients by a sequence of solutions of Fokker-Planck equations with $``$smooth$"$ coefficients, this part is largely inspired by the proof of \cite[Lemma 2.1]{gyongy1986mimicking}.

    
    \medskip
    Let $\bb \in \Cc^\ell,$ $(\nb_t)_{t \in [0,T]}$ and $(\zb_t)_{t \in [0,T]}$ belong to $\Cc_{\Wc}^n$ and also $\hat \qb_t(\mathrm{d}m,\mathrm{d}m')\mathrm{d}t\in \M((\Pc^n_U)^2)$. Moreover, $(\nb,\zb,\hat \qb, \bb)$ satisfy the following equation: $\nb_0=\nu$ and
    \begin{align*} 
	    \mathrm{d}\langle f(t,.),\nb_t \rangle
	    =
	    \Big[ \langle \partial_t f(t,.),\nb_t \rangle +  \int_{(\Pc^n_U)^2} \langle \Ac_t[f(t,\cdot)](.,\bb,\nb,\zb,m,\nub,.),m \rangle \hat \qb_t(\mathrm{d}m,\mathrm{d}\nub) \Big]\mathrm{d}t, 
	\end{align*}
    for all $(t,f) \in [0,T] \x C^{1,2}_b([0,T] \x \R^n),$ where the generator $\Ac$ is defined by
    \begin{align}
    \label{equation:generator-general-bis}
        \Ac_t\varphi(x,\bb,\nb,\zb,m,\nub,u) 
        &:= 
        \frac{1}{2}  \text{Tr}\big[\hat \sigma \hat \sigma^\top(t,x,\bb,\nb,\zb,m,\nub,u) \nabla^2 \varphi(x) 
        \big] 
        + \hat b(t,x,\bb,\nb,\zb,m,\nub,u)^\top \nabla \varphi(x),
    \end{align}
    with $(\hat b,\hat \sigma):[0,T] \x \R^n \x \Cc^\ell \x (\Cc^n_\Wc)^2 \x (\Pc^n_U)^2 \x U \to \R^n \x \S^n$ is bounded and continuous function in all arguments, and for each $\nub \in \Pc^n_U,$ the map $(\hat b,\hat \sigma)(\cdot,\cdot,\bb,\cdot,\zb,\cdot,\nub,\cdot)$ satisfies Assumption \ref{assum:main1} with constant $\theta$ independent of $\nub$.
    
\begin{remark}
    As said in the end of {\rm \Cref{sec:main_ideas}}, we consider this type of general Fokker--Planck equation because we want to have a formulation useful both in mean field game and mean field control. Here, the mean field game aspect appears in the integration over $\mathrm{d}\nub$ in $\hat \qb$ and $\zb.$ The integration over $\mathrm{d}\nub$ in $\hat \qb$ and $\zb$ play the role of fixed measures as it can happen in mean field game.
\end{remark}    
    
    \medskip
    Let $G \in C^{\infty}(\R^n;\R)$ with compact support satisfying $G \ge 0,$ $G(x)=G(-x)$ for $x \in \R^n,$ and $\int_{\R^n}G(y)\mathrm{d}y=1,$ and define $G_{\varepsilon}(x):={\varepsilon}^{-n} G({\varepsilon}^{-1}x)$ and for all $\pi \in \Pc(\R^n),$ $\pi^{(\varepsilon)}(x):=\int_{\R^n} G_{\varepsilon}(x-y)\pi(\mathrm{dy})$ for all $x \in \R^n.$
	Now, for each $\varepsilon >0,$ let us introduce the generator of the $regularized$ Fokker--Planck equation $\Ac^\varepsilon$: for all $(t,\hat q,x) \in [0,T] \x \Pc((\Pc^n_U)^2) \x \R^n$
	\begin{align} \label{eq:generator-regularized}
        \Ac^\varepsilon_t \varphi[\bb,\nb,\zb,\hat q](x) &:= \frac{1}{2}  \text{Tr}\big[\hat a^\varepsilon[\bb,\nb,\zb,\hat q](t,x) \nabla^2 \varphi(x)\big] 
        + \hat b^\varepsilon[\bb,\nb,\zb,\hat q](t,x)^\top \nabla \varphi(x),
    \end{align}
    where for $(t,x,\gamma,\pi,\beta,m,\nub,u) \in [0,T] \x \R^n \x \Cc^\ell \x (\Cc^n_{\Wc})^2 \x (\Pc^n_U)^2 \x U$, $\hat a(t,x,\gamma,\pi,\beta,m,\nub,u):=\hat \sigma \hat \sigma^\top(t,x,\gamma,\pi,\beta,m,\nub,u)$ and $(\hat a^{\varepsilon},\hat b^{\varepsilon})$ are defined by:
    
\begin{align*} 
    (\hat a^{\varepsilon},\hat b^{\varepsilon})[\bb,\pi,\beta,q](t,x)
    :=
    \int_{(\Pc^n_U)^2} \int_{\R^n} \int_U (a,b)(t,y,\bb_{t \wedge \cdot},\pi_{t \wedge \cdot},\beta_{t \wedge \cdot},m,\nub,u)\frac{G_\varepsilon(x-y)}{(m(\mathrm{d}z,U))^{(\varepsilon)}(x)}m(\mathrm{d}u,\mathrm{d}y)q(\mathrm{d}m,\mathrm{d}\nub),
\end{align*}


\medskip    
    We are now ready to formulate our regularization/approximation result of Fokker--Planck equation. The following proposition is proved in Appendix \ref{proof:lemm:reguralization_FP}.
    
{\color{black}    
\begin{proposition}[Regularization of Fokker-Planck equation] \label{lemm:reguralization_FP}
   Let $\nu \in \Pc_{p}(\R^n)$, for each $\varepsilon > 0,$ there exists a unique solution $(\nb^\varepsilon_t)_{t \in [0,T]} \in \Cc^{n,p}_{\Wc}$  of: $\nb^\varepsilon_0=\nu$ and for all $f \in C^{1,2}_b([0,T] \x \R^n)$ and
    \begin{align} \label{eq:FP-original}
	    \mathrm{d}\langle f(t,.),\nb^\varepsilon_t \rangle
	    ~=~
	    \bigg[ \int_{\R^n} \partial_t f(t,y)\nb^\varepsilon_t(\mathrm{d}y) + \int_{\R^n} \Ac^\varepsilon_t f(t,\cdot)[\bb,\nb,\zb,\hat \qb_r](t,y)\nb^\varepsilon_t(\mathrm{d}y) \bigg] \mathrm{d}t. 
	\end{align}
	Moreover, if $\nu \in \Pc_{p'}(\R^n)$ and $\hat \qb_t(\Z_{\nb_t} \x \Pc^n_U)=1$ $\mathrm{d}t$--for almost every $t \in [0,T],$ then
	\begin{align} \label{eq:conv-FP-regualarized}
	    \lim_{\varepsilon \to 0} \sup_{t \in [0,T]} \Wc_{p}(\nb^\varepsilon_t,\nb_t)=0.
	\end{align}
\end{proposition}
}

\begin{remark} \label{rm:SDE-representation} 
    $(i)$ Let $(\Omh,\widehat{\F},\widehat{\Fc},\P)$ be a probability space supporting $W$ a $\widehat{\F}$--Wiener process of dimension $\R^n$ and $\xi$ a $\Fc_0$--random variable such that $\Lc^{\P}(\xi)(\mathrm{d}y)=\nu(\mathrm{d}y).$ Given $\varepsilon>0,$ let $Y^{\varepsilon}$ be the unique strong solution $($well defined, see {\rm Appendix \ref{proof:lemm:reguralization_FP}} $($more precisely the Proof of {\rm \Cref{lemm:reguralization_FP}$)$} $)$
    \begin{align} \label{eq:SDE-representation}
        \mathrm{d}Y^\varepsilon_t
        =
        \hat b^\varepsilon[\bb,\nb,\zb,\hat \qb_t](t,Y^\varepsilon_t) \mathrm{d}t
        +
        (\hat a^\varepsilon)^{1/2}[\bb,\nb,\zb,\hat \qb_t](t,Y^\varepsilon_t) \mathrm{d}W_t,
        Y^\varepsilon_0=\xi,
    \end{align}
    one has, by uniqueness of \eqref{eq:FP-original},  $\Lc^{\P}(Y^\varepsilon_t)=\nb^\varepsilon_t$ for all $t \in [0,T]$ where $\nb^\varepsilon$ is the solution of \eqref{eq:FP-original}.

\medskip    
    $(ii)$ We will sometimes use the previous lemma with {\rm \Cref{prop:conservation_to_the_limit}}, in which $\nb^\varepsilon$ must be obtainable through a diffusion process that has a volatility term which verifies $\hat a^\varepsilon[\bb,\nb,\zb,\hat \qb_r](t,Y^\varepsilon_t) \ge \theta \mathrm{I}_{n \x n}.$ The SDE \eqref{eq:SDE-representation} allows to say that $\nb^\varepsilon$ satisfies these conditions. Also, from {\rm \Cref{lemm:reguralization_FP}} and the SDE representation \eqref{eq:SDE-representation}, it is straightforward to see that the measure $\nb_t(\mathrm{d}x)\mathrm{d}t$ is equivalent to the Lebesgue measure on $\R^n \x [0,T]$ $($see for instance {\rm \Cref{Prop:absolutelyContinuous}} $).$

\end{remark}

\begin{remark}
\label{remark:estimates}
    Combining {\rm \Cref{rm:SDE-representation}} $($diffusion form \eqref{eq:SDE-representation} of $\nb^\varepsilon$ $)$ with {\rm \Cref{lemm:reguralization_FP}} $($convergence result \eqref{eq:conv-FP-regualarized}$)$, as $(b,\sigma)$ are bounded, there exists a constant $C>0,$ depending only of coefficients $(b,\sigma),$ $p$ and $p',$ such that
    \begin{align*}
        \sup_{r \in [0,T]} \int_{\R^n} |x|^{p'} \nb_{r}(\mathrm{d}x) 
        \le C~\Big(1+ \int_{\R^n} |x|^{p'} \nu(\mathrm{d}x) \Big)\;\;\mbox{and}\;\;\Wc_p \big(\nb_s, \nb_t \big)^p \le C|t-s|,~\mbox{for all}~(t,s) \in [0,T] \x [0,T].
    \end{align*}
\end{remark}

\subsection{Approximation by $N$--agents} \label{sec:F--P_Nagents}


Now, let us formulate the approximation result of Fokker--Planck equation by $N$-interacting SDE equations. In order to achieve this, we first describe the associated framework. 

\medskip
Let $\big(\Om^{\qb},\Fc^{\qb},\F^{\qb},\Q \big)$ be a filtered probability space supporting $(B_t)_{t \in [0,T]}$ a $\R^\ell$--valued $\F^{\qb}$--adapted continuous process,  $(\mu_t)_{t \in [0,T]}$ and $(\zeta_t)_{t \in [0,T]}$  two $\Pc(\R^n)$--valued $\F^{\qb}$--continuous processes, $\overline{\Lambda}$ a $\M\big((\Pc^n_U)^2 \big)$--valued variable such that $(\overline{\Lambda}_t)_{t \in [0,T]}$ is $\F^{\qb}$--predictable. Besides, $(\mu,B,\zeta,\overline{\Lambda})$ satisfy: $\overline{\Lambda}_t\big(\Z_{\mu_t} \x \Pc^n_U \big)=1,$ for $\mathrm{d}\Q \otimes \mathrm{d}t$--almost surely, and $\Q$--a.e.
 \begin{align} \label{eq:main-FokkerPlanck}
	    \mathrm{d}\langle f,\mu_t \rangle
	    =
	    \int_{\Pc_U^n \x \Pc_U^n} \int_{\R^n \x U} \Ac_t f(y,B,\phi(\mu),\zeta,m,\nub,u)m(\mathrm{d}y,\mathrm{d}u) \overline{\Lambda}_t(\mathrm{d}m,\mathrm{d}\nub)
	   \mathrm{d}t,\;\mu_0=\nu,
	\end{align}
	
for all $t \in [0,T]$ and $f \in C^{2}_b(\R^n),$ where
\begin{equation}
\label{equation:generator-general}
    \Ac_t\varphi(x,\bb,\pi,\beta,m,\nub,u) 
    := 
    \frac{1}{2}  \text{Tr}\big[\hat \sigma \hat \sigma^{\top}(t,x,\bb,\pi,\beta,m,\nub,u) \nabla^2 \varphi(x) 
    \big] 
    +
    \hat b(t,x,\bb,\pi,\beta,m,\nub,u)^\top \nabla \varphi(x),
\end{equation}
with, as in \eqref{equation:generator-general-bis}, $(\hat b,\hat \sigma)$ is continuous in all arguments and bounded, and the map $(\hat b,\hat \sigma)(\cdot,\cdot,\bb,\cdot,\beta,\cdot,\nub,\cdot)$ satisfies {\color{black}Assumption} \ref{assum:main1} with constant $C$ and $\theta$ independent of $(\bb,\beta,\nub)$ (see Assumption \ref{assum:main1}). Besides, $\phi: \Cc^n_{\Wc} \to \Cc^n_{\Wc}$ is a Lipschitz function s.t. for all $t \in [0,T]$, $\phi_t(\pi)=\phi_t(\pi_{t \wedge \cdot}).$

\begin{remark}
    $(i)$ Notice that,  {\color{black}\eqref{eq:main-FokkerPlanck} is an equation over $\mu$ in the sense that with the condition $\overline{\Lambda}_t\big(\Z_{\mu_t} \x \Pc^n_U \big)=1,$ for $\mathrm{d}\Q \otimes \mathrm{d}t$--almost surely, the process $\mu$ appears on both sides on the equality.} Under general {\rm \Cref{assum:main1}}, it is not difficult to show that there are processes $(\mu,\overline{\Lambda})$ {\color{black}verifying} equation \eqref{eq:main-FokkerPlanck} $($see for instance {\rm \cite[Theorem A.2]{djete2019general}}$).$ However, without additional assumptions, a uniqueness result cannot be expected.
        
    \medskip
    $(ii)$ This type of Fokker--Planck equation appears especially in the study of optimal control of McKean-Vlasov equation $($see {\rm \Cref{sec:proofs}} above$)$ and mean field game $($see {\rm \cite{MFD-2020_MFG}}$)$. One the most important variable is $\overline{\Lambda}.$ It can play the role of control in optimal control of McKean-Vlasov equation, but also of external parameter as it is the case in the mean field game.
    \end{remark}

\medskip
Let $(\Omh,\widehat{\Fc},\widehat{\F},\widehat{\P})$ be another filtered probability space supporting: 
\begin{itemize}
    \item $(W^i)_{i \in \N^*}$ a sequence of $\R^{n}$--valued independent $\widehat{\F}$--Brownian motions and $(\xi^i)_{i \in \N^*}$ a sequence of independent  $\widehat{\Fc}_0$--random variables s.t. $\Lc^{\widehat{\P}}(\xi_i)=\nu_i \in \Pc_{p'}(\R^n),$
    
    
    
    \item $(\mu^N)_{N\in \N^*}$ and $(\zeta^N)_{N\in \N^*}$ two sequences of $\Pc(\R^n)$--valued $\widehat{\F}$--adapted continuous processes, and $(B^N)_{N \in \N^*}$ a sequence of $\R^{\ell}$--valued  $\widehat{\F}$--adapted continuous processes,
    
    \item $(m^N)_{N\in \N^*}$ and $(\nub^N)_{N\in \N^*}$ two  sequences of $\Pc^n_U$--valued $\widehat{\F}$--predictable processes,
\end{itemize}
satisfying:
\begin{align} \label{eq:general-condition}
        \Lim_{N \to \infty} \Wc_{p'} \bigg(\frac{1}{N} \sum_{i=1}^N \nu^i,\nu \bigg)=0\;\;\mbox{and}\;\Lim_{N \to \infty} \Lc^{\widehat{\P}} \Big(\phi(\mu^{N}),\zeta^{N},\overline{\Lambda}^N,\;B^{N} \Big)
        =
        \Lc^{\Q} \big(\phi(\mu),\zeta,\overline{\Lambda},B \big),\;\;\mbox{in}\;\Wc_p,
    \end{align}
    $\mbox{where}\;\;\overline{\Lambda}^{N}_t(\mathrm{d}m,\mathrm{d}\nub)\mathrm{d}t:=\delta_{(m^N_t,\;\nub^N_t )}(\mathrm{d}m,\mathrm{d}\nub)\mathrm{d}t.$
    
\medskip
    Furthermore, let $(Z^i)_{i \in \N^*}$ be a sequence of independent $[0,1]$--valued $\widehat{\Fc}$--measurable uniform variables independent of other variables, and for each $(i,N) \in \N^* \x \N^*,$ denote by $\widehat{\F}^{i,N}:=(\widehat{\Fc}^{i,N}_t)_{t \in [0,T]}$ the filtration defined by:
    \begin{align} \label{def:filtration}
        \widehat{\Fc}^{i,N}_t
        :=
        \sigma 
        \Big{\{}
            \xi^i, \overline{\Lambda}^{N}_{t \wedge \cdot},\phi_{t \wedge \cdot}(\mu^N),\zeta^N_{t \wedge \cdot},W^i_{t \wedge \cdot}, B^N_{t \wedge \cdot},Z^i
        \Big{\}},\;\mbox{for each}\;t \in [0,T].
    \end{align}
    
The next proposition describes an approximation by a sequence of $N$--interacting processes of the Fokker--Planck equation \eqref{eq:main-FokkerPlanck}.
\begin{proposition} \label{prop:approximation-FP_BY_SDE-2}

    There exists a sequence of processes $(\alpha^{i,N})_{(i,N) \in \N^* \x \N^*}$ satisfying for each $(i,N) \in \N^* \x \N^*,$ $\alpha^{i,N}$ is $\widehat{\F}^{i,N}$--predictable, s.t. if we let $(\widehat{X}^{1}_t,...,\widehat{X}^{N}_t)_{t \in [0,T]}$ be the continuous processes unique strong solution of: for each $i \in \{1,...,N\},$ $\E^{\widehat{\P}}[\|\widehat{X}^i\|^{p'}]< \infty,$  for all $t \in [0,T]$
\begin{align} \label{eq:general-Nagents}
    \widehat{X}^i_t
    =
    \xi^i
    &+
    \int_0^t \hat b\big(r,\widehat{X}^i_r,B^N,\phi(\widehat{\mu}^N),\zeta^N, \widehat{m}^N_r,\nub^N_r,\alpha^{i,N}_{r}\big) \mathrm{d}r
    +
    \int_0^t \hat \sigma \big(r,\widehat{X}^i_r,B^N,\phi(\widehat{\mu}^N),\zeta^N, \widehat{m}^N_r,\nub^N_r,\alpha^{i,N}_{r}\big)
    \mathrm{d}W^i_r,
    ~
    \widehat{\P}\mbox{--a.e.}
\end{align}
    where $\widehat{m}^N_t(\mathrm{d}x,\mathrm{d}u):=\frac{1}{N} \sum_{i=1}^N \delta_{(\widehat{X}^i_t,\;\alpha^{i,N}_t)}(\mathrm{d}x,\mathrm{d}u),\;\widehat{\mu}^N_t(\mathrm{d}x):=\widehat{m}^N_t(\mathrm{d}x,U),$
    then, one has, for a sub-sequence $(N_k)_{k \in \N^*} \subset \N^*,$
    \begin{align*}
        \Lim_{k \to \infty} \E^{\widehat{\P}} \bigg[\int_0^T \Wc_p \big(\widehat{m}^{N_k}_t,m^{N_k}_t \big)^p \mathrm{d}t + \sup_{t \in [0,T]} \Wc_p \Big(\phi_t(\widehat{\mu}^{N_k}),\;\phi_t(\mu^{N_k}) \Big) \bigg]=0
    \end{align*}
    and 
    \begin{align} \label{eq:cv_result_1}
        \Lim_{k \to \infty} \Lc^{\widehat{\P}} \Big(\widehat{\mu}^{N_k},\zeta^{N_k},\widehat{\Lambda}^{N_k},B^{N_k} \Big)
        =
        \Lc^{\Q} \big(\mu,\zeta,\overline{\Lambda},B \big),\;\mbox{in}\;\Wc_p\;\;\mbox{with}\;\;\widehat{\Lambda}^{N_k}_s(\mathrm{d}m,\mathrm{d}\nub)\mathrm{d}s:=\delta_{(\widehat{m}^{N_k}_s,\nub^{N_k}_s)}(\mathrm{d}m,\mathrm{d}\nub)\mathrm{d}s.
    \end{align}
\end{proposition}

\begin{remark}
    $(i)$ {\rm \Cref{prop:approximation-FP_BY_SDE-2}} as well as {\rm\Cref{prop:weak-appr}} $($see below$)$ can be considered as a general characterization of Fokker--Planck equation of type \eqref{eq:main-FokkerPlanck} via a sequence of SDE processes interacting through the empirical distribution of the states and $``$controls$"$. These results are very useful both in the study of extended mean field control problem $($see {\rm \Cref{prop:approximation_weak}}$)$ and in mean field game of controls $($see our companion paper {\rm \cite{MFD-2020_MFG}}$).$
    
    \medskip
    $(ii)$  Because of non--uniqueness of Fokker--Planck equation \eqref{eq:main-FokkerPlanck}, the condition \eqref{eq:general-condition} is a crucial and essential assumption.  Furthermore, notice that, the condition \eqref{eq:general-condition} does not require any equation verified by the sequence $\big(\phi(\mu^{N}),\zeta^{N},\overline{\Lambda}^N,\;B^{N} \big)_{N \in \N^*}.$ Only the convergence result \eqref{eq:general-condition} is necessary. 
    
    \medskip
    $(iii)$ Observe that, the sequence $(\Lambda^N)_{N \in \N^*}$ is a subset of $\M_0\big( (\Pc^n_U)^2 \big)$ and not a general subset of $\M\big( (\Pc^n_U)^2 \big).$  
    For an understandable and easy presentation, we consider this type of sequence, but a general subset of $\M\big( (\Pc^n_U)^2 \big)$ is possible $($see {\rm Proposition \ref{prop:approximation-FP_BY_SDE-1}} below$)$.
    
    \medskip
    {\color{black}$(iv)$ The presence of the map $\phi,$ notably in \eqref{eq:general-condition}, specifies the condition needed on $\mu$ for the result. In particular, if $\phi$ is null, it means that no assumption of convergence towards $\mu$ is necessary to find a sequence of SDE processes converging to $\mu.$ }
\end{remark}

\begin{proof}[Proof of Proposition \ref{prop:approximation-FP_BY_SDE-2}]


The proof is divided in three steps for a better understanding.


\medskip
$\mathbf{\underline{Step\;1:Approximation\;by\;regularization\;of\;\mbox{F-P}\;equation}}$: Let $\varepsilon>0$ and  recall that $\Ac^\varepsilon$ is defined in \eqref{eq:generator-regularized}. For all $\om \in \Om^{\qb},$ by \Cref{lemm:reguralization_FP}, there exists a continuous process $(\mu^{\varepsilon}_t(\om))_{t \in [0,T]}$ verifying
\begin{align} \label{eq:first_regularized_FP}
    \mathrm{d}\langle f,\mu^{\varepsilon}_t(\om) \rangle
    ~=~
    \int_{\R^n} \Ac^{\varepsilon}_t f\big[B(\om),\phi(\mu(\om)),\zeta(\om),\overline{\Lambda}_t(\om)\big](x) \mu^{\varepsilon}_t(\om)(\mathrm{d}x) \mathrm{d}t,\;\mu^\varepsilon_0=\nu, 
\end{align}
for all $f \in C^{2}_b(\R^n;\R)$ and for $\Q$--a.e. $\om \in \Om^{\qb},$ $\lim_{\varepsilon \to 0} \sup_{t \in [0,T]}\Wc_p(\mu^{\varepsilon}_t(\om),\mu_t(\om)).$
Also, by \Cref{lemma:existence-measurability_FP}, there is a function $\Phi^\varepsilon: \Cc^\ell \x \Cc^n_{\Wc} \x \Cc^n_{\Wc} \x \M \big((\Pc^n_U)^2 \big) \to \Cc^n_{\Wc}$ such that $\Q$--a.e. $\om \in \Om^{\qb}$
\begin{align} \label{eq:function_FK-PL}
    \mu^\varepsilon_t(\om)
    =
    \Phi^\varepsilon_t\Big( B_{t \wedge \cdot}(\om),\phi_{t \wedge \cdot}(\mu(\om)),\zeta_{t \wedge \cdot}(\om),\overline{\Lambda}_{t \wedge \cdot}(\om) \Big),\;\mbox{for all}\;t \in [0,T] .
\end{align}

\medskip
$\mathbf{\underline{Step\;2:Approximation\;by\;discretization}}$:
Now, let us define for all $(x,m) \in \R^n \x \Pc^n_U,$ the probability
\begin{align*}
    H^\varepsilon(x,m)(\mathrm{d}u)
    :=
    \int_{\R^n} m(\mathrm{d}u,\mathrm{d}y)\frac{G_\varepsilon(x-y)}{(m(U,\mathrm{d}z))^{(\varepsilon)}(x)}.
\end{align*}
Recall that $G \in C^{\infty}(\R^n;\R)$ with compact support satisfying $G \ge 0,$ $G(x)=G(-x)$ for $x \in \R^n,$ and $\int_{\R^n}G(y)\mathrm{d}y=1.$ We denoted $G_{\varepsilon}(x):={\varepsilon}^{-n} G({\varepsilon}^{-1}x)$ and for all $\pi \in \Pc(\R^n),$ $\pi^{(\varepsilon)}(x):=\int_{\R^n} G_{\varepsilon}(x-y)\pi(\mathrm{dy})$ for all $x \in \R^n.$    
By \citeauthor*{blackwellDubins83} \cite{blackwellDubins83}, there exists a Borel application $N^\varepsilon: (x,m,v) \in \R^n \x \Pc^n_U \x [0,1] \to N^\varepsilon(x,m)(v) \in U$ s.t. for all $(x,m) \in \R^n \x \Pc^n_U$ and any $[0,1]$--valued uniform random variable $F,$
\begin{align*}
    \widehat{\P} \circ \big(N^\varepsilon(x,m)(F) \big)^{-1}(\mathrm{d}u)=H^\varepsilon(x,m)(\mathrm{d}u).
\end{align*}


\hspace{2mm}$\underline{Step\;2.1:Construction\;of\;scheme\;of\;discretization}$: Let us consider the partition $(t^N_k)_{1 \le k \le 2^N}$ with $t^N_k=\frac{kT}{2^N}$, and take a sequence of $\R^n$--valued independent Brownian motions $(Z^i)_{i \in \N^*}$, independent of all of other variables. Let $\varphi:[0,T] \x \R^n \to [0,1]$ be a Borel function such that, for all $t \in [0,T],$ $\Lc^{\widehat{\P}}(\varphi(t-t^N_k,Z^i_t-Z^i_{t^N_k}))$ is the uniform law when $t>t^N_k.$  For all $i \in \{1,...,N\},$ denote by $V^{i,N}_t:=\varphi(t-t^N_k,Z^i_t-Z^i_{t^N_k}),$ when $t \in [t^N_k,t^N_{k+1}),$ and given $\varepsilon>0,$ we define on $(\Omh,\widehat{\F},\widehat{\Fc},\widehat{\P}),$ by Euler scheme, $X^{\varepsilon,i,N}:=X^{i}$ as follows: $X^{i}_0:=\xi^i$ and
\begin{align} \label{eq:auxillary-appr}
    X^i_t
    =
    X^i_{0}
    &+
    \int_{0}^t \widehat{B} \Big(s,X^i_{[s]^N},B^N,\phi(\mu^N),\zeta^N,m^N_s,\nub^{N}_s, N^\varepsilon(X^i_{[s]^N},m^N_s)(V^{i,N}_s) \Big) \mathrm{d}s \nonumber
    \\
    &+
    \int_{0}^t \widehat{\Sigma} \Big(s,X^i_{[s]^N},B^N,\phi(\mu^N),\zeta^N,m^N_s,\nub^{N}_s, N^\varepsilon(X^i_{[s]^N},m^N_s)(V^{i,N}_s) \Big)
    \mathrm{d}W^i_s,\;\mbox{for all}\;t \in [0,T],\;i \in \{1,...,N\},
\end{align}
where $[s]^N=t^N_k$ if $t^N_k \le s < t^N_{k+1}$, and, for $s \in [t^N_k,t^N_{k+1}),$
\begin{align*}
    &\widehat B \Big(s,X^i_{t^N_k},B^N,\phi(\mu^N),\zeta^N,m^N_s,\nub^{N}_s, N^\varepsilon(X^i_{t^N_{k}},m^N_s)(V^{i,N}_s) \Big)
    \\
    &~~~~~~:=
    \hat b\Big(s,X^i_{t^N_{k}},B^N,\phi(\mu^N),\zeta^N,m^N_s,\nub^{N}_s,N^\varepsilon(X^i_{t^N_{k}},m^N_s)(V^{i,N}_s)\Big)
    +
    B \Big(s,X^i_{t^N_k},B^N,\phi(\mu^N),\zeta^N,m^N_s,\nub^{N}_s \Big),
\end{align*}
and 
\begin{align*}
    &\widehat \Sigma \Big(s,X^i_{t^N_k},B^N,\phi(\mu^N),\zeta^N,m^N_s,\nub^{N}_s, N^\varepsilon(X^i_{t^N_{k}},m^N_s)(V^{i,N}_s) \Big)
    \\
    &~~~~~~~~:=
    \Sigma\Big(s,X^i_{t^N_k},B^N,\phi(\mu^N),\zeta^N,m^N_s,\nub^{N}_s \Big) \hat \sigma\Big(s,X^i_{t^N_{k}},B^N,\phi(\mu^N),\zeta^N,m^N_s,\nub^{N}_s,N^\varepsilon(X^i_{t^K_{k}},m^N_s)(V^{i,N}_s )\Big),
\end{align*}
with
\begin{align*}
    &B \Big(s,X^i_{t^N_k},B^N,\phi(\mu^N),\zeta^N,m^N_s,\nub^{N}_s \Big)
    \\
    &:=
    \Bigg[
    \hat b^\varepsilon\big[B^N,\phi(\mu^N),\zeta^N,\overline{\Lambda}^N_s \big](s,X^i_{t^N_k})
    -
    \int_U \hat b\big(s,X^i_{t^N_k},B^N,\phi(\mu^N),\zeta^N,m^N_s,\nub^{N}_s,u\big) H^\varepsilon(X^i_{t^N_k}, m^N_s)(\mathrm{d}u)
    \Bigg],
\end{align*}
and
\begin{align} \label{def:coeff2nd}
    &\Sigma\Big(s,X^i_{t^N_k},B^N,\phi(\mu^N),\zeta^N,m^N_s,\nub^{N}_s \Big) \nonumber
    \\
    &:=
    \Bigg[
    \hat a^\varepsilon\big[B^N,\phi(\mu^N),\zeta^N,\overline{\Lambda}^N_s \big](s,X^i_{t^N_k})^{1/2} \bigg( \int_U \hat a\big(s,X^i_{t^N_k},B^N,\phi(\mu^N),\zeta^N,m^N_s,\nub^{N}_s,u\big) H^\varepsilon(X^i_{t^N_k},m^N_s)(\mathrm{d}u)\bigg)^{-1/2}
    \Bigg],
\end{align}
recall that $\overline{\Lambda}^N_s(\mathrm{d}m,\mathrm{d}\nub)\mathrm{d}s:=\delta_{m^N_s}(\mathrm{d}m)\delta_{\nub^N_s}\mathrm{d}\nub)\mathrm{d}s.$

\medskip
Notice that, there exists a Borel function $F^N: \R^n \x \M\big((\Pc^n_U)^2\big)\x \Cc^n_{\Wc} \x \Cc^n_{\Wc} \x \Cc^n \x \Cc^n \x \Cc^\ell \to \Cc^n$ s.t. for each $i \in \{1,...,N\},$
\begin{align} \label{eq:measurable-property}
    X^{i}_t
    =
    F^{N}_t\big(\xi^i,\overline{\Lambda}^{N}_{t \wedge \cdot},\phi_{t \wedge \cdot}(\mu^N),\zeta^N_{t \wedge \cdot},W^i_{t \wedge \cdot},Z^{i}_{t \wedge \cdot}, B^N_{t \wedge \cdot} \big),\;\mbox{for all}\;t \in [0,T],\;\widehat{\P}\mbox{--a.e.}
\end{align}

\medskip
\hspace{2mm}$\underline{Step\;2.2:Compactness\;and\;identification\;of\;the\;limit}$: At this stage, we want to show a compactness result and identify the limit of a certain sequence of probability measures constructed from the SDE process $(X^1,...,X^N).$

Using the assumptions imposed on coefficients  $(\hat b,\hat \sigma)$ (see the definition of the generator $\Ac$ in \eqref{equation:generator-general}), especially the fact that $\hat \sigma \hat \sigma^\top \ge \theta \mathrm{I}_{n}$ and $(\hat b,\hat \sigma)$ are bounded, one has that $[\widehat B, \widehat \Sigma]$ are bounded and there exists a constant $\mathrm{D}>0$ such that for all $\varepsilon$ and $N$
\begin{align} \label{ineq:regularity-trajectory}
    \sup_{i \in \{1,...,N\}}\E^{\widehat{\P}} 
    \Big[
    \big|X^{\varepsilon,i,N}_t
    -
    X^{\varepsilon,i,N}_s
    \big|^{p}
    \Big]
    \le
    \mathrm{D}|t-s|,\;\mbox{for all}\;(t,s) \in [0,T] \x [0,T].
\end{align}

Moreover, by using the fact that $\sup_{N \ge 1} \frac{1}{N} \sum_{i=1}^N \int_{\R^n} |x|^{p'} \nu^i(\mathrm{d}x) < \infty$ (see condition \eqref{eq:general-condition}), it is straightforward to verify that: $\sup_{N \ge 1} \frac{1}{N} \sum_{i=1}^N \E^{\widehat{\P}} \Big[\sup_{t \in [0,T]}|X^{\varepsilon,i,N}_t|^{p'} \Big] < \infty.$ Then, by \cite[Proposition A.2]{carmona2014mean}  or/and  \cite[Proposition-B.1]{carmona2014mean}, for each $\varepsilon>0,$ the sequence $(\mathrm{P}^{N})_{N \in \N^*}$ is relatively compact in $\Wc_p,$ where
\begin{align*}
    \mathrm{P}^{N}:=\widehat{\P} \circ \Big(\vartheta^{N},\phi(\mu^N),\zeta^N,\overline{\Lambda}^{N},B^N \Big)^{-1} \in \Pc \Big(\Cc^n_{\Wc} \x \Cc^n_{\Wc} \x \Cc^n_{\Wc} \x \M\big((\Pc^n_U)^2 \big) \x \Cc^\ell \Big)
\end{align*}
with $\vartheta^{N}_t(\mathrm{d}x):=\frac{1}{N} \sum_{i=1}^N \delta_{X^{\varepsilon,i,N}_t}(\mathrm{d}x).$

\medskip
Let us identify the limit of any convergent sub--sequence of $(\mathrm{P}^{N})_{N \in \N^*}.$  For sake of clarity, we use the notation $X^i$ instead of $X^{\varepsilon,i,N}.$ Recall that for the time being $\varepsilon >0$ is considered as fixed.

For each $N \in \N^*,$ $i \in \{1,...,N\},$ and $(s,u) \in [0,T] \x U,$ let $\big[\hat b^{\varepsilon,i,N}_s,\hat a^{\varepsilon,i,N}_s \big]:=[\hat b^\varepsilon,\hat a^{\varepsilon}]\big[B^N,\phi(\mu^N),\zeta^N,\overline{\Lambda}^N_s \big](s,X^i_{[s]^N})$ and
\begin{align*}
    \Big[\hat b^{i,N}_s, \hat a^{i,N}_s,\widehat B^{i,N}_s, \widehat{\Sigma}^{i,N}_s, \widehat A^{i,N}_s \Big](u):=\Big[\hat b,\hat a,\widehat B, \widehat \Sigma,\widehat \Sigma \widehat \Sigma^{\top}\Big]\big(s,X^i_{[s]^N},B^N,\phi(\mu^N),\zeta^N,m^N_s,\nub^{N}_s,u\big).
\end{align*}

By It\^o's formula, for all $f \in C^{\infty}_b(\R^n)$ and $t \in [0,T]$
\begin{align*}
    &\langle f,\vartheta^{N}_t \rangle
    =
    \langle f,\vartheta^{N}_0 \rangle
    +
    \frac{1}{N} \sum_{i=1}^N~\int_0^t\nabla f(X^{i}_s)\widehat \Sigma^{i,N}_s\big(N^\varepsilon(X^i_{[s]^N},m^N_s)(V^{i,N}_s) \big) \mathrm{d}W^i_s
    \\
    &~~~+\frac{1}{N} \sum_{i=1}^N\int_0^t\;\Bigg[\nabla f(X^i_s)\widehat B^{i,N}_s\big(N^\varepsilon(X^i_{[s]^N},m^N_s)(V^{i,N}_s) \big)\;
    +\;\frac{1}{2} \mathrm{Tr}\Big[\widehat A^{i,N}_s \big(N^\varepsilon(X^i_{[s]^N},m^N_s)(V^{i,N}_s) \big)  \nabla^2 f(X^{i}_s) \Big]\Bigg]\;\mathrm{d}s
    \\
    &=
    \langle f,\vartheta^{N}_0 \rangle
    +
    \frac{1}{N} \sum_{i=1}^N~\int_0^t\nabla f(X^{i}_{s})\widehat \Sigma^{i,N}_s\big(N^\varepsilon(X^i_{[s]^N},m^N_s)(V^{i,N}_s) \big) \mathrm{d}W^i_s
    \\
    &~+\frac{1}{N} \sum_{i=1}^N\int_0^t \Bigg[ \nabla f(X^{i}_{[s]^N})\widehat B^{i,N}_s\big(N^\varepsilon(X^i_{[s]^N},m^N_s)(V^{i,N}_s) \big)+\frac{1}{2} \mathrm{Tr}\Big[\widehat A^{i,N}_s \big(N^\varepsilon(X^i_{[s]^N},m^N_s)(V^{i,N}_s) \big)  \nabla^2 f(X^{i}_{[s]^N}) \Big] \Bigg] \mathrm{d}s
    \\
    &~+
    \frac{1}{N} \sum_{i=1}^N~\int_0^t\big[\nabla f(X^{i}_{s})-\nabla f(X^{i}_{[s]^N}) \big]\widehat B^{i,N}_s\big(N^\varepsilon(X^i_{[s]^N},m^N_s)(V^{i,N}_s) \big) \mathrm{d}s
    \\
    &~~~~~~~~~~~~~~~~~~~~+\frac{1}{N} \sum_{i=1}^N \int_0^t \frac{1}{2} \mathrm{Tr}\Big[\widehat A^{i,N}_s \big(N^\varepsilon(X^i_{[s]^N},m^N_s)(V^{i,N}_s) \big)  \big[\nabla^2 f(X^{i}_{s})-\nabla^2 f(X^{i}_{[s]^N}) \big] \Big] \mathrm{d}s.
\end{align*}

Observe that, for $s \in (t^N_k,t^N_{k+1}),$ for each $i \neq j,$ $[\widehat B]^{i,j}_s=[\widehat A]^{i,j}_s=0,$ where
\begin{align*} 
    &[\widehat B]^{i,j}_s 
    :=\E^{\widehat{\P}} \bigg[ \nabla f(X^{i}_{[s]^N})\Big\{\widehat B^{i,N}_s\big(N^\varepsilon(X^i_{[s]^N},m^N_s)(V^{i,N}_s) \big)-\hat b^{\varepsilon,i,N}_s  \Big\} 
    \nabla f(X^{j}_{[s]^N})\Big\{\widehat B^{j,N}_s\big(N^\varepsilon(X^j_{[s]^N},m^N_s)(V^{j,N}_s) \big)-\hat b^{\varepsilon,j,N}_s  \Big\}  \bigg]
\end{align*} 
and
\begin{align*}
    &[\widehat A]^{i,j}_s 
    :=\E^{\widehat{\P}} \bigg[ \Big\{\widehat A^{i,N}_s\big(N^\varepsilon(X^i_{[s]^N},m^N_s)(V^{i,N}_s) \big)-\hat a^{\varepsilon,i,N}_s  \Big\}\nabla^2 f(X^{i}_{[s]^N}) 
    \Big\{\widehat A^{j,N}_s\big(N^\varepsilon(X^j_{[s]^N},m^N_s)(V^{j,N}_s) \big)-\hat a^{\varepsilon,j,N}_s   \Big\}\nabla^2 f(X^{j}_{[s]^N})  \bigg].
\end{align*}

Indeed, by using the fact that: for all $(x,m,e) \in \R^n \x \Pc^n_U \x \{1,...,N\},$ $\widehat{\P} \circ \big(N^\epsilon(x,m)(V^{e,K}_s) \big)^{-1}(\mathrm{d}u)=H^\varepsilon(x,m)(\mathrm{d}u),$ and $(V^i_s,V^j_s)$ are independent and independent of other variables, one has
\begin{align} \label{eq:vanish-productB}
    [\widehat B]^{i,j}_s&=
    \E^{\widehat{\P}} \Bigg[ \nabla f(X^{i}_{[s]^N})\Big\{ \hat b^{i,N}_s\big( N^\varepsilon(X^i_{[s]^N},m^N_s)(V^{i,N}_s) \big)-\int_U \hat b^{i,N}_s(u) H^\varepsilon(X^i_{[s]^N}, m^N_s)(\mathrm{d}u)  \Big\} \nonumber
    \\
    &~~~~~~~~~\nabla f(X^{j}_{[s]^N})\Big\{ \hat b^{i,N}_s\big(N^\varepsilon(X^j_{[s]^N},m^N_s)(V^{j,N}_s) \big)-\int_U \hat b^{i,N}_s(u) H^\varepsilon(X^j_{[s]^N}, m^N_s)(\mathrm{d}u)  \Big\}  \Bigg]=0.
\end{align}

By similar way, if we denote by $\Sigma^{i,N}_s:=\Sigma\Big(s,X^i_{[s]^N},B^N,\phi(\mu^N),\zeta^N,m^N_s,\nub^{N}_s \Big),$ one finds
\begin{align} \label{eq:vanish-productA}
    [\widehat A]^{i,j}_s
    &=\E^{\widehat{\P}} \Bigg[ \nabla^2 f(X^{i}_{[s]^N})\Big\{ \Sigma^{i,N}_s\;\hat a^{i,N}_s\big(N^\varepsilon(X^i_{[s]^N},m^N_s)(V^{i,N}_s) \big)\;(\Sigma^{i,N}_s)^{\top}-\hat a^{\varepsilon,i,N}_s  \Big\} \nonumber
    \\
    &~~~~~~~~~~~~~~~~~~~~~~~~~~~~~~~~~\nabla^2 f(X^{j}_{[s]^N})\Big\{ \Sigma^{j,N}_s\;\hat a^{j,N}_s\big(N^\varepsilon(X^j_{[s]^N},m^N_s)(V^{j,N}_s) \big)\;(\Sigma^{j,N}_s)^{\top}-\hat a^{\varepsilon,j,N}_s  \Big\}  \Bigg] \nonumber
    \\
    &=
    \E^{\widehat{\P}} \Bigg[ \nabla^2 f(X^{i}_{[s]^N})\Big\{ \Sigma^{i,N}_s\;\int_U \hat a^{i,N}_s(u)  H^\varepsilon(X^i_{[s]^N}, m^N_s)(\mathrm{d}u)\;(\Sigma^{i,N}_s)^{\top}-\hat a^{\varepsilon,i,N}_s \Big\} \nonumber
    \\
    &~~~~~~~~~~~~~~~~~~~~~~~~~~~~~~~~~\nabla^2 f(X^{j}_{[s]^N})\Big\{ \Sigma^{j,N}_s\;\int_U \hat a^{j,N}_s(u)  H^\varepsilon(X^j_{[s]^N}, m^N_s)(\mathrm{d}u)\;(\Sigma^{j,N}_s)^{\top}-\hat a^{\varepsilon,j,N}_s  \Big\}  \Bigg] \nonumber
    \\
    &=\E^{\widehat{\P}} \Bigg[ \nabla^2 f(X^{i}_{[s]^N})\Big\{\hat a^{\varepsilon,i,N}_s-\hat a^{\varepsilon,i,N}_s \Big\}\nabla^2 f(X^{j}_{[s]^N})\Big\{\hat a^{\varepsilon,j,N}_s-\hat a^{\varepsilon,j,N}_s  \Big\}  \Bigg]=0.
\end{align}

By simple calculations,
\begin{align*}
    &\langle f,\vartheta^{N}_t \rangle - \langle f,\vartheta^{N}_0 \rangle - \int_0^t \int_{\R^n} \Ac^{\varepsilon}_r f\big[B^N,\phi(\mu^N),\zeta^N,\overline{\Lambda}^N_r\big](x) \vartheta^{N}_{[r]^N}(\mathrm{d}x) \mathrm{d}r
    \\
    &=
    \frac{1}{N} \sum_{i=1}^N\;\;\int_0^t\nabla f(X^{i}_{[s]^N})\widehat \Sigma^{i,N}_s\big(N^\varepsilon(X^i_{[s]^N},m^N_s)(V^{i,N}_s) \big) \mathrm{d}W^i_s
    +
    \int_0^t \Bigg[ \nabla f(X^{i}_{[s]^N})\Big\{\widehat B^{i,N}_s\big(N^\varepsilon(X^i_{[s]^N},m^N_s)(V^{i,N}_s) \big)-\hat b^{\varepsilon,i,N}_s \Big\}
    \\
    &+\frac{1}{2} \mathrm{Tr}\Big[\Big\{\widehat A^{i,N}_s \big(N^\varepsilon(X^i_{[s]^N},m^N_s)(V^{i,N}_s) \big)-\hat a^{\varepsilon,i,N}_s \Big\}\nabla^2 f(s,X^{i}_{[s]^N}) \Big]\;+
    \big[\nabla f(X^{i}_{s})-\nabla f(X^{i}_{[s]^N}) \big]\widehat B^{i,N}_s\big(N^\varepsilon(X^i_{[s]^N},m^N_s)(V^{i,N}_s) \big)
    \\
    &~~~~~~~~~~~~~~~~~~~~~~~~~~+\frac{1}{2} \mathrm{Tr}\Big[\widehat A^{i,N}_s\big(N^\varepsilon(X^i_{[s]^N},m^N_s)(V^{i,N}_s) \big)  \big[\nabla^2 f(X^{i}_{s})-\nabla^2 f(X^{i}_{[s]^N}) \big] \Big] \Bigg] \;\mathrm{d}s,
\end{align*}

consequently, there exists a constant $C>0$ (independent of $N$) such that
\begin{align*}
    &\E^{\widehat{\P}} \Bigg[ \Big| \langle f,\vartheta^{N}_t \rangle - \langle f,\vartheta^{N}_0 \rangle - \int_0^t \int_{\R^n} \Ac^{\varepsilon}_r f\big[B^N,\phi(\mu^N),\zeta^N,\overline{\Lambda}^N_r\big](x) \vartheta^{N}_{[r]^N}(\mathrm{d}x) \mathrm{d}r \Big|^2 \Bigg]
    \\
    &\le C~ \Bigg( \E^{\widehat{\P}} \Bigg[ \Big|\frac{1}{N} \sum_{i=1}^N\int_0^t \nabla f(X^{i}_{[s]^N})\widehat \Sigma^{i,N}_s\big(N^\varepsilon(X^i_{[s]^N},m^N_s)(V^{i,N}_s) \big) \mathrm{d}W^i_s\Big|^2 \Bigg]
    \\
    &~~~~~~~~~~~~~~~~~~~~~~~~~~~~~~~~+\int_0^t \E^{\widehat{\P}} \Bigg[ \Big|\frac{1}{N} \sum_{i=1}^N \nabla f(X^{i}_{[s]^N})\Big\{\widehat B^{i,N}_s\big(N^\varepsilon(X^i_{[s]^N},m^N_s)(V^{i,N}_s) \big)-\hat b^{\varepsilon,i,N}_s \Big\} \Big|^2 \Bigg]\mathrm{d}s
    \\
    &~~+ \int_0^t \E^{\widehat{\P}} \Bigg[ \Big| \frac{1}{N} \sum_{i=1}^N\frac{1}{2} \mathrm{Tr}\Big[\Big\{\widehat A^{i,N}_s \big(N^\varepsilon(X^i_{[s]^N},m^N_s)(V^{i,N}_s) \big)-\hat a^{\varepsilon,i,N}_s \Big\}\nabla^2 f(X^{i}_{[s]^N}) \Big] \Big|^2 \Bigg] \mathrm{d}s
    \\
    &~~+
    \int_0^t\frac{1}{N} \sum_{i=1}^N\E^{\widehat{\P}} \Bigg[ \Big|\big[\nabla f(X^{i}_{s})-\nabla f(X^{i}_{[s]^N}) \big]\Big|^2 + \Big|\frac{1}{2}  \big[\nabla^2 f(X^{i}_{s})-\nabla^2 f(X^{i}_{[s]^N}) \big] \Big|^2  \Bigg]\mathrm{d}s \Bigg).
\end{align*}

By successively applying the results \eqref{eq:vanish-productB} and \eqref{eq:vanish-productA},  and inequality \eqref{ineq:regularity-trajectory}, one gets a constant $M>0$ depending on $(f,b,\sigma)$ (which changes from line to line) s.t.
\begin{align} \label{inequality:majorationFP-app}
    &\E^{\widehat{\P}} \Bigg[ \Big| \langle f,\vartheta^{N}_t \rangle - \langle f,\vartheta^{N}_0 \rangle - \int_0^t \int_{\R^n} \Ac^{\varepsilon}_r f\big[B^N,\phi(\mu^N),\zeta^N,\overline{\Lambda}^N_r\big](x) \vartheta^{N}_{[r]^N}(\mathrm{d}x)  \mathrm{d}r \Big|^2 \Bigg] \nonumber
    \\
    &\le M
    \Bigg( \E^{\widehat{\P}} \Bigg[ \Bigg|  \frac{1}{N} \sum_{i=1}^N\int_0^t\nabla f(X^{i}_s) \widehat \Sigma^{i,N}_s\big(N^\varepsilon(X^i_{[s]^N},m^N_s)(V^{i,N}_s) \big) \mathrm{d}W^i_s \Bigg|^2 \Bigg] + \frac{1}{2^N} + \frac{1}{N} \Bigg) \nonumber
    \\
    &\le M
    \Bigg( \frac{1}{N^2} \sum_{i=1}^N\E^{\widehat{\P}} \Bigg[   \int_0^t \Big|\nabla f(X^{i}_s) \widehat \Sigma^{i,N}_s\big(N^\varepsilon(X^i_{[s]^N},m^N_s)(V^{i,N}_s) \big) \Big|^2 \mathrm{d}s  \Bigg] + \frac{1}{2^N} + \frac{1}{N} \Bigg).
\end{align}

Remark that as $\nabla f$ and $\widehat \Sigma$ are bounded, 
\begin{align} \label{inequality:majoration_covariation}
    \frac{1}{N^2} \sum_{i=1}^N\E^{\widehat{\P}} \Bigg[   \int_0^t \Big|\nabla f(X^{i}_s) \widehat \Sigma^{i,N}_s\big(N^\varepsilon(X^i_{[s]^N},m^N_s)(V^{i,N}_s) \big) \Big|^2 \mathrm{d}s  \Bigg] \le M \frac{1}{N}.
\end{align}


\medskip
Thanks to inequality \eqref{ineq:regularity-trajectory}, it is straightforward to verify that
\begin{align} \label{resultat:delay}
    \Lim_{N \to \infty}\Wc_p \Big(\Lc^{\widehat{\P}} \big(\vartheta^{N},\vartheta^{N},\phi(\mu^N),\zeta^N,\overline{\Lambda}^{N},B^N \big), \Lc^{\widehat{\P}} \big(\vartheta^{N},(\vartheta^{N}_{[t]^N})_{t \in [0,T]},\phi(\mu^N),\zeta^N,\overline{\Lambda}^{N},B^N \big) \Big)=0.
\end{align}

Let $\mathrm{P}^{\infty} \in \Pc \big(\Cc^n_{\Wc} \x \Cc^n_{\Wc} \x \Cc^n_{\Wc} \x \M\big((\Pc^n_U)^2\big) \x \Cc^\ell \big)$ be the limit of any sub--sequence $(\mathrm{P}^{N_k})_{k \in \N^*}$ of $(\mathrm{P}^{N})_{N \in \N^*},$ and denote by $(\beta^\vartheta,\beta^{\mu},\beta^\zeta,\overline{\beta},B)$ the canonical process on $\Cc^n_{\Wc} \x \Cc^n_{\Wc} \x \Cc^n_{\Wc} \x \M\big((\Pc^n_U)^2 \big) \x \Cc^\ell.$ By combining inequalities \eqref{inequality:majorationFP-app} and \eqref{inequality:majoration_covariation} with the result \eqref{resultat:delay},
by passing to the limit, using continuity of coefficients, given $\varepsilon > 0$: for all $(t,f) \in [0,T] \x C^{\infty}_b(\R^n)$
\begin{align*}
    &\Lim_k \E^{\widehat{\P}} \Bigg[ \bigg|\langle f,\vartheta^{N_k}_t \rangle
    -
    \langle f,\vartheta^{N_k}_0 \rangle
    -
    \int_0^t \int_{\R^n} \Ac^{\varepsilon}_r f\big[B^{N_k},\phi(\mu^{N_k}),\zeta^{N_k},\overline{\Lambda}^{N_k}_r\big](x) \vartheta^{N_k}_{[r]^{N_k}}(\mathrm{d}x)
     \mathrm{d}r \bigg|^2 \Bigg]
    \\
    &~~~~=\E^{\mathrm{P}^{\infty}} \Bigg[ \bigg|\langle f,\beta^\vartheta_t \rangle
    -
    \langle f,\nu \rangle
    -
    \int_0^t \bigg[\int_{\R^n} \Ac^{\varepsilon}_r f\big[B,\beta^\mu,\beta^\zeta,\overline{\beta}_r\big](x) \beta^\vartheta_{r}(\mathrm{d}x)
    \bigg] \mathrm{d}r \bigg|^2 \Bigg]=0.
\end{align*}
Therefore, after taking a countable family of $(f,t),$ one gets: for all $(t,f) \in [0,T] \x C^{\infty}_b(\R^n)$
\begin{align*}
    &\langle f,\beta^\vartheta_t \rangle
    =
    \langle f,\nu \rangle
    +
    \int_0^t \int_{\R^n} \Ac^{\varepsilon}_r f\big[B,\beta^\mu,\beta^\zeta,\overline{\beta}_r\big](x) \beta^\vartheta_{r}(\mathrm{d}x) \mathrm{d}r,\;\mbox{for all}\;t \in [0,T],\;\mathrm{P}^{\infty}\mbox{--a.e.}
\end{align*}
from this equality, we can show the previous equality holds true for all  $f \in C^{2}_b(\R^n)$. For each $\varepsilon>0,$ by uniqueness $\beta^\vartheta:=\Phi^{\varepsilon} \big(B,\beta^\mu,\beta^\zeta,\overline{\beta} \big)$ with $\Phi^\varepsilon: \Cc^\ell \x \Cc^n_{\Wc} \x \Cc^n_{\Wc} \x \M\big((\Pc^n_U)^2 \big) \to \Cc^n_{\Wc}$ a Borel function used in \eqref{eq:function_FK-PL}. Notice that, by assumptions \eqref{eq:general-condition},
\begin{align*}
    \mathrm{P}^{\infty} \circ \big( \beta^\mu,\beta^\zeta,\overline{\beta},B \big)^{-1}
    =
    \Lim_k \widehat{\P} \circ \big(\phi(\mu^{N_k}),\zeta^{N_k},\overline{\Lambda}^{N_k},B^{N_k} \big)^{-1}
    =
    \Q \circ \big(\phi(\mu),\zeta,\overline{\Lambda},B \big)^{-1}\;\mbox{in}\;\Wc_p.
\end{align*}
This result is enough to deduce that $\mathrm{P}^{\infty}=\Q \circ \big(\mu^\varepsilon,\phi(\mu),\zeta,\overline{\Lambda},B \big)^{-1}.$ This is true for any limit $\mathrm{P}^{\infty}$ for any sub--sequence of $(\mathrm{P}^{N})_{N \in \N^*},$ therefore
\begin{align} \label{equation-identification_limit}
    \Lim_{N \to \infty} \widehat{\P} \circ \big(\vartheta^{N},\phi(\mu^{N}),\zeta^{N},\overline{\Lambda}^{N},B^{N} \big)^{-1}
    =
    \Q\circ \big( \mu^\varepsilon,\phi(\mu),\zeta,\overline{\Lambda},B \big)^{-1}\;\mbox{in}\;\Wc_p.
\end{align}

\medskip
$\mathbf{\underline{Step\;3:Last\;approximation}}$:
To finish, now, let us define $\widehat{X}^{\varepsilon,i,N}:=\widehat{X}^i$ the strong solution of
\begin{align*}
    \widehat{X}^i_t
    =
    \xi^i
    &+
    \int_0^t \hat b\big(r,\widehat{X}^i_r,B^N,\phi(\widehat{\mu}^N),\zeta^N, \widehat{m}^N_r,\nub^N_r,\alpha^{i}_{r}\big) \mathrm{d}r
    +
    \int_0^t \hat \sigma \big(r,\widehat{X}^i_r,B^N,\phi(\widehat{\mu}^N),\zeta^N, \widehat{m}^N_r,\nub^N_r,\alpha^{i}_{r}\big)
    \mathrm{d}W^i_r,
    ~
    \mbox{for all}
    ~
    t \in [0,T]
\end{align*}
where
$$
    \alpha^{i}_{t}
    :=
    N^\varepsilon(X^i_{t^N_{k}},m^N_t)(V^{i,N}_t)
    ~
    \mbox{for all}
    ~
    t \in [t^N_k,t^N_{k+1}[
    ,
    ~
    \widehat{m}^N_t(\mathrm{d}x,\mathrm{d}u)
    :=
    \frac{1}{N} \sum_{i=1}^N \delta_{(\widehat{X}^i_t,\;\alpha^i_t)}(\mathrm{d}x,\mathrm{d}u)
    ~
    \mbox{and}~\widehat{\mu}^N_t(\mathrm{d}x):=\widehat{m}^N_t(\mathrm{d}x,U),
$$
recall that $(X^1,...,X^N)$ are defined in \eqref{eq:auxillary-appr}.
It is straightforward to check that: there exists a constant $\mathrm{D}>0$ (independent of $\varepsilon$ and $N$)
\begin{align} \label{ineq:regularity-trajectory2}
    \sup_{i \in \{1,...,N\}}
    \E^{\widehat{\P}} 
    \Big[
    |\widehat{X}^{i}_t
    -
    \widehat{X}^{i}_s
    \big|^{p}
    \Big]
    \le
    \mathrm{D}|t-s|,\;\mbox{for all}\;(t,s) \in [0,T] \x [0,T].
\end{align}
By Bukholder--Davis--Gundy inequality, lipschitz property of coefficients and previous inequality \eqref{ineq:regularity-trajectory2}, 
\begin{align*}
    &\E^{\widehat{\P}} 
    \bigg[
    \sup_{s \in [0,t]} |\widehat{X}^i_s-X_s^{i}|^{p}
    \bigg]
    \\
    &\le
    \hat D \Bigg(
    \E^{\widehat{\P}}\bigg[\int_0^t
    \Big|\hat b\big(r,\widehat{X}^i_r,B^N,\phi(\widehat{\mu}^N),\zeta^N, \widehat{m}^N_r,\nub^N_r,\alpha^{i}_{r}\big)
    -
    \widehat B \Big(r,X^i_{[r]^N},B^N,\phi(\mu^N),\zeta^N,m^N_r,\nub^{N}_r, N^\varepsilon(X^i_{[r]^N},m^N_r)(V^{i,N}_r) \Big)\Big|^{p} \mathrm{d}r \bigg]
    \\
    &~~~+
    \E^{\widehat{\P}}\bigg[\int_0^t
    \Big| \hat \sigma \big(r,\widehat{X}^i_r,B^N,\phi(\widehat{\mu}^N),\zeta^N, \widehat{m}^N_r,\nub^N_r,\alpha^{i}_{r}\big)
    -
    \widehat \Sigma \Big(r,X^i_{[r]^N},B^N,\phi(\mu^N),\zeta^N,m^N_r,\nub^{N}_r, N^\varepsilon(X^i_{[r]^N},m^N_r)(V^{i,N}_r) \Big)\Big|^{p} \mathrm{d}r \bigg] \Bigg)
    \\
    &\le
    \hat D \Bigg(
    \E^{\widehat{\P}}\bigg[\int_0^t
    \Big|\hat b\big(r,\widehat{X}^i_r,B^N,\phi(\widehat{\mu}^N),\zeta^N, \widehat{m}^N_r,\nub^N_r,\alpha^{i}_{r}\big)
    -
    \hat b(r,X^i_{[r]^N},B^N,\phi(\mu^N),\zeta^N,m^N_r,\nub^N_r,\alpha^{i}_{r})\Big|^{p} \bigg]
    \\
    &~~+\E^{\widehat{\P}}\bigg[\int_0^t\Big| b(r,X^i_{[r]^N},B^N,\phi(\mu^N),\zeta^N,m^N_r,\nub^N_r,\alpha^{i}_{r})-
    \widehat B \Big(r,X^i_{[r]^N},B^N,\phi(\mu^N),\zeta^N,m^N_r,\nub^{N}_r, N^\varepsilon(X^i_{[r]^N},m^N_r)(V^{i,N}_r) \Big)\Big|^{p} \mathrm{d}r \bigg]
    \\
    &~~+
    \E^{\widehat{\P}}\bigg[\int_0^t
    \Big|\hat \sigma \big(r,\widehat{X}^i_r,B^N,\phi(\widehat{\mu}^N),\zeta^N, \widehat{m}^N_r,\nub^N_r,\alpha^{i}_{r}\big)
    -
    \hat \sigma(r,X^i_{[r]^N},B^N,\phi(\mu^N),\zeta^N,m^N_r,\nub^N_r,\alpha^{i}_{r})\Big|^{p} \mathrm{d}r \bigg]
    \\
    &~~+
    \E^{\widehat{\P}}\bigg[\int_0^t
    \Big|\hat \sigma(r,X^i_{[r]^N},B^N,\phi(\mu^N),\zeta^N,m^N_r,\nub^N_r,\alpha^{i}_{r})
    -
    \widehat \Sigma \Big(r,X^i_{[r]^N},B^N,\phi(\mu^N),\zeta^N,m^N_r,\nub^{N}_r, N^\varepsilon(X^i_{[r]^N},m^N_r)(V^{i,N}_r) \Big)\Big|^{p} \mathrm{d}r \bigg] \Bigg)
    \\
    &\le
    \hat D \Bigg(
    \E^{\widehat{\P}} \bigg[
    \int_0^t \Big|
    \hat b^{\varepsilon}\big[B^N,\phi(\mu^N),\zeta^N,\overline{\Lambda}^{N}_r \big](r,X^i_{[r]^N})
    -
    \int_U \hat b\big(r,X^i_{[r]^N},B^N,\phi(\mu^N),\zeta^N,m^N_r,\nub^N_r,u\big) H^\varepsilon(m^N,X^i_{[s]^N})(\mathrm{d}u)\Big|^{p}
    \mathrm{d}r
    \bigg]
    \\
    &~~~~+
    \E^{\widehat{\P}} \bigg[ \int_0^t \Big| 1- \Sigma\Big(r,X^i_{[r]^N},B^N,\phi(\mu^N),\zeta^N,m^N_r,\nub^{N}_r \Big) \Big|^{p} \mathrm{d}r
    +
    \int_0^t \Wc_p \Big(\frac{1}{N} \sum_{i=1}^N \delta_{\big(X^i_{[r]^N},\;\alpha^i_r \big)}(\mathrm{d}x,\mathrm{d}u),m^N_r(\mathrm{d}x,\mathrm{d}u)\Big)^{p}
    \mathrm{d}r
    \bigg]
    \\
    &~~~~~~~~+
    \E^{\widehat{\P}}\bigg[ \sup_{e' \in [0,T]} \Wc_p(\phi_{e'}(\vartheta^{N}),\phi_{e'}(\mu^N)) +
    \int_0^t \sup_{e \in [0,r]}\big|\widehat{X}^i_e-X^i_e\big|^{p} \mathrm{d}r \bigg] + \frac{1}{2^N}
    \Bigg),
\end{align*}
then by Gronwall lemma
\begin{align*}
    \E^{\widehat{\P}} 
    \bigg[
    \sup_{t \in [0,T]} |\widehat{X}^i_t-X^i_t|^{p}
    \bigg]
    \le
    \hat D \bigg( \E^{\widehat{\P}}\bigg[
    \sup_{e' \in [0,T]} \Wc_p\Big(\phi_{e'}(\vartheta^{N}),\phi_{e'}(\mu^N) \bigg) \bigg] 
    +
    \frac{1}{2^N} 
    +
    E^{\varepsilon,i,N} + C^{\varepsilon,N} \bigg) 
\end{align*}
where $C^{\varepsilon,N}:=\E^{\widehat{\P}} \Big[\int_0^T \Wc_p \Big(\frac{1}{N} \sum_{i=1}^N \delta_{\big(X^i_{[r]^N},\;\alpha^i_r \big)}(\mathrm{d}x,\mathrm{d}u),m^N_r(\mathrm{d}x,\mathrm{d}u)\Big)^{p}\mathrm{d}r \Big],$ and
\begin{align*}
    &E^{\varepsilon,i,N}
    \\
    &:=\E^{\widehat{\P}} \Bigg[
    \int_0^T \Big|
    \big[\hat b^{\varepsilon},\hat a^{\varepsilon} \big]\big[B^N,\phi(\mu^N),\zeta^N,\overline{\Lambda}^{N}_r \big](r,X^i_{[r]^N})
    -
    \int_U \big[\hat b,\hat a \big]\big(r,X^i_{[r]^N},B^N,\phi(\mu^N),\zeta^N,m^N_r,\nub^N_r,u\big) H^\varepsilon(X^i_{[r]^N},m^N_r)(\mathrm{d}u)\Big|^{p}
    \mathrm{d}r
    \Bigg].
\end{align*}
Firstly, thanks to results \eqref{equation-identification_limit} and the approximation realized in \eqref{eq:first_regularized_FP}, one gets
\begin{align} \label{eq:convergence_mu}
    \Lim_{\varepsilon \to 0} \Lim_{N \to \infty}\E^{\widehat{\P}}\bigg[
    \sup_{e' \in [0,T]} \Wc_p \Big(\phi_{e'}(\vartheta^{N}),\phi_{e'}(\mu^N) \Big) \bigg] 
    =
    \Lim_{\varepsilon \to 0} \E^{\Q}\bigg[
    \sup_{e' \in [0,T]} \Wc_p\Big(\phi_{e'}(\mu^{\varepsilon}),\phi_{e'}(\mu) \Big) \bigg]
    =
    0.
\end{align}
Secondly, after calculations, it is straightforward to deduce that
\begin{align*}
    \frac{1}{N} \sum_{i=1}^N E^{\varepsilon,i,N}
    &=
    \E^{\widehat{\P}} \Bigg[
    \int_0^T \int_{\R^n} \int_{(\Pc^n_U)^2} \Big|
    \int_{U \x \R^n} \big[\hat b,\hat a \big]\big(r,y,B^N,\phi(\mu^N),\zeta^N,m,\nub,u\big) \frac{G_\varepsilon(x-y)}{(m(U,\mathrm{d}z))^{(\varepsilon)}(x)} m(\mathrm{d}u,\mathrm{d}y)
    \\
    &~~~~~~~~-
    \int_{U \x \R^n} \big[\hat b,\hat a \big]\big(r,x,B^N,\phi(\mu^N),\zeta^N,m,\nub,u\big) \frac{G_\varepsilon(x-y)}{(m(U,\mathrm{d}z))^{(\varepsilon)}(x)} m(\mathrm{d}u,\mathrm{d}y)\Big|^{p} \overline{\Lambda}^{N}_r(\mathrm{d}m,\mathrm{d}\nub)\vartheta^{N}_{[r]^N}(\mathrm{d}x)
    \mathrm{d}r
    \Bigg].
\end{align*}
By regularity of coefficients (Assumption \ref{assum:main1} and $(\hat b,\hat \sigma)$ bounded), the results \eqref{equation-identification_limit} and \eqref{resultat:delay} allow to get
\begin{align*}
    \Lim_{N \to \infty}\frac{1}{N} \sum_{i=1}^N E^{\varepsilon,i,N}
    &\le
    \E^{\Q} \Bigg[
    \int_0^T \int_{\R^n} \int_{(\Pc^n_U)^2} \int_{U \x \R^n} \Big|
    \big[\hat b, \hat a \big]\big(r,y,B,\phi(\mu),\zeta,m,\nub,u\big)
    \\
    &~~~~~~~~~~~~~~~~~~-
    \big[\hat b, \hat a \big]\big(r,x,B,\phi(\mu),\zeta,m,\nub,u\big) \Big|^{p} \frac{G_\varepsilon(x-y)}{(m(U,\mathrm{d}z))^{(\varepsilon)}(x)} m(\mathrm{d}u,\mathrm{d}y) \overline{\Lambda}_r(\mathrm{d}m,\mathrm{d}\nub)\mu^\varepsilon_{r}(\mathrm{d}x)
    \mathrm{d}r
    \Bigg],
\end{align*}
then, by \Cref{lemm:appr_coef}, $\Lim_{\varepsilon \to 0}\Lim_{N \to \infty}\frac{1}{N} \sum_{i=1}^N E^{\varepsilon,i,N}=0.$  

\medskip
Next, let us define the variable
\begin{align*}
    \Upsilon^{N}_r(\mathrm{d}e',\mathrm{d}e)\mathrm{d}r
    :=
    \E^{\widehat{\P}} \bigg[\delta_{\big(\overline{m}^{N}_r, m^N_r \big)}(\mathrm{d}e',\mathrm{d}e)\mathrm{d}r \bigg] \in \M\big( (\Pc^n_U)^2\big),\;\mbox{where}\;\overline{m}^{N}_r(\mathrm{d}x,\mathrm{d}u):=\frac{1}{N} \sum_{i=1}^N \delta_{\big(X^i_{[r]^N},\;\alpha^i_r \big)}(\mathrm{d}x,\mathrm{d}u),
\end{align*}
It is easy to check that the sequence $(\Upsilon^{N})_{N \in \N^*}$ is relatively compact for the Wasserstein metric $\Wc_p.$ Denote by $\Upsilon^\infty$ the limit of a sub--sequence $(\Upsilon^{N_k})_{k \in \N^*}.$
Let $Q \in \N^*,$  $(f^q)_{q \in \{1,...Q\}}:\R^n \x U \to \R^Q$ be bounded continuous functions and $g: [0,T] \x \Pc^n_U \to \R.$ One has
\begin{align*}
    &~\int_0^T \int_{(\Pc^n_U)^2} \Prod_{q=1}^Q\langle f^q,e' \rangle g(t,e) \Upsilon^\infty_t(\mathrm{d}e',\mathrm{d}e)\mathrm{d}t
    =
    \E^{\Q} \bigg[\int_0^T \int_{\Pc^n_U} \Prod_{q=1}^Q\int_{\R^n} \langle f^q(x,\cdot ),H^{\varepsilon}(x,m) \rangle \mu^{\varepsilon}_t(\mathrm{d}x) g(t,m)\overline{\Lambda}_t(\mathrm{d}m,\Pc^n_U)\mathrm{d}t \bigg].
\end{align*}
We prove this equality when $Q=2,$ the case $Q \in \N^*$ follows immediately. Indeed,
\begin{align*}
    &\int_0^T \int_{(\Pc^n_U)^2} \langle f^1,e' \rangle \langle f^2,e \rangle g(t,e) \Upsilon^\infty_t(\mathrm{d}e',\mathrm{d}e)\mathrm{d}t
    =
    \Lim_k \frac{1}{N_k}\frac{1}{N_k} \sum_{i,j=1}^{N_k}\E^{\widehat{\P}} \bigg[\int_0^T f^1\big(X^{i}_{[t]^{N_k}},\alpha^i_t \big) f^2\big(X^{j}_{[t]^{N_k}},\alpha^j_t \big) g(t,m^{N_k}_t)\mathrm{d}t \bigg]
    \\
    &=
    \Lim_k \Bigg( \frac{1}{N_k}\frac{1}{N_k} \sum_{i \neq j}\E^{\widehat{\P}} \bigg[\int_0^T \int_U f^1\Big(X^{i}_{[t]^{N_k}},u \Big) H^{\varepsilon}(X^{i}_{[t]^{N_k}},m^{N_k}_t)(\mathrm{d}u) \int_U f^2\Big(X^{j}_{[t]^{N_k}},u \Big) H^{\varepsilon}(X^{i}_{[t]^{N_k}},m^{N_k}_t)(\mathrm{d}u) g(t,m^{N_k}_t)\mathrm{d}t \bigg]
    \\
    &~~~~+\frac{1}{N_k}\frac{1}{N_k} \sum_{i=1}^{N_k}\E^{\widehat{\P}} \bigg[\int_0^T f^1\Big(X^{i}_{[t]^{N_k}},N^\varepsilon(X^i_{[t]^{N_k}},m^{N_k}_t)(V^{i,N_k}_t) \Big) f^2\Big(X^{j}_{[t]^{N_k}},N^\varepsilon(X^i_{[t]^{N_k}},m^{N_k}_t)(V^{i,N_k}_t) \Big) g(t,m^{N_k}_t)\mathrm{d}t \bigg] \Bigg)
    \\
    &=
    \Lim_k \Bigg( \E^{\widehat{\P}} \bigg[\int_0^T \int_{\R^n} \int_U f^1(x,u ) H^{\varepsilon}(x,m^{N_k}_t)(\mathrm{d}u) \vartheta^{N_k}_{[t]^{N_k}}(\mathrm{d}x) \int_{\R^n} \int_U f^2(y,u ) H^{\varepsilon}(y,m^{N_k}_t)(\mathrm{d}u)\vartheta^{N_k}_{[t]^{N_k}}(\mathrm{d}y) g(t,m^{N_k}_t)\mathrm{d}t \bigg]
    \\
    &~~~~-\frac{1}{N_k}\frac{1}{N_k} \sum_{i=1}\E^{\widehat{\P}} \bigg[\int_0^T \int_U f^1\Big(X^{i}_{[t]^{N_k}},u \Big) H^{\varepsilon}(X^{i}_{[t]^{N_k}},m^{N_k}_t)(\mathrm{d}u) \int_U f^2\Big(X^{i}_{[t]^{N_k}},u \Big) H^{\varepsilon}(X^{i}_{[t]^{N_k}},m^{N_k}_t)(\mathrm{d}u) g(t,m^{N_k}_t)\mathrm{d}t \bigg]
    \\
    &~~~~+\frac{1}{N_k}\frac{1}{N_k} \sum_{i=1}^{N_k}\E^{\widehat{\P}} \bigg[\int_0^T f^1\Big(X^{i}_{[t]^{N_k}},N^\varepsilon(X^i_{[t]^{N_k}},m^{N_k}_t)(V^{i,N_k}_t) \Big) f^2\Big(X^{i}_{[t]^{N_k}},N^\varepsilon(X^i_{[t]^{N_k}},m^{N_k}_t)(V^{i,N_k}_t) \Big) g(t,m^{N_k}_t)\mathrm{d}t \bigg] \Bigg)
    \\
    &=
    \E^{\Q} \bigg[\int_0^T \int_{\Pc^n_U} \int_{\R^n} \int_U f^1(x,u ) H^{\varepsilon}(x,m)(\mathrm{d}u) \mu^{\varepsilon}_t(\mathrm{d}x) \int_{\R^n} \int_U f^2(y,u ) H^{\varepsilon}(y,m)(\mathrm{d}u)\mu^{\varepsilon}_t(\mathrm{d}y) g(t,m)\overline{\Lambda}_t(\mathrm{d}m,\Pc^n_U)\mathrm{d}t \bigg],
\end{align*}
where the fourth equality is true because of the same argument used in \eqref{eq:vanish-productB} and \eqref{eq:vanish-productA}, i.e. for all $(s,v) \in (t^{N_l}_{k},t^{N_l}_{k+1}) \x  \{1,...,N_l\},$ $\widehat{\P} \circ \big(N^\varepsilon(x,m)(V^{v,N_l}_s) \big)^{-1}(\mathrm{d}u)=H^\varepsilon(x,m)(\mathrm{d}u),$ and for $i \neq j$ $(V^i_s,V^j_s)$ are independent and independent of other variables, and the last equality follows from \eqref{equation-identification_limit} and \eqref{resultat:delay}, and the terms starting with $\frac{1}{(N_l)^2} \sum_{i=1}^{N_l}$ go to zero because $(f^1,f^2,g)$ are bounded.
Hence,
\begin{align*}
    \Upsilon^\infty_t(\mathrm{d}e',\mathrm{d}e)\mathrm{d}t
    =
    \widehat{\Upsilon}_t(\mathrm{d}e',\mathrm{d}e)\mathrm{d}t,\;\mbox{where}\;
    \widehat{\Upsilon}_t(\mathrm{d}e',\mathrm{d}e)\mathrm{d}t
    :=
    \E^{\Q}\bigg[\delta_{\big( H^{\varepsilon}(x,e)(\mathrm{d}u) \mu^{\varepsilon}_t(\mathrm{d}x)\big)}(\mathrm{d}e')\overline{\Lambda}_t(\mathrm{d}e,\Pc^n_U)\mathrm{d}t \bigg],
\end{align*}
this is true for any limit $\Upsilon^\infty$ of any sub--sequence. Therefore, the sequence $(\Upsilon^{N})_{N \in \N^*}$ converges towards $\widehat{\Upsilon}$ for the wasserstein metric $\Wc_p.$ Then, to finish, by \Cref{lemm:appr_coef},
\begin{align*}
    \Lim_{\varepsilon \to 0}\Lim_{N \to \infty} C^{\varepsilon,N}
    &= \Lim_{\varepsilon \to 0} \Lim_{N \to \infty} \E^{\widehat{\P}} \bigg[\int_0^T \Wc_p \Big(\frac{1}{N} \sum_{i=1}^N \delta_{\big(X^i_{[r]^N},\;\alpha^i_r \big)}(\mathrm{d}x,\mathrm{d}u),m^N_r\Big)^{p}
    \mathrm{d}r \bigg]
    \\
    &= \Lim_{\varepsilon \to 0}
    \E^{\Q} \bigg[\int_0^T \int_{\Pc^n_U}\Wc_p \Big(H^{\varepsilon}(x,m)(\mathrm{d}u) \mu^{\varepsilon}_t(\mathrm{d}x),m\Big)^{p}\overline{\Lambda}_t\big(\mathrm{d}m,\Pc^n_U\big)\mathrm{d}t \bigg]=0.
\end{align*}

All these results allow to deduce that  $\Lim_{\varepsilon \to 0} \Lim_{N \to \infty}\frac{1}{N} \sum_{i=1}^N\E^{\widehat{\P}} 
    \bigg[
    \sup_{t \in [0,T]} \big|\widehat{X}^{\varepsilon,i,N}_t-X^{\varepsilon,i,N}_t \big|^{p}
    \bigg]=0.$ As
\begin{align*}
    &\E^{\widehat{\P}} \bigg[\int_0^T  \Wc_p\big(\widehat{m}^N_t,m^N_t \big) ^p \mathrm{d}r \bigg]
    \\
    &\le
    \E^{\widehat{\P}} \bigg[\int_0^T  \Wc_p\big(\widehat{m}^N_t(\mathrm{d}x,\mathrm{d}u),\frac{1}{N} \sum_{i=1}^N \delta_{\big(X^i_{[t]^N},\;\alpha^i_t \big)}(\mathrm{d}x,\mathrm{d}u) \big)^p  \mathrm{d}r \bigg]
    +
    \E^{\widehat{\P}} \bigg[\int_0^T  \Wc_p\big(\frac{1}{N} \sum_{i=1}^N \delta_{\big(X^i_{[t]^N},\;\alpha^i_t \big)}(\mathrm{d}x,\mathrm{d}u),m^N_t(\mathrm{d}x,\mathrm{d}u) \big)^p  \mathrm{d}r \bigg]
    \\
    &\le
    \frac{1}{N} \sum_{i=1}^N\E^{\widehat{\P}} 
    \bigg[
    \int_0^T \big|\widehat{X}^{\varepsilon,i,N,K}_t-X^{\varepsilon,i,N}_{[t]^N} \big|^{p} \mathrm{d}t
    \bigg]
    +
    \E^{\widehat{\P}} \bigg[\int_0^T  \Wc_p\big(\frac{1}{N} \sum_{i=1}^N \delta_{\big(X^i_{[t]^N},\;\alpha^i_t \big)}(\mathrm{d}x,\mathrm{d}u),m^N_t(\mathrm{d}x,\mathrm{d}u) \big)^p  \mathrm{d}r \bigg]
    \\
    &\le
    \frac{1}{N} \sum_{i=1}^N\E^{\widehat{\P}} 
    \bigg[
    \int_0^T \big|\widehat{X}^{\varepsilon,i,N}_t-X^{\varepsilon,i,N}_{t} \big|^{p} \mathrm{d}t
    \bigg]
    +
    \frac{1}{2^N}
    +
    \E^{\widehat{\P}} \bigg[\int_0^T  \Wc_p\big(\frac{1}{N} \sum_{i=1}^N \delta_{\big(X^i_{[t]^N},\;\alpha^i_t \big)}(\mathrm{d}x,\mathrm{d}u),m^N_t(\mathrm{d}x,\mathrm{d}u) \big)^p  \mathrm{d}r \bigg]
\end{align*}
then $\Lim_{\varepsilon \to 0} \Lim_{N \to \infty} \E^{\widehat{\P}} \bigg[\int_0^T  \Wc_p\big(\widehat{m}^N_t,m^N_t \big) ^p \mathrm{d}r \bigg]=0,$
similarly, using \eqref{eq:convergence_mu},

\begin{align*}
    &\Lim_{\varepsilon \to 0} \Lim_{N \to \infty}\E^{\widehat{\P}}\bigg[
    \sup_{e' \in [0,T]} \Wc_p(\phi_{e'}(\widehat{\mu}^N),\phi_{e'}(\mu^N)) \bigg]
    \\
    &~~~~~~~~~~~~~~~~~~\le 
    \Lim_{\varepsilon \to 0} \Lim_{N \to \infty}\bigg( \E^{\widehat{\P}}\bigg[
    \sup_{e' \in [0,T]} \Wc_p(\phi_{e'}(\widehat{\mu}^N),\phi_{e'}(\vartheta^{N})) \bigg] + \E^{\widehat{\P}}\bigg[
    \sup_{e' \in [0,T]} \Wc_p(\phi_{e'}(\vartheta^{N}),\phi_{e'}(\mu^N)) \bigg] \bigg)
    \\
    &~~~~~~~~~~~~~~~~~~\le K 
    \Lim_{\varepsilon \to 0} \Lim_{N \to \infty}\bigg( \frac{1}{N} \sum_{i=1}^N\E^{\widehat{\P}} 
    \bigg[
    \sup_{t \in [0,T]} \big|\widehat{X}^{\varepsilon,i,N}_t-X^{\varepsilon,i,N}_{t} \big|^{p}
    \bigg]
    +
    \frac{1}{2^N} + \E^{\widehat{\P}}\bigg[
    \sup_{e' \in [0,T]} \Wc_p(\phi_{e'}(\vartheta^{N}),\phi_{e'}(\mu^N)) \bigg] \bigg)=0.
\end{align*}
All previous result combined with measurability property \eqref{eq:measurable-property} allowed to say $(\alpha^1,...,\alpha^N)$ and $(\widehat{X}^1,...,\widehat{X}^N)$ are the controls and the processes we are looking for.

\end{proof}


\medskip
In fact, in \Cref{prop:approximation-FP_BY_SDE-2}, instead of interaction processes of type \eqref{eq:general-Nagents}, it is possible to use a sequence of $weak$ McKean--Vlasov processes and obtain similar result.   
Let us assume conditions and inputs previously mentioned for \Cref{prop:approximation-FP_BY_SDE-2} are satisfied. Let $W$ be a $(\widehat{\P},\widehat{\F})$--Brownian motion, $\xi$ be a $\widehat{\Fc}_0$--random variable with $\Lc^{\widehat{\P}}(\xi)=\nu$, and $Z$ be a uniform variable independent of $(\xi,W).$ In addition,
    \begin{align} \label{eq:assum-independence}
        \big(\psi(\mu^N),\zeta^N,\overline{\Lambda}^{N}, B^N \big)_{N \in \N^*}\;\;\mbox{are}\;\;\widehat{\P}\mbox{--independent of}\;\big(W, \xi, Z \big).
    \end{align}

For each $N \in \N^*,$ define the filtrations $\widehat{\F}^{N}:=(\widehat{\Fc}^{N}_t)_{t \in [0,T]}$ and $\widehat{\G}:=(\widehat{\Gc}^N_t)_{t \in [0,T]}$ by
    \begin{align*}
        \widehat{\Fc}^{N}_t
        :=
        \sigma 
        \Big{\{}
            \xi, \overline{\Lambda}^{N}_{t \wedge \cdot},\phi_{t \wedge \cdot}(\mu^N),\zeta^N_{t \wedge \cdot},W_{t \wedge \cdot}, B^N_{t \wedge \cdot},Z
        \Big{\}}\;\;\mbox{and}\;\;
        \widehat{\Gc}^N_t
        :=
        \sigma \big\{\psi_{t \wedge \cdot}(\mu^N),\zeta^N_{t \wedge \cdot},\overline{\Lambda}^{N}_{t \wedge \cdot}, B^N_{t \wedge \cdot} \big\},\;\mbox{for all }\;t \in [0,T].
    \end{align*}
    $\widehat{\G}$ will play the role of the common noise filtration.
We now provide  approximations by $weak$ McKean--Vlasov processes. The proofs of the next \Cref{prop:weak-appr} and \Cref{prop:approximation-FP_BY_SDE-1} are left in Appendix \ref{proof:prop:weak-appr}.    
\begin{proposition} \label{prop:weak-appr}
There exists a sequence of processes $(\alpha^N)_{N \in \N^*}$ satisfying: for each $N \in \N^*,$ $\alpha^N$ is $\widehat{\F}^{N}$--predictable, such that if $X^N$ is the unique strong solution of: $\E^{\widehat{\P}}[\|X^N\|^{p'}]< \infty,$ for all $t \in [0,T],$
\begin{align} \label{eq:general-weak}
    X^N_t
    =
    \xi
    &+
    \int_0^t \hat b\big(r,X^N_r,B^N,\phi(\widehat{\mu}^N),\zeta^N, \widehat{m}^N_r,\nub^N_r,\alpha^{N}_{r}\big) \mathrm{d}r
    +
    \int_0^t \hat \sigma \big(r,X^N_r,B^N,\phi(\widehat{\mu}^N),\zeta^N, \widehat{m}^N_r,\nub^N_r,\alpha^{N}_{r}\big)
    \mathrm{d}W_r,\;\widehat{\P} \mbox{--a.e.},
\end{align}
    where $\;\widehat{m}^N_t:=\Lc^{\widehat{\P}}\big(X^N_t,\alpha^{N}_t\big|\widehat{\Gc}^N_t\big)\;\mbox{and}\;\widehat{\mu}^N_t:=\Lc^{\widehat{\P}}\big(X^N_t\big|\widehat{\Gc}^N_t\big),$ then for the sub--sequence $(N_k)_{k \in \N^*}$ given in {\rm Proposition \ref{prop:approximation-FP_BY_SDE-2}},
    \begin{align*}
        \Lim_{k \to \infty} \E^{\widehat{\P}} \bigg[\int_0^T \Wc_p \big(\widehat{m}^{N_k}_t,m^{N_k}_t \big)^p \mathrm{d}t + \sup_{t \in [0,T]} \Wc_p \Big(\phi_t(\widehat{\mu}^{N_k}),\phi_t(\mu^{N_k}) \Big) \bigg]=0,
    \end{align*}
    and if $\widehat{\Lambda}_s(\mathrm{d}m,\mathrm{d}\nub)\mathrm{d}s:=\delta_{(\hat{m}^{N_k}_s,\nub^{N_k}_s)}(\mathrm{d}m,\mathrm{d}\nub)\mathrm{d}s,$
    \begin{align} \label{eq:cv_result_2}
        \Lim_{k \to \infty} \Lc^{\widehat{\P}} \Big(\widehat{\mu}^{N_k},\zeta^{N_k},\widehat{\Lambda},B^{N_k} \Big)
        =
        \Lc^{\Q} \big(\mu,\zeta,\overline{\Lambda},B \big),\;\mbox{in}\;\Wc_p.
    \end{align}
    
\end{proposition}

\paragraph*{Another useful approximation} Using roughly the same arguments as those used in the proof of the \Cref{prop:approximation-FP_BY_SDE-2}, another approximation result can be provided. This can be seen as another version of \Cref{prop:weak-appr} where the sequence $(\overline{\Lambda}^N)_{N \in \N^*}$ is not necessarily a subset of $\M_0\big((\Pc^n_U)^2 \big)$ and the controls that achieve the approximation are  probability measures.

\begin{proposition} \label{prop:approximation-FP_BY_SDE-1}
     Let us stay in the context of {\rm \Cref{prop:weak-appr}} with  $(\overline{\Lambda}^N)_{N \in \N^*}$ not necessarily a subset of $\M_0\big((\Pc^n_U)^2 \big)$. There exists $(\beta^{N})_{N \in \N^*}$ a sequence of $\Pc(U)$--valued $(\widehat{\Fc}_t \otimes \Bc(\Pc^n_U))_{t \in [0,T]}$--predictable processes such that if $(X^N_t)_{t \in [0,T]}:=(X_t)_{t \in [0,T]}$ is the unique strong solution of: $\E^{\widehat{\P}}[\|X^N\|^{p'}]< \infty,$ for all $t \in [0,T]$
    \begin{align*}
        X_t
        =
        \xi
        &+
        \int_0^t
        \int_{(\Pc^n_U)^2} \int_U \hat b \big(r,X_r,B^N,\phi(\eta^{N}),\zeta^{N},\widehat{\mb}^{N}_{r}[m],\nub,u \big) \beta^{N}_{r}(m)(\mathrm{d}u)\;\; \overline{\Lambda}^{N}_r(\mathrm{d}m,\mathrm{d}\nub) \mathrm{d}r
        \\
        &+
        \int_0^t
        \bigg( \int_{(\Pc^n_U)^2} \int_U \hat \sigma \hat \sigma^\top \big(r,X_r,B^N,\phi(\eta^{N}),\zeta^{N},\widehat{\mb}^{N}_{r}[m],\nub,u \big) \beta^{N}_{r}(m)(\mathrm{d}u)\;\; \overline{\Lambda}^{N}_r(\mathrm{d}m,\mathrm{d}\nub) \bigg)^{1/2} \mathrm{d}W_r,\;\widehat{\P}\mbox{--a.e.},
    \end{align*}
    where
    $$
        \widehat{\mb}^{N}_t[m](\mathrm{d}x,\mathrm{d}u)
        :=
        \E^{\widehat{\P}}
        \Big[
        \beta^{N}_t(m)(\mathrm{d}u) \delta_{X^{N}_t}(\mathrm{d}x) \Big|\widehat{\Gc}^N_t
        \Big]
        ~
        \mbox{and}
        ~
        \widehat{\mu}^{N}_t:=\Lc^{\widehat{\P}}(X^{N}_t\big|\widehat{\Gc}^N_t)
        ~~
        \mbox{for all}
        ~
        t \in [0,T],
    $$
    then, one has, for a sub--sequence $(N_j)_{j \in \N^*} \subset \N^*,$ 
    \begin{align*}
        \Lim_{j \to \infty} \E^{\widehat{\P}} \bigg[ \int_0^T \int_{\Pc^n_U} \Wc_{p} \big(\widehat{\mb}^{k_j}_r[m],m \big) \overline{\Lambda}^{N_j}_r(\mathrm{d}m,\Pc^n_U) \mathrm{d}r\bigg]=0\;\mbox{and}\; \Lim_{j \to \infty} \E^{\widehat{\P}} \bigg[\sup_{s \in [0,T]}  \Wc_p\Big( \phi_s(\widehat{\mu}^{N_j}),\phi_s(\mu^{N_j}) \Big) \bigg]=0,
    \end{align*}
   in addition if $\widehat{\Lambda}^{N}_s(\mathrm{d}m,\mathrm{d}\nub)\mathrm{d}s:=\int_{\Pc^n_U} \delta_{\hat{\mb}^N_s[e]}(\mathrm{d}m)\overline{\Lambda}^{N}_s(\mathrm{d}e,\mathrm{d}\nub)\mathrm{d}s,$
\begin{align} \label{eq:cv_result_3}
    \Lim_{j \to \infty} \Lc^{\widehat{\P}}\Big(\widehat{\mu}^{N_j}, \zeta^{N_j}, \widehat{\Lambda}^{N_j},B^{N_j} \Big)=\Lc^{\Q}\Big(\mu, \phi(\mu),\zeta,\overline{\Lambda},B\Big),\;\mbox{in}\;\Wc_p.
\end{align}
\end{proposition}

\begin{remark}
    With exactly the same proof, an important observation is the following: if the coefficients functions $(\hat b, \hat \sigma)$ are of the form of type
    \begin{align*}
        \big(\hat b, \hat \sigma \hat \sigma^\top \big)(t,x,\bb,\pi,\beta,m,\nub,u)
        :=
        \big(\hat b^\star, \hat a^\star \big)(t,\bb,\pi,\beta,\nub)
        +
        \big(\hat b^\circ, \hat a^\circ \big)(t,x,\bb,\pi,\beta,m,u),
    \end{align*}
    where $(\hat b^\star, \hat a^\star,\hat b^\circ, \hat a^\circ)$ are bounded continuous functions, we can replace the convergence assumptions \eqref{eq:general-condition} by 
    \begin{align} \label{eq:general-condition-weak}
        \Lim_{N \to \infty} \Wc_{p'} \bigg(\frac{1}{N} \sum_{i=1}^N \nu^i,\nu \bigg)=0\;\;\mbox{and}\;\Lim_{N \to \infty} \Lc^{\widehat{\P}} \Big(\phi(\mu^{N}),\zeta^{N},\Lambda^{\circ,N},\Lambda^{\star,N},\;B^{N} \Big)
        =
        \Lc^{\Q} \big(\phi(\mu),\zeta,\Lambda^{\circ},\Lambda^{\star},B \big),\;\;\mbox{in}\;\Wc_p,
    \end{align}
    with $\Lambda^{\circ,N}:=\overline{\Lambda}^N_t(\mathrm{d}m,\Pc^n_U)\mathrm{d}t,$  $\Lambda^{\star,N}:=\overline{\Lambda}^N_t(\Pc^n_U,\mathrm{d}\nub)\mathrm{d}t,$ $\Lambda^{\circ}:=\overline{\Lambda}_t(\mathrm{d}m,\Pc^n_U)\mathrm{d}t,$ and $\Lambda^{\star}:=\overline{\Lambda}_t(\Pc^n_U,\mathrm{d}\nub)\mathrm{d}t.$
    And then, in {\rm \Cref{prop:approximation-FP_BY_SDE-2}}, {\rm \Cref{prop:weak-appr}} and {\rm \Cref{prop:approximation-FP_BY_SDE-1}}, the convergence results \eqref{eq:cv_result_1}, \eqref{eq:cv_result_2} and \eqref{eq:cv_result_3} are replaced by
\begin{align*}
    \Lim_{j \to \infty} \Lc^{\widehat{\P}}\Big(\widehat{\mu}^{N_j}, \zeta^{N_j}, \;\widehat{\Lambda}^{N_j}_t(\mathrm{d}m,\Pc^n_U)\mathrm{d}t,\;\widehat{\Lambda}^{N_j}_t(\Pc^n_U,\mathrm{d}\nub)\mathrm{d}t,B^{N_j} \Big)=\Lc^{\Q}\Big(\mu, \phi(\mu),\zeta,\Lambda^\circ,\Lambda^\star,B\Big),\;\mbox{in}\;\Wc_p.
\end{align*}
    In other words, when the variables $(m,\nub)$ of $(\hat b, \hat \sigma \hat \sigma^\top)$ are $``separated",$ we just need separated condition on $(\overline{\Lambda}^N)_{N \in \N^*}$ of type \eqref{eq:general-condition-weak}, i.e. $\overline{\Lambda}^N$ $``separated"$.
\end{remark}

\bibliographystyle{plain}
\bibliography{Extended_mean-field_type_control_problem_revision_1_Arxiv}

\begin{appendix}
\section{Some technical results}

\subsection{Technical proofs}

We will give here successively the proofs of \Cref{lemm:reguralization_FP}, \Cref{prop:weak-appr} and \Cref{prop:approximation-FP_BY_SDE-1}.

\begin{proof}[Proof of {\rm \Cref{lemm:reguralization_FP}}] \label{proof:lemm:reguralization_FP}


Let $\delta>0$ and define
\begin{align*}
    \qb^{\delta}_t(\mathrm{d}m,\mathrm{d}m')
    :=
    \frac{1}{\delta}\int_{(t-\delta) \vee 0}^t \hat \qb^\delta_s(\mathrm{d}m,\mathrm{d}m')\mathrm{d}s,\;\mbox{for all}\;t \in [0,T].
\end{align*}
By using similar approach to \cite[Lemma 4.4]{LipstserShiryaev1977}, the sequence $(\hat \qb^{\delta})_{\delta > 0}$ satisfying: for each $\delta>0,$ $\hat \qb^\delta_t(\mathrm{d}m,\mathrm{d}m')\mathrm{d}t \in \M((\Pc^n_U)^2),$ $\hat \qb^{\delta}:t \in [0,T] \to \hat \qb^{\delta}_t(\mathrm{d}m,\mathrm{d}m') \in (\Pc^n_U)^2$ is continuous, and $\Lim_{\delta \to 0} \hat \qb^{\delta}_t=\hat \qb_t,$ in weakly sense for $\mathrm{d}s$ almost every $t \in [0,T].$

\medskip
Let us fix $t_0 \in (0,T]$, $\phi \in C_b^2(\R^n)$, by \cite[Chapter 2 Section 9 Theorem 10]{KrylovControlledDiffusion}, there exists $v^{\varepsilon,\delta} \in C^{1,2}_b([0,t_0] \x \R^n)$ satisfying:
\begin{align} \label{eq:ParabolicPDE}
    \partial_t v^{\varepsilon,\delta}(t,x)
    +
    \Ac^\varepsilon_t [v^{\varepsilon,\delta}(t,.)][\bb,\nb,\zb,\hat \qb^{\delta}_t](x)
    =
    0
    ~
    \mbox{for all}~(t,x) \in [0,t_0) \x \R^n
    ~
    \mbox{and}
    ~
    v^{\varepsilon,\delta}(t_0,x)=\phi(x).
\end{align}
Notice that, under \Cref{assum:main1}, for each $\varepsilon > 0,$  $\hat a^\varepsilon[\bb,\nb,\zb,\kappa](t,x) \ge \theta \mathrm{I}_{n \x n}$ for all $(t,x,\kappa) \in [0,T] \x \R^n \x \Pc((\Pc^n_U)^2).$ By Proposition \ref{prop:ineq_matrix}, for all $t \in [0,T],$ $x \in \R^n \to (\hat a^\epsilon)^{1/2}[\nb,\zb,\kappa](t,x) \in \S^{n \x n}$ is Lipschitz (with Lipschitz constant independent of $(t,\nb,\zb,\kappa)$).

\medskip
Let $(\Om,\F,\Fc,\P)$ be a probability space supporting $W$ a $\R^n$--valued $(\P,\F)$--Brownian motion, and $\xi$ a $\Fc_0$--random variable such that $\Lc^{\P}(\xi) \in \Pc_{p}(\R^n).$
Now, for every $t \in [0,t_0]$, denote by $X^{\varepsilon,\delta,t,\xi}:=X$ the continuous process unique strong solution of:
\begin{align*}
    X_s
    =
    \xi
    +
    \int_t^s
    \hat b^\varepsilon[\bb,\nb,\zb,\hat \qb^\delta_r](r,X_r) \mathrm{d}r
    +
    \int_t^s
    (\hat a^\varepsilon)^{1/2}[\bb,\nb,\zb,\hat \qb^\delta_r](r,X_r) \mathrm{d}W_r
    ~
    \mbox{for all}~
    s \in [t,T],\;\P\mbox{--a.e.}.
\end{align*}
By applying It\^o's formula, one has that (Feynman Kac's formula)
\begin{align} \label{eq:FeynmanKac}
    v^{\varepsilon,\delta}(t,x)
    =
    \E^{\P} \Big[
    \phi(X^{\varepsilon,\delta,t,\xi}_{t_0})
    \big|
    \xi=x
    \Big]
    =
    \E^{\P} \Big[
    \phi(X^{\varepsilon,\delta,t,x}_{t_0})
    \Big]
    ~
    \mbox{for all}~
    (t,x) \in [0,t_0] \x \R^n.
\end{align}
By definition of $\hat a^{\varepsilon}$ and $\hat b^\varepsilon$ (see \eqref{def_appr-function}), and by using the fact that $\hat \qb^\delta \in \M((\Pc^n_U)^2),$ there exists a constant $C_{\varepsilon}$ (independent of $\delta>0$) such that: for all $(t,x) \in [0,T] \x \R^n,$
\begin{align*}
    \big| \nabla^2 \big(\hat b^{\varepsilon}[\bb,\nb,\zb,\hat \qb^{\delta}_t],\hat a^{\varepsilon}[\bb,\nb,\zb,\hat \qb^{\delta}_t] \big)(t,x) \big| +\big| \nabla \big(\hat b^{\varepsilon}[\bb,\nb,\zb,\hat \qb^{\delta}_t],\hat a^{\varepsilon}[\bb,\nb,\zb,\hat \qb^{\delta}_t] \big)(t,x) \big| \le C_{\varepsilon}.
\end{align*}
Then, by \cite[Chapter 2 Section 8 Theorem 8, Theorem 7]{KrylovControlledDiffusion}, for two unit vectors $(w^1,w^2) \in \R^n \x \R^n,$ there exist two $\R^n$--valued $\F$--adapted continuous processes $Y^{\varepsilon,\delta,t,x,w^1}:=Y$ and $Z^{\varepsilon,\delta,t,x,w^1,w^2}:=Z$ such that
\begin{align*}
    \lim_{h \to 0}\E^{\P} \bigg[ \sup_{s \in [t,t_0]} \Big|\frac{X^{\varepsilon,\delta,t,x+h w^1}_s - X^{\varepsilon,\delta,t,x}_s}{h}-Y_s \Big| \bigg]=0\;\mbox{and}\;\lim_{h \to 0}\E^{\P} \bigg[ \sup_{s \in [t,t_0]} \Big|\frac{Y^{\varepsilon,\delta,t,x+h w^2,w^1}_s - Y^{\varepsilon,\delta,t,x,w^1}_s}{h}-Z_s \Big| \bigg]=0,
\end{align*}
formally speaking, $Y$ can be seen as the $``$derivative$"$ (given a direction $w^1$) of $x \to X^x$, and $Z$ the $``$derivative$"$ (given $w^1$ and another direction $w^2$) of $Y$. In addition $\E^{\P} \big[ \sup_{s \in [t,t_0]}|Y_s| + |Z_s| \big] \le K_{\varepsilon},$ with $K_{\varepsilon}$ depending on $\varepsilon$ but not of $\delta.$
As $\phi \in C^2_b(\R^n),$ by using the previous results and equation \eqref{eq:FeynmanKac}, there exists $\hat K_{\varepsilon}>0$ (independent of $\delta$) satisfying: for all $(t,x) \in [0,T] \x \R^n$
\begin{align} \label{eq:majoration}
    \big|\nabla^2 v^{\varepsilon,\delta}(t,x) \big| + \big|\nabla v^{\varepsilon,\delta}(t,x) \big| + \big|v^{\varepsilon,\delta}(t,x) \big| \le \hat K_{\varepsilon}.
\end{align}
Therefore, for all $\varepsilon>0,$
\begin{align*}
    &\big| \Ac^\varepsilon_t v^{\varepsilon,\delta}(t,\cdot)[\bb,\nb,\zb,\hat \qb_t](x)  - \Ac^\varepsilon_t v^{\varepsilon,\delta}(t,\cdot)[\bb,\nb,\zb,\hat \qb^{\delta}_t](x) \big|
    \le \hat K_{\varepsilon}\big( \big| [\hat b^{\varepsilon},\hat a^{\varepsilon}][\bb,\nb,\zb,\hat \qb_t](t,x)-[\hat b^{\varepsilon},\hat a^{\varepsilon}][\bb,\nb,\zb,\hat \qb^{\delta}_t](t,x) \big| \big),
\end{align*}
by definition \eqref{def_appr-function}, as $\Lim_{\delta \to 0} \hat \qb^{\delta}_t=\hat \qb_t,$ for $\mathrm{d}s$ almost every $t \in [0,T],$ one gets: 
\begin{align} \label{eq:limit-majoration}
    \Lim_{\delta \to 0} \big| \Ac^\varepsilon_t v^{\varepsilon,\delta}(t,\cdot)[\bb,\nb,\zb,\hat \qb_t](x)  - \Ac^\varepsilon_t v^{\varepsilon,\delta}(t,\cdot)[\bb,\nb,\zb,\hat \qb^{\delta}_t](x) \big|=0,
\end{align}
for each $\varepsilon>0$ and $x \in \R^n,$ for $\mathrm{d}s$ almost every $t \in [0,T].$

\medskip
$\underline{Uniqueness}$: For each $\varepsilon>0$ fixed, let us prove the uniqueness of $(\nb^\varepsilon_t)_{t \in [0,T]}$ solution of equation \eqref{eq:FP-original}.  Let $\nb^{1,\varepsilon}$ and $\nb^{2,\varepsilon}$ be two solutions of the Fokker--Planck equation \eqref{eq:FP-original} mentioned in the Lemma, for any $t_0 \in [0,T]$ and $\phi \in C_b^2(\R^n),$ denote by $v:=v^{\varepsilon,\delta,\phi,t_0}$ solution of \eqref{eq:ParabolicPDE} associated to $(t_0,\phi).$ One finds
\begin{align*}
    &\int_{\R^n} \phi(y) \nb_{t_0}^{1,\varepsilon}(\mathrm{d}y)
    -
    \int_{\R^n} \phi(y) \nb_{t_0}^{2,\varepsilon}(\mathrm{d}y)
    \\
    &=
	    \int_0^{t_0} \langle \partial_t v(r,.),\nb^{1,\varepsilon}_r \rangle
	    -
	     \langle \partial_t v(r,.),\nb^{2,\varepsilon}_r \rangle
	     + 
	     \langle \Ac^\varepsilon_r v[\bb,\nb,\zb,\hat \qb_r](.),\nb^{1,\varepsilon}_r \rangle
	     -
	     \langle \Ac^\varepsilon_r v[\bb,\nb,\zb,\hat \qb_r](.),\nb^{2,\varepsilon}_r \rangle\mathrm{d}r
	     \\
	     &=
	     \int_0^{t_0} \langle \Ac^\varepsilon_r v[\bb,\nb,\zb,\hat \qb_r](\cdot)  - \Ac^\varepsilon_r v[\bb,\nb,\zb,\hat \qb^{\delta}_r](\cdot),\nb^{1,\varepsilon}_r  \rangle + \langle \Ac^\varepsilon_r v[\bb,\nb,\zb,\hat \qb_r](\cdot)  - \Ac^\varepsilon_r v[\bb,\nb,\zb,\hat \qb^{\delta}_r](\cdot),\nb^{2,\varepsilon}_r  \rangle~\mathrm{d}r,
\end{align*}
by \eqref{eq:limit-majoration}, given $\varepsilon >0,$ after taking $\delta \to 0,$ by Lebesgue's dominated convergence theorem, $\int_{\R^n} \phi(y)\nb_{t_0}^{1,\varepsilon}(\mathrm{d}y)=\int_{\R^n} \phi(y)\nb_{t_0}^{2,\varepsilon}(\mathrm{d}y),$ this is true for all $(t_0,\phi) \in [0,T] \x C_b^2(\R^n),$ then $\nb^{1,\varepsilon}=\nb^{2,\varepsilon}.$

\medskip
$\underline{Convergence\;of\;\nb^\varepsilon}$: Now, we show the second assertion of our Lemma. Using the fact that $\hat \qb_t(\Z_{\nb_t} \x \Pc^n_U)=1$ $\mathrm{d}t$--almost surely $t \in [0,T],$ one gets for all $t \in [0,T],$
\begin{align*}
    &\int_{\R^n} v^{\varepsilon,\delta}(t,y) \int_{\R^n} G_\varepsilon(z-y)\nb_t(\mathrm{d}z) \mathrm{d}y
    =\int_{\R^n} \int_{\R^n} v^{\varepsilon,\delta}(t,z-y) \nb_t(\mathrm{d}z) G_\epsilon(y) \mathrm{d}y 
    \\
    &=
    \int_{\R^n} v^{\varepsilon,\delta}(0,y) \int_{\R^n} G_\epsilon(z-y) \nu(\mathrm{d}z) \mathrm{d}y
    +
	\int_0^t \int_{\R^n}\bigg[\int_{\R^n} \partial_t v^{\varepsilon,\delta}(s,z-y)\nb_s(\mathrm{d}z) 
	\\
	&~~~~~~~~+ \int_{(\Pc^n_U)^2} \int_{\R^n \x U} \Ac_s [v^{\varepsilon,\delta}(s,\cdot-y)](z,\bb,\nb,\zb,m,\nub,u)m(\mathrm{d}z,\mathrm{d}u)\hat \qb_s(\mathrm{d}m,\mathrm{d}\nub) \bigg] G_\varepsilon(y)~\mathrm{d}y~\mathrm{d}s
	\\
    &=
    \int_{\R^n} v^{\varepsilon,\delta}(0,y) \int_{\R^n} G_\varepsilon(z-y) \nu(\mathrm{d}z) \mathrm{d}y
    +
	\int_0^t \int_{\R^n}\int_{\R^n}\bigg[ \partial_t v^{\varepsilon,\delta}(s,z-y)
	\\
	&~~~~~~~~+ \int_{(\Pc^n_U)^2} \int_{U} \Ac_s [v^{\varepsilon,\delta}(s,\cdot-y)](z,\bb,\nb,\zb,m,\nub,u)m^z(\mathrm{d}u)\hat \qb_s(\mathrm{d}m,\mathrm{d}\nub) \bigg]  G_\varepsilon(y)~\nb_s(\mathrm{d}z)~\mathrm{d}y~\mathrm{d}s
	\\
    &=
    \int_{\R^n} v^{\varepsilon,\delta}(0,y) \int_{\R^n} G_\varepsilon(z-y) \nu(\mathrm{d}z) \mathrm{d}y
    +
	\int_0^t \int_{\R^n}\bigg[ \partial_t v^{\varepsilon,\delta}(s,y)\int_{\R^n}G_\varepsilon(z-y)~\nb_s(\mathrm{d}z)
	\\
	&~~~~~~~~+ \int_{(\Pc^n_U)^2} \int_{\R^n \x U} \hat b(s,z,\bb,\nb,\zb,m,\nub,u)\nabla v^{\varepsilon,\delta}(s,y)G_\varepsilon(z-y)m^z(\mathrm{d}u)\nb_s(\mathrm{d}z)\hat \qb_s(\mathrm{d}m,\mathrm{d}\nub)
	\\
	&~~~~~~~~+ \int_{(\Pc^n_U)^2} \int_{\R^n \x U} \frac{1}{2}\text{Tr}\big[\hat a(s,z,\bb,\nb,\zb,m,\nub,u)\nabla^2 v^{\varepsilon,\delta}(s,y) \big]G_\varepsilon(z-y)m^z(\mathrm{d}u)\nb_s(\mathrm{d}z)\hat \qb_s(\mathrm{d}m,\mathrm{d}\nub) \bigg]  ~\mathrm{d}y~\mathrm{d}s
	\\
    &=
    \int_{\R^n} v^{\varepsilon,\delta}(0,y) \int_{\R^n} G_\epsilon(z-y) \nu(\mathrm{d}z) \mathrm{d}y
    +
	\int_0^t \int_{\R^n}\bigg[ \partial_t v^{\varepsilon,\delta}(s,y)(\nb_s)^{(\varepsilon)}(y)
	\\
	&~~~~~~~~+ \int_{(\Pc^n_U)^2} \int_{\R^n \x U} \hat b(s,z,\bb,\nb,\zb,m,\nub,u)\frac{G_\varepsilon(z-y)}{(m(\mathrm{d}z',U))^{(\varepsilon)}(y)}m(\mathrm{d}z,\mathrm{d}u)\hat \qb_s(\mathrm{d}m,\mathrm{d}\nub)\nabla v^{\varepsilon,\delta}(s,y)(\nb_s)^{(\varepsilon)}(y)
	\\
	&~~~~~~~~+ \frac{1}{2}\text{Tr}\Big[\int_{(\Pc^n_U)^2} \int_{\R^n \x U} \hat a(s,z,\bb,\nb,\zb,m,\nub,u)\frac{G_\varepsilon(z-y)}{(m(\mathrm{d}z',U))^{(\varepsilon)}(y)}m(\mathrm{d}z,\mathrm{d}u)\hat \qb_s(\mathrm{d}m,\mathrm{d}\nub)\nabla^2 v^{\varepsilon,\delta}(s,y) \Big](\nb_s)^{(\varepsilon)}(y) \bigg]  \mathrm{d}y~\mathrm{d}s
    \\
    &=
    \int_{\R^n} v^{\varepsilon,\delta}(0,y) \int_{\R^n} G_\varepsilon(z-y) \nu(\mathrm{d}z) \mathrm{d}y
	+
	\int_0^t \int_{\R^n} [\partial_t v^{\varepsilon,\delta}(r,y)  + \Ac^\varepsilon_r [v^{\varepsilon,\delta}(r,\cdot)][\bb,\nb,\zb,\hat \qb_r](r,y) ](\nb_r)^{(\varepsilon)}(y)~\mathrm{d}y~\mathrm{d}r,
\end{align*}
where for each $\pi \in \Pc(\R^n),$ we write $\pi^{(\varepsilon)}(x):=\int_{\R^n} G_{\varepsilon}(x-z)\pi(\mathrm{d}z),$ for all $x \in \R^n.$

\medskip
Then by \eqref{eq:ParabolicPDE}
\begin{align*}
    &\int_{\R^n} v^{\varepsilon,\delta}(0,y)\nu^{(\varepsilon)}(y)\mathrm{d}y
    \\
    &=
    \int_{\R^n} \phi(y) (\nb_{t_0})^{(\varepsilon)}(y) \mathrm{d}y 
    +
	\int_0^{t_0} \int_{\R^n} [\Ac^\varepsilon_r [v^{\varepsilon,\delta}(r,\cdot)][\bb,\nb,\zb,\hat \qb^{\delta}_r](y)  - \Ac^\varepsilon_r [v^{\varepsilon,\delta}(r,\cdot)][\bb,\nb,\zb,\hat \qb_r](y) ](\nb_r)^{(\varepsilon)}(y)\mathrm{d}y\mathrm{d}r.
\end{align*}
By equation \eqref{eq:FeynmanKac}, one has
\begin{align*}
    \int_{\R^n} v^{\varepsilon,\delta}(0,y)\int_{\R^n} G_\varepsilon(z-y) \nu(\mathrm{d}z) \mathrm{d}y
    =
    \int_{\R^n}
    \E \Big[
    \phi(X^{\varepsilon,\delta,0,\xi}_{t_0})
    \big|
    \xi=y
    \Big]
    \nu^{(\varepsilon)}(y) \mathrm{d}y
    =\int_{\R^n} \phi(x) \nb^{\varepsilon,\delta}_{t_0}(\mathrm{d}x),
\end{align*}
where $\nb^{\varepsilon,\delta}_t:=\Lc^\P(X^{\varepsilon,\delta,0,\xi^{\varepsilon}}_t)$ for $t \in [0,T],$ with $\Lc^{\P}(\xi^{\varepsilon})(\mathrm{d}y)=\nu^{(\varepsilon)}(y)\mathrm{d}y.$ 
Combining the previous equality,
\begin{align*}
    \int_{\R^n} \phi(y)\nb_{t_0}(\mathrm{d}y)
    &=
    \int_{\R^n} \phi(y)\nb_{t_0}(\mathrm{d}y)
    -
    \int_{\R^n} \phi(y) \int_{\R^n} G_\varepsilon(z-y)\nb_{t_0}(\mathrm{d}z) \mathrm{d}y
    +
    \int_{\R^n} \phi(y) \int_{\R^n} G_\varepsilon(z-y)\nb_{t_0}(\mathrm{d}z) \mathrm{d}y
    \\
    &=
    \int_{\R^n} \phi(y)\nb_{t_0}(\mathrm{d}y)
    -
    \int_{\R^n} \phi(y) \int_{\R^n} G_\varepsilon(z-y)\nb_{t_0}(\mathrm{d}z) \mathrm{d}y
    +
    \int_{\R^n} v^\varepsilon(0,y)\int_{\R^n} G_\varepsilon(z-y) \nu(\mathrm{d}z) \mathrm{d}y
    \\
    &~~~~~~~~~~+
	\int_0^{t_0} \int_{\R^n} \Big[\Ac^\varepsilon_r v^{\varepsilon,\delta}(r,\cdot)[\bb,\nb,\zb,\hat \qb_r](y)  - \Ac^\varepsilon_r v^{\varepsilon,\delta}(r,\cdot)[\bb,\nb,\zb,\hat \qb^\delta_r](y) \Big]\int_{\R^n} G_\varepsilon(z-y)\nb_r(\mathrm{d}z)~\mathrm{d}y~\mathrm{d}r
    \\
    &=
    \int_{\R^n} \phi(y) \nb^{\varepsilon,\delta}_{t_0}(\mathrm{d}y)
    +
    \int_{\R^n} \phi(y)\nb_{t_0}(\mathrm{d}y)
    -
    \int_{\R^n} \phi(y) \int_{\R^n} G_\varepsilon(z-y)\mathrm{d}y~ \nb_{t_0}(\mathrm{d}z)
    \\
    &~~~~~~~~~~+
	\int_0^{t_0} \int_{\R^n} \Big[\Ac^\varepsilon_r v^{\varepsilon,\delta}(r,\cdot)[\bb,\nb,\zb,\hat \qb_r](y)  - \Ac^\varepsilon_r v^{\varepsilon,\delta}(r,\cdot)[\bb,\nb,\zb,\hat \qb^\delta_r](y) \Big]\int_{\R^n} G_\varepsilon(z-y)\nb_r(\mathrm{d}z)~\mathrm{d}y~\mathrm{d}r.
\end{align*}
Consequently, for each $\varepsilon>0,$
\begin{align*}
    \Limsup_{\delta \to 0} \Big|\int_{\R^n} \phi(y)\nb_{t_0}(\mathrm{d}y)- \int_{\R^n} \phi(y) \nb^{\varepsilon,\delta}_{t_0}(\mathrm{d}y) \Big| \le \Big| \int_{\R^n} \phi(y)\nb_{t_0}(\mathrm{d}y)
    -
    \int_{\R^n} \phi(y) \int_{\R^n} G_\varepsilon(z-y)\mathrm{d}y~ \nb_{t_0}(\mathrm{d}z) \Big|.
\end{align*}
Finally
\begin{align} \label{eq:convergence}
    \lim_{\varepsilon \to 0}
    \Limsup_{\delta \to 0} \Big|\int_{\R^n} \phi(y)\nb_{t_0}(\mathrm{d}y)- \int_{\R^n} \phi(y) \nb^{\varepsilon,\delta}_{t_0}(\mathrm{d}y) \Big|=0,
\end{align}
for any $\phi \in C^2_b(\R^n)$ and $t_0 \in [0,T]$, where we used that $\lim_{\eps \to 0} |\int_{\R^n} \phi(y)  G_\varepsilon(z-y)\mathrm{d}y-\phi(z)|=0,$ for all $z \in \R^n$.

\medskip
Notice that $\nu^{(\varepsilon)}(y)(\mathrm{d}y)$ converges weakly to $\nu(\mathrm{d}y).$ By Skorokhod's representation theorem, one can find a probability space $(\tilde \Om,\tilde \Fc,\tilde \P)$ supporting $(\xi^{\varepsilon})_{\varepsilon>0}$ and $\xi$ such that $\Lc^{\tilde \P}(\xi^{\varepsilon})=\nu^{(\varepsilon)}(y)(\mathrm{d}y)$ and $\Lc^{\tilde \P}(\xi)=\nu(\mathrm{d}y),$ and $\lim_{\varepsilon \to 0}\xi^{\varepsilon}=\xi$ $\tilde \P$ a.e.. And when $\Lc^{\P}(\xi)=\nu \in \Pc_{p'}(\nu),$ one has $\sup_{\varepsilon>0} \E^{\tilde \P}[|\xi^{\varepsilon}|^{p'}]=\sup_{\varepsilon>0} \int_{\R^n} |y|^{p'} \nu^{(\varepsilon)}(y)(\mathrm{d}y) \le C(1+ \int_{\R^n} |y|^{p'} \nu(\mathrm{d}y)) < \infty,$ by using standard techniques of uniform integrability, $\lim_{\varepsilon \to 0}\E^{\tilde \P}[|\xi^{\varepsilon}-\xi|^{p}]=0,$ recall that $p'>p.$ If necessary, it is possible to enlarge the initial space, for sake of clarity and without technical problems, let us assume $(\tilde \Om,\tilde \Fc, \tilde \P)$ is equal to the initial space $(\Om, \Fc, \P).$
For each $\varepsilon>0,$ let $X^{\varepsilon}$ be the continuous process unique strong solution of 
\begin{align*}
    X^{\varepsilon}_s
    =
    \xi
    +
    \int_0^s
    \hat b^\varepsilon[\bb,\nb,\zb,\hat \qb_r](r,X^{\varepsilon}_r) \mathrm{d}r
    +
    \int_0^s
    (\hat a^\varepsilon)^{1/2}[\bb,\nb,\zb,\hat \qb_r](r,X^{\varepsilon}_r) \mathrm{d}W_r
    ~
    \mbox{for all}~
    s \in [0,T],\;\P\mbox{--a.e.}.
\end{align*}
By using the regularity of $(\hat b^\varepsilon,\hat a^{\varepsilon})$ for $\varepsilon$ fixed, it is straightforward to find that
\begin{align*}
    \Lim_{\varepsilon \to 0}\Lim_{\delta \to 0} \E^{\P} \bigg[ \sup_{t \in [0,T]} \big|X^{\varepsilon}_t-X^{\varepsilon,\delta,0,\xi^{\varepsilon}}_t \big|^p \bigg]=0.
\end{align*}
By It\^o's formula and uniqueness of the Fokker--Planck equation \eqref{eq:FP-original}, $\nb^{\varepsilon}_t=\Lc^{\P}(X^\varepsilon_t)$ for each $t \in [0,T].$ Thanks to \eqref{eq:convergence} and the previous result, one gets that, in weakly convergence sense, $\Lim_{\varepsilon} \nb^{\varepsilon}_t=\Lim_{\varepsilon \to 0} \Lim_{\delta \to 0} \nb^{\varepsilon,\delta}_t=\nb_t$ for each $t \in [0,T].$ Therefore, we proved that: for each $t \in [0,T],$ $\nb^{\varepsilon}_t$ converges weakly to $\nb_t.$  To deduce the Wasserstein convergence $\Wc_p,$ notice that: $\sup_{\varepsilon > 0} \sup_{t \in [0,T]} \int_{\R^n} |x|^{p'} \nb^\epsilon_t(\mathrm{d}x) \le C(1+ \int_{\R^n} |y|^{p'} \nu(\mathrm{d}y)) < \infty,$ and
\begin{align*}
    \Limsup_{\delta' \to 0}\sup_{\varepsilon > 0} \sup_{s \in [0,T]} \Wc_p \big(\nb^\varepsilon_{(s + \delta') \wedge T},\nb^\varepsilon_{s} \big)^p
    &=
    \Limsup_{\delta' \to 0}\sup_{\varepsilon > 0} \sup_{s \in [0,T]} \Wc_p \big(\Lc^\P(X^{\varepsilon}_{(s + \delta') \wedge T}), \Lc^\P(X^{\varepsilon}_{s})\big)^p
    \\
    &\le 
    \Limsup_{\delta' \to 0}\sup_{\varepsilon > 0} \sup_{s \in [0,T]} \E^{\P} \big[\big|X^{\varepsilon}_{(s + \delta') \wedge T} - X^{\varepsilon}_{s} \big|^p \big]
    \le \hat C
    \Limsup_{\delta' \to 0}  \delta'=0,
\end{align*}
where the last equality follows from the Holder's property of trajectories of $X^{\varepsilon}$ with a constant independent of $\varepsilon$ (essentially because $(\hat b, \hat \sigma)$ are bounded). By Aldous' criterion \cite[Lemma 16.12]{kallenberg2002foundations} (see also proof of \cite[Proposition-B.1]{carmona2014mean} ), $(\nb^\varepsilon)_{\varepsilon >0}$ is relatively compact in $C([0,T];\Pc_p(\R^n))$ with the metric $\Delta(\nu,\nu'):=\sup_{t \in [0,T]} \Wc_p(\nu_t,\nu'_t)$ for all $(\nu,\nu') \in C([0,T];\Pc_p(\R^n)) \x C([0,T];\Pc_p(\R^n)).$ As for each $t \in [0,T],$ $\nb^\varepsilon_t$ converges weakly to $\nb_t,$ then the limit of each sub--sequence of $(\nb^\varepsilon)_{\varepsilon >0}$ is $\nb,$ consequently $\Lim_{\varepsilon \to 0}\sup_{t \in [0,T]} \Wc_p(\nb^{\varepsilon}_t,\nb_t)=0.$




\end{proof}

\medskip
\begin{proof}[Proof of {\rm \Cref{prop:weak-appr}}] \label{proof:prop:weak-appr}
Before starting, let us mention that many parts of this proof use \Cref{prop:approximation-FP_BY_SDE-2} and its associated proof. 

Let us take the sequence of processes $(\alpha^{i,N})_{(i,N) \in \N^* \x \N^*}$ given in \Cref{prop:approximation-FP_BY_SDE-2} with $\Lc^{\widehat \P}(\xi^i)=\nu^i=\nu$ for each $i,$ and define the unique strong solution $X^{i,N}$ of: $X^{i,N}_0=\xi$ and
\begin{align*} 
    \mathrm{d}X^{i,N}_t
    =
    \hat b\big(t,X^{i,N}_t,B^N,\phi(\widehat{\mu}^{i,N}),\zeta^N, \widehat{m}^{i,N}_t,\nub^N_t,\alpha^{i,N}_{t}\big) \mathrm{d}t
    +
    \hat \sigma \big(t,X^{i,N}_t,B^N,\phi(\widehat{\mu}^{i,N}),\zeta^N, \widehat{m}^{i,N}_t,\nub^N_t,\alpha^{i,N}_{t}\big)
    \mathrm{d}W^i_t,
\end{align*}
with $\;\widehat{m}^{i,N}_t:=\Lc^{\widehat{\P}}\big(X^{i,N}_t,\alpha^{i,N}_t\big|\widehat{\Gc}^N_t\big)\;\mbox{and}\;\widehat{\mu}^{i,N}_t:=\Lc^{\widehat{\P}}\big(X^{i,N}_t\big|\widehat{\Gc}^N_t\big).$  As $\alpha^{i,N}$ is $\widehat{\F}^{i,N}$--predictable ($\widehat{\F}^{i,N}$ is defined in \eqref{def:filtration}), there exists a Borel function $G:[0,T] \x \R^n \x \M\big((\Pc^n_U)^2 \big) \x \Cc^n_{\Wc} \x \Cc^n_{\Wc} \x \Cc^n \x \Cc^\ell \x [0,1] \to U$ satisfying $\alpha^{i,N}_t= G\big(t,\xi^i, \overline{\Lambda}^{N}_{t \wedge \cdot},\phi_{t \wedge \cdot}(\mu^N),\zeta^N_{t \wedge \cdot},W^i_{t \wedge \cdot}, B^N_{t \wedge \cdot},Z^i \big),$ $\mathrm{d}t \otimes \mathrm{d}\widehat{\P}$--a.e. . Define $\alpha^N_t:=G\big(t,\xi, \overline{\Lambda}^{N}_{t \wedge \cdot},\phi_{t \wedge \cdot}(\mu^N),\zeta^N_{t \wedge \cdot},W_{t \wedge \cdot}, B^N_{t \wedge \cdot},Z \big).$ Let $X^N$ be the unique strong solution of equation \eqref{eq:general-weak} (associated to $\alpha^N$). 
By independence condition in Assumption \eqref{eq:assum-independence}, recall that $\widehat{m}^{N}$ is given in equation \eqref{eq:general-weak}, 
\begin{align} \label{eq:assum-independence-property}
    \widehat{m}^{i,N}_t=\widehat{m}^{N}_t,\;\widehat{\P}\mbox{--a.e.},\;\mbox{and}\;\mbox{given the}\;\sigma\mbox{--field}\;\widehat{\Gc}^N_t,\;\mbox{for}\; i \neq j,\;(X^{i,N}_{t \wedge \cdot},\alpha^{i,N}_t)\;\mbox{are independent of}\;(X^{j,N}_{t \wedge \cdot},\alpha^{j,N}_t)
\end{align}
and $\Lc^{\widehat{\P}}\big(X^{i,N},\xi^i,\overline{\Lambda}^{N},\phi(\mu^N),\zeta^N,W^i, B^N,Z^i \big)=\Lc^{\widehat{\P}}\big(X^{N},\xi,\overline{\Lambda}^{N},\phi(\mu^N),\zeta^N,W, B^N,Z \big)$ for each $i.$


\medskip
Let us introduce for each $N \in \N^*,$ the measure on $[0,T] \x \Pc(\Cc^n \x U) \x \Pc(\Cc^n \x U) $
\begin{align*}
    \Gamma^N_t(\mathrm{d}e,\mathrm{d}e')\mathrm{d}t
    :=
    \E^{\widehat{\P}} \bigg[\delta_{\big(\overline{\beta}^N_t,\;\Lc^{\P} (X^{i,N}_{t \wedge \cdot},\alpha^{i,N}_t | \widehat{\Gc}^N_t )  \big)}(\mathrm{d}e,\mathrm{d}e')\mathrm{d}t\bigg],\;\mbox{with}\;\overline{\beta}^N_t(\mathrm{d}\xb,\mathrm{d}u):=\frac{1}{N} \sum_{i=1}^N \delta_{(X^{i,N}_{t \wedge \cdot},\alpha^{i,N}_t)}(\mathrm{d}\xb,\mathrm{d}u).
\end{align*}
As $(\hat b,\hat \sigma)$ are bounded and $\nu \in \Pc_{p'}(\R^n)$, it is straightforward to check that $\sup_{N \ge 1} \sup_{i \in \{1,...,N\}} \E^{\widehat{\P}} \big[ \sup_{t \in [0,T]} \big|X^{i,N}_t \big|^{p'} \big] < \infty,$ and hence $(\Gamma^N)_{N \in \N^*}$ is relatively compact for the Wasserstein metric $\Wc_p$. Denote by $\Gamma^\infty$ the limit of a sub--sequence of $(\Gamma^N)_{N \in \N^*}.$ For simplicity, we will use the same notation for the sequence and the sub--sequence. One gets
\begin{align} \label{eq:property-subsequence}
    \Gamma^\infty_t(\mathrm{d}e,\mathrm{d}e')\mathrm{d}t
    =
    \delta_{e}(\mathrm{d}e')\Gamma^\infty_t \big(\mathrm{d}e,\Pc(\Cc^n \x U) \big)\mathrm{d}t.
\end{align}
It is enough to show that: for all $Q \in \N^*,$ any bounded functions $(f^q)_{d \in \{1,...,Q \}}: \Cc^n \x U \to \R^Q$ and $g: [0,T] \x \Pc(\Cc^n \x U) \to \R$
\begin{align*}
    \int_0^T \int_{\Pc(\Cc^n \x U)^2} \prod_{q=1}^Q \langle f^q, e \rangle g(t,e')\Gamma^\infty_t(\mathrm{d}e,\mathrm{d}e')\mathrm{d}t
    =
    \int_0^T \int_{\Pc(\Cc^n \x U)} \prod_{q=1}^Q \langle f^q, e \rangle g(t,e)\Gamma^\infty_t \big(\mathrm{d}e,\Pc(\Cc^n \x U) \big)\mathrm{d}t.
\end{align*}
Let us prove this result when $Q=2,$ the case $Q \in \N^*$ is true by similar way.
\begin{align*}
    &\int_0^T \int_{\Pc(\Cc^n \x U)^2} \prod_{q=1}^Q \langle f^q, e \rangle g(t,e')\Gamma^\infty_t(\mathrm{d}e,\mathrm{d}e')\mathrm{d}t=
    \Lim_N \frac{1}{N} \sum_{i,j=1}^N \E^{\widehat{\P}} \bigg[\int_0^T f^1 \big(X^{i,N}_{t \wedge \cdot},\alpha^{i,N}_t \big) f^2 \big(X^{j,N}_{t \wedge \cdot},\alpha^{j,N}_t \big) g(t,\widehat{m}^{N}_t) \mathrm{d}t \bigg]
    \\
    &=
    \Lim_N  \frac{1}{N^2} \Bigg( \sum_{i \neq j} \E^{\widehat{\P}} \bigg[\int_0^T f^1 \big(X^{i,N}_{t \wedge \cdot},\alpha^{i,N}_t \big) f^2 \big(X^{j,N}_{t \wedge \cdot},\alpha^{j,N}_t \big) g(t,\widehat{m}^{N}_t) \mathrm{d}t \bigg] 
    +\sum_{i=1}^N \E^{\widehat{\P}} \bigg[\int_0^T f^1 \big(X^{i,N}_{t \wedge \cdot},\alpha^{i,N}_t \big) f^2 \big(X^{i,N}_{t \wedge \cdot},\alpha^{i,N}_t \big) g(t,\widehat{m}^{N}_t) \mathrm{d}t \bigg]\Bigg)
    \\
    &=
    \Lim_N \Bigg( \frac{1}{N^2} \sum_{i \neq j} \E^{\widehat{\P}} \bigg[ \int_0^T \E^{\widehat{\P}} \big[f^1 \big(X^{i,N}_{t \wedge \cdot},\alpha^{i,N}_t \big) \big| \widehat{\Gc}^N_t \big] \E^{\widehat{\P}} \big[f^2 \big(X^{j,N}_{t \wedge \cdot},\alpha^{j,N}_t \big) \big| \widehat{\Gc}^N_t \big] g(t,\widehat{m}^{N}_t) \mathrm{d}t \bigg] 
    \\
    &~~~~~~~~~~~~~~~~~~~~~~~~~~~~~~~~~~~~~~~+\frac{1}{N^2} \sum_{i=1}^N \E^{\widehat{\P}} \bigg[\int_0^T f^1 \big(X^{i,N}_{t \wedge \cdot},\alpha^{i,N}_t \big) f^2 \big(X^{i,N}_{t \wedge \cdot},\alpha^{i,N}_t \big) g(t,\widehat{m}^{N}_t) \mathrm{d}t \bigg]\Bigg)
    \\
    &=
    \Lim_N \Bigg( \E^{\widehat{\P}} \bigg[ \int_0^T \langle f^1,\widehat{m}^N_t \rangle \langle f^2,\widehat{m}^N_t \rangle  g(t,\widehat{m}^{N}_t) \mathrm{d}t \bigg] 
    -\frac{1}{N^2} \sum_{i=1}^N \E^{\widehat{\P}} \bigg[ \int_0^T \E^{\widehat{\P}} \big[f^1 \big(X^{i,N}_{t \wedge \cdot},\alpha^{i,N}_t \big) \big| \widehat{\Gc}^N_t \big] \E^{\widehat{\P}} \big[f^2 \big(X^{i,N}_{t \wedge \cdot},\alpha^{i,N}_t \big) g(t,\widehat{m}^{N}_t) \big| \widehat{\Gc}^N_t \big] \mathrm{d}t \bigg]
    \\
    &~~~+ \frac{1}{N^2} \sum_{i=1}^N\E^{\widehat{\P}} \bigg[\int_0^T f^1 \big(X^{i,N}_{t \wedge \cdot},\alpha^{i,N}_t \big) f^2 \big(X^{i,N}_{t \wedge \cdot},\alpha^{i,N}_t \big) g(t,\widehat{m}^{N}_t) \mathrm{d}t \bigg]\Bigg)
    =\int_0^T \int_{\Pc(\R^n \x U)}  \langle f^1, e \rangle \langle f^2, e \rangle g(t,e)\Gamma^\infty_t \big(\mathrm{d}e,\Pc(\Cc^n \x U) \big)\mathrm{d}t,
\end{align*}
where we used result \eqref{eq:assum-independence-property} and the fact that the terms starting with $\frac{1}{(N_l)^2} \sum_{i=1}^{N_l}$ go to zero because $(f^1,f^2,g)$ are bounded.

Next, for all $t \in [0,T],$ using Lipshitz property, there exists a constant $C>0$ (which changes from line to line)
\begin{align*}
    &\E^{\widehat{\P}} \bigg[ \sup_{s \in [0,t]} \big|X^{i,N}_s-\widehat{X}^i_s \big|^p \bigg] 
    \\
    &\le C  \E^{\widehat{\P}} \bigg[\int_0^t \sup_{r \in [0,s]} \big|X^{i,N}_r-\widehat{X}^i_r \big|^p + \sup_{r \in [0,s]} \Wc_p \big( \widehat{\mu}^N_r, \frac{1}{N}\sum_{i=1}^N \delta_{X^{i,N}_r} \big)^p + \Wc_p \big(\widehat{m}^N_s,\frac{1}{N}\sum_{i=1}^N \delta_{(X^{i,N}_s,\alpha^{i,N}_s)} \big)^p  \mathrm{d}s \bigg]
    \\
    &\le C  \E^{\widehat{\P}} \bigg[\int_0^t \sup_{r \in [0,s]} \big|X^{i,N}_r-\widehat{X}^i_r \big|^p + \Wc_p \big(\Lc^{\P} \big( (X^{i,N}_{s \wedge \cdot},\alpha^{i,N}_s \big| \widehat{\Gc}^N_s\big),\overline{\beta}^N_s\big)^p  \mathrm{d}s \bigg],
\end{align*}
recall that $(\widehat{X}^1,...,\widehat{X}^N)$ are defined in equation \eqref{eq:general-Nagents} (in \Cref{prop:approximation-FP_BY_SDE-2}), and $\;\widehat{m}^{N}_t:=\Lc^{\widehat{\P}}\big(X^{N}_t,\alpha^{N}_t\big|\widehat{\Gc}^N_t\big)\;\mbox{and}\;\widehat{\mu}^{N}_t:=\Lc^{\widehat{\P}}\big(X^{N}_t\big|\widehat{\Gc}^N_t\big).$ 

\medskip
Then by Gronwall Lemma $\E^{\widehat{\P}} \Big[ \sup_{s \in [0,T]} \big|X^{i,N}_s-\widehat{X}^i_s \big|^p \Big] 
    \le C  \E^{\widehat{\P}} \Big[\int_0^T \Wc_p \big(\Lc^{\widehat{\P}} \big( (X^{i,N}_{s \wedge \cdot},\alpha^{i,N}_s \big| \widehat{\Gc}^N_s\big),\overline{\beta}^N_s\big)^p  \mathrm{d}s \Big].$
As,
\begin{align*}
    &\E^{\widehat{\P}} \bigg[\int_0^T \Wc_p \big(\widehat{m}^{N}_s, \frac{1}{N}\sum_{i=1}^N \delta_{(\widehat{X}^{i,N}_s,\alpha^{i,N}_s)} \big)^p \mathrm{d}s \bigg]
    \\
    &\le \E^{\widehat{\P}} \bigg[\int_0^T \Wc_p \big(\widehat{m}^{N}_s,\frac{1}{N}\sum_{i=1}^N \delta_{(X^{i,N}_s,\alpha^{i,N}_s)} \big)^p \mathrm{d}s \bigg] 
    + 
    \E^{\widehat{\P}} \bigg[\int_0^T \Wc_p \big(\frac{1}{N}\sum_{i=1}^N \delta_{(X^{i,N}_s,\alpha^{i,N}_s)},\frac{1}{N}\sum_{i=1}^N \delta_{(\widehat{X}^{i,N}_s,\alpha^{i,N}_s)} \big)^p \mathrm{d}s \bigg]
    \\
    &\le C \bigg( \E^{\widehat{\P}} \bigg[\int_0^T \Wc_p \big(\Lc^{\widehat{\P}} \big( (X^{i,N}_{s \wedge \cdot},\alpha^{i,N}_s \big| \widehat{\Gc}^N_s\big),\overline{\beta}^N_s\big)^p  \mathrm{d}s \bigg] 
    + 
    \E^{\widehat{\P}} \bigg[\int_0^T \big|X^{i,N}_s-\widehat{X}^i_s \big|^p \mathrm{d}s \bigg] \bigg)
    \\
    &\le C  \E^{\widehat{\P}} \bigg[\int_0^T \Wc_p \big(\Lc^{\widehat{\P}} \big( (X^{i,N}_{s \wedge \cdot},\alpha^{i,N}_s \big| \widehat{\Gc}^N_s\big),\overline{\beta}^N_s\big)^p  \mathrm{d}s \bigg],
\end{align*}
therefore, by taking the sub--sequence corresponding to the $\Limsup,$ by result \eqref{eq:property-subsequence}, 
$$
    \Limsup_{l \to \infty} \E^{\widehat{\P}} \bigg[\int_0^T \Wc_p \big(\widehat{m}^{N_l}_s, \frac{1}{N_l}\sum_{i=1}^{N_l} \delta_{(\widehat{X}^{i,N}_s,\alpha^{i,N}_s)} \big)^p \mathrm{d}s + \sup_{t \in [0,T]} \Wc_p \big(\phi_t(\widehat{\mu}^{N_l}),\phi_t(\mu^{N_l}) \bigg]=0.
$$
From all previous results, it is straightforward to check that
\begin{align*}
    \Lim_{N \to \infty} \Wc_p \Big( \Lc^{\widehat{\P}} \big(\widehat{\mu}^{N},\zeta^{N},\delta_{(\hat{m}^N_s,\nub^N_s)}(\mathrm{d}m,\mathrm{d}\nub)\mathrm{d}s,B^{N} \big), \Lc^{\widehat{\P}} \big(\widehat{\gamma}^N,\zeta^{N},\delta_{(\hat{\theta}^N_s,\nub^N_s)}(\mathrm{d}m,\mathrm{d}\nub)\mathrm{d}s,B^{N} \big) \Big)=0,
\end{align*}
where $\widehat{\gamma}^N_t:=\frac{1}{N}\sum_{i=1}^N \delta_{(\widehat{X}^{i,N}_t)}$ and $\widehat{\theta}_t:=\frac{1}{N}\sum_{i=1}^N \delta_{(\widehat{X}^{i,N}_t,\alpha^{i,N}_t)}.$
Consequently, by \Cref{prop:approximation-FP_BY_SDE-2}
\begin{align*}
    \Lim_{k \to \infty} \Lc^{\widehat{\P}} \big(\widehat{\mu}^{N_k},\zeta^{N_k},\delta_{(\hat{m}^{N_k}_s,\nub^{N_k}_s)}(\mathrm{d}m,\mathrm{d}\nub)\mathrm{d}s,B^{N_k} \big)=\Lim_{k \to \infty} \Lc^{\widehat{\P}} \big(\widehat{\gamma}^{N_k},\zeta^{N_k},\delta_{(\hat{\theta}^{N_k}_s,\nub^{N_k}_s)}(\mathrm{d}m,\mathrm{d}\nub)\mathrm{d}s,B^{N_k} \big)=\Lc^{\Q} \big(\mu,\zeta,\overline{\Lambda},B \big),
\end{align*}
recall that $\;\widehat{m}^{N}_t:=\Lc^{\widehat{\P}}\big(X^{N}_t,\alpha^{N}_t\big|\widehat{\Gc}^N_t\big)\;\mbox{and}\;\widehat{\mu}^{N}_t:=\Lc^{\widehat{\P}}\big(X^{N}_t\big|\widehat{\Gc}^N_t\big).$
\end{proof}

\medskip
\begin{proof}[Proof of {\rm \Cref{prop:approximation-FP_BY_SDE-1}}]\label{proof:prop:approximation-FP_BY_SDE-1}
The proof of this Proposition is exactly the same as Proposition \ref{prop:approximation-FP_BY_SDE-2}, we essentially recall the main step.


\medskip
$\underline{Approximation\;by\;SDE:Tightness\;and\;identification\;of\;the\;limit\;process}$: Let us define the unique strong solution $Z^{\varepsilon,N}$ of:
\begin{align*}
    Z^{\varepsilon,N}_t
    =
    \xi
    +
    \int_0^t
    \hat b^{\varepsilon}[B^N,\phi(\mu^N),\zeta^N,\overline{\Lambda}^N_r](r,Z^{\varepsilon,N}_r) \mathrm{d}r
    +
    \int_0^t
    (\hat a^{\varepsilon})^{1/2}[B^N,\phi(\mu^N),\zeta^N,\overline{\Lambda}^N_r](r,Z^{\varepsilon,N}_r) \mathrm{d}W_r,\;t \in [0,T],\;\widehat{\P}\mbox{--a.e.}.
\end{align*}
And for all $(t,\om) \in [0,T] \x \Om,$ denote $\vartheta^{\varepsilon,N}_t(\om):=\Lc^{\widehat{\P}}\big( Z^{\varepsilon,N}_{t} \big| \widehat{\Gc}^N_t \big)(\om),$  and 
\begin{align*}
    \mathrm{P}^{\varepsilon,N}
    :=
    \Lc^{\widehat{\P}} \Big( \vartheta^{\varepsilon,N}, B^N, \phi(\mu^N),\zeta^N, \overline{\Lambda}^{N} \Big) \in \Pc\big( \Cc^n_{\Wc} \x \Cc^\ell \x \Cc^n_{\Wc} \x \Cc^n_{\Wc} \x \M\big((\Pc^n_U)^2 \big) \big).
\end{align*}

As $[\hat b^\varepsilon,\hat a^{\varepsilon}]$ are bounded, again it is straightforward to check that $(\mathrm{P}^{\varepsilon,N})_{N \in \N^*}$ is relatively compact for the Wasserstein metric $\Wc_p.$
Denote by $\mathrm{P}^{\varepsilon,\infty}$ the limit of a sub--sequence of $(\mathrm{P}^{\varepsilon,N})_{N \in \N^*}.$
Therefore, under \Cref{assum:main1}, by applying similar techniques to those used in $step\;2.2$ of proof of \Cref{prop:approximation-FP_BY_SDE-2}, one gets for all $(f,t) \in C^{2}_b(\R^n;\R)\x [0,T],$ one gets
\begin{align} \label{eq:limit_FK-PL}
    &\langle f,\beta_t \rangle
    =
    \int_{\R^n} f(y) \nu(\mathrm{d}y)
    +
    \int_0^t \int_{\R^n} \Ac^{\varepsilon}_r f\big[B,\beta^{\mu},\beta^\zeta, \overline{\beta}\big](x)\beta_r(\mathrm{d}x) \mathrm{d}r,\;\;\;\mathrm{P}^{\varepsilon,\infty}\mbox{--a.e.},
\end{align}
where $(\beta,B,\beta^\mu,\beta^\zeta,\overline{\beta})$ is the canonical element on $\Cc^n_{\Wc} \x \Cc^\ell \x \Cc^n_{\Wc} \x \Cc^n_{\Wc} \x \M\big((\Pc^n_U)^2.$  Using a countable family of $(f,t),$ we can deduce $\mathrm{P}^{\varepsilon,\infty}$--a.e. equation \eqref{eq:limit_FK-PL} is true  for all $(f,t) \in C^{2}_b(\R^n;\R)\x [0,T].$
By \Cref{lemma:existence-measurability_FP}, one has $\beta=\Phi^\varepsilon\big(B,\beta^{\mu},\beta^\zeta,\overline{\beta} \big)$ where $\Phi^\varepsilon$ is the function used in equation \eqref{eq:function_FK-PL}. Also
\begin{align*}
    \Lc^{\mathrm{P}^{\varepsilon,\infty}} \big( B,\beta^{\mu},\beta^\zeta,\overline{\beta} \big)
    &=
    \Lim_{N \to \infty} \Lc^{\mathrm{P}^{\varepsilon,N}} \big( B,\beta^{\mu},\beta^\zeta,\overline{\beta} \big)
    =
    \Lim_{N \to \infty} \Lc^{\widehat{\P}} \big( B,\phi(\mu),\zeta^N,\overline{\Lambda}^{N} \big)
    =
    \Lc^{\Q}\big( B,\phi(\mu),\zeta,\overline{\Lambda} \big).
\end{align*}
Then $\Lc^{\mathrm{P}^{\varepsilon,\infty}} \big(\beta,B,\beta^{\mu},\beta^\zeta,\overline{\beta} \big)=\Lc^{\Q}\big(\mu^\varepsilon, B,\phi(\mu),\zeta,\overline{\Lambda}\big).$ This result is true for any limit of any sub--sequence of $(\mathrm{P}^{\varepsilon,N})_{N \in \N^*},$ consequently $(\mathrm{P}^{\varepsilon,N})_{N \in \N^*}$ converges and
\begin{align*}
    &\Lim_{N \to \infty} \Lc^{\widehat{\P}} \big( \vartheta^{\varepsilon,N}, B^N, \phi(\mu^N),\zeta^N, \overline{\Lambda}^N \big)
    =
    \Lc^{\Q} \big( \mu^{\varepsilon}, B, \phi(\mu),\zeta, \overline{\Lambda} \big).
\end{align*}

\medskip
$\underline{Last\;approximation}$:
Let us consider for all $(\varepsilon,N) \in (0,\infty) \x \N^*,$ the $\widehat{\F}$--adapted $\R^n$--valued continuous process $X^{\varepsilon,N}:=X$ strong solution of : for all $s \in [0,T]$
\begin{align*}
    X_s
    =
    \xi
    &+
    \int_0^s
    \int_{(\Pc^n_U)^2} \int_{U} \hat b(r, X_r,B^N, \phi(\widehat{\mu}^{\varepsilon,N}),\zeta^N,\widehat{\mb}^{\varepsilon,N}_r[m],\nub,u) {H}^{\varepsilon}(Z^{\varepsilon,N}_r,m)(\mathrm{d}u) \overline{\Lambda}^N_r(\mathrm{d}m,\mathrm{d}\nub) \mathrm{d}r
    \\
    &+
    \int_0^s
    \bigg(  \int_{(\Pc^n_U)^2} \int_{U} \hat \sigma \hat \sigma^\top(r, X_r,B^N, \phi(\widehat{\mu}^{\varepsilon,N}),\zeta^N,\widehat{\mb}^{\varepsilon,N}_r[m],\nub,u) {H}^{\varepsilon}(Z^{\varepsilon,N}_r,m)(\mathrm{d}u) \overline{\Lambda}^N_r(\mathrm{d}m,\mathrm{d}\nub) \bigg)^{1/2} \mathrm{d}W_r,\;\widehat{\P} \mbox{--a.e.}
\end{align*}
where recall that ${H}^{\varepsilon}(x,m)(\mathrm{d}u):=\int_{\R^n}m(\mathrm{d}u,\mathrm{d}y)\frac{G_\varepsilon(x-y)}{(m(U,\mathrm{d}z))^{(\varepsilon)}(x)}$ and
$$
    \widehat{\mb}^{\varepsilon,N}_r[m](\mathrm{d}z,\mathrm{d}u)
    :=
    \E^{\widehat{\P}} 
    \Big[
    {H}^{\varepsilon}(Z^{\varepsilon,N}_r,m)(\mathrm{d}u)\delta_{ X_r^{\varepsilon,N}}(\mathrm{d}z)
    \Big| \widehat{\Gc}^N_r\Big]\;\mbox{and}\; \widehat{\mu}^{\varepsilon,N}_r:=\Lc^{\widehat{\P}}( X^{\varepsilon,N}_r| \widehat{\Gc}^N_r).
$$

Combining Proposition \ref{prop:ineq_matrix} and the techniques applied in $step\;3$ of Proof of Proposition \ref{prop:approximation-FP_BY_SDE-2}, one gets 
$$
    \Lim_{ \varepsilon \to 0} \Lim_{N \to \infty} \E^{\widehat{\P}} \bigg[\sup_{t \in [0,T]}|X^{\varepsilon,N}_t-Z^{\varepsilon,N}_t|^{p} \bigg]=0\;\;\mbox{and}\;\;\Lim_{ \varepsilon \to 0} \Lim_{N \to \infty} \E^{\widehat{\P}} \bigg[\int_0^T \int_{\Pc^n_U}  \Wc_p ( \widehat{\mb}^{\varepsilon,N}_r[m],m) \overline{\Lambda}^{N}_r(\mathrm{d}m,\Pc^n_U) \mathrm{d}r \bigg]=0.
$$
Similarly, $\Limsup_{\varepsilon \to 0} \Limsup_{N \to \infty} \E^{\widehat{\P}} \bigg[\sup_{s \in [0,T]}  \Wc_p\big( \phi(\widehat{\mu}^{\varepsilon,N}),\phi_s(\mu^N) \big) \bigg]=0.$ 
$X^{\varepsilon, N}$ is the process we are looking for.

\end{proof}

\subsection{Regularization by convolution and consequence}

This part presents results about the approximation of Borel measurable functions through a sequence of $``$smooth$"$ functions. The main point is that this approximation is achieved via a convolution. The convolution is realized by a probability measure constructed by an SDE process. Before presenting the main results, we start by recalling an equivalence result coming from \cite[Proposition 4.2]{gyongy1986mimicking}.



\medskip
Let $(\Om,\F,\Fc,\P)$ be a filtered probability space supporting $W$ a $\R^n$--valued $\F$--Brownian motion and $\xi$ a $\Fc_0$--random variable verifying $\E^\P[|\xi|^p] < \infty$, $(b_t,\sigma_t)_{t \in [0,T]}$ $\R^n \x \S^{n}$ bounded predictable process such that there exists $\theta>0$ satisfying $[\sigma_t] [\sigma_t]^\top \ge \theta \mathrm{I}_{n \x n}.$ For all $t \in [0,T],$ denote by
$$
    X_t
    =
    \xi
    +
    \int_0^t b_s \mathrm{d}s
    +
    \int_0^t \sigma_s \mathrm{d}W_s,\;\;\P\mbox{--a.e.}
$$
the following proposition is just an application of \cite[Proposition 4.2]{gyongy1986mimicking} (see also \cite{KrylovControlledDiffusion})
\begin{proposition}[equivalence of measures] \label{Prop:absolutelyContinuous}
    With the previous considerations, the measure $\nb$ on $\R^n \x [0,T]$ defined by
    $$
        \nb(\mathrm{d}x,\mathrm{d}t)
        :=
        \P \circ (X_t)^{-1}(\mathrm{d}x)\mathrm{d}t
    $$
    is equivalent to the Lebesque measure on $\R^n \x [0,T]$.
\end{proposition}

\medskip
\paragraph*{Approximation by convolution} We set $G \in C^{\infty}(\R^n;\R)$ with compact support satisfying $G \ge 0,$ $G(x)=G(-x)$ for $x \in \R^n,$ and $\int_{\R^n}G(y)\mathrm{d}y=1.$ Let us introduce $G_{\varepsilon}(x):={\varepsilon}^{-n} G({\varepsilon}^{-1}x)$ for all $x \in \R^n.$

\medskip
Let $X^k$ be the process defined by
    $$
        X^k_t
        =
        \xi
        +
        \int_0^t b^k_r \mathrm{d}r
        +
        \int_0^t \sigma^k_r \mathrm{d}W_r
        ~
        \mbox{for all}
        ~ t \in [0,T],\;\P\mbox{--a.e.},
    $$
    where there exists $\mathrm{D}>0$ s.t. for all $k$ and $t,$ $|\sigma^k_t|+|b^k_t|\le \mathrm{D},$ $\P$--a.e., $[\sigma^k_t] [\sigma^k_t]^{\top} \ge \theta \mathrm{I}_{n \x n},$ $\P$--a.e. In addition  $\E^\P[|\xi|^{p}] < \infty$ where $p \ge 1.$ 
    Also, we take $(\nb_t)_{t \in [0,T]} \in \Cc^n_{\Wc}$ such that $\nb_t(\mathrm{d}x)\mathrm{d}t$ is equivalent to the Lebesgue measure on $[0,T] \x \R^n,$ and for the weak topology, \begin{align*}
        \Lim_{k \to \infty} \Lc^\P(X^k_t)=\nb_t\;\;\mbox{for each}\;t \in [0,T].
    \end{align*}
    
    
    The following proposition shows that it is possible to approach some bounded measurable functions via smooth functions (bounded derivative functions) by using the marginal distributions of $X^k.$ We consider $(\varepsilon_k)_{k \in \N^*} \subset (0,\infty)$ such that $\Lim_{k \to \infty} \varepsilon_k=0.$ We pose $G_k(x)=G_{\varepsilon_k}(x)$ and for $\pi \in \Pc(\R^n),$ $\pi^{(k)}(x):=\int_{\R^n} G_k(x-y)\pi(\mathrm{dy}),$ for all $x \in \R^n.$

\begin{proposition}[regularization by convolution] \label{prop:conservation_to_the_limit}
    For any bounded Borel measurable function $\varphi:[0,T] \x \R^n \x \R^n \to \R^q$, such that for all $(t,z) \in [0,T] \x \R^n,$ $\varphi(t,.,z): y \in \R^n \to \varphi(t,y,z) \in \R^q$ is continuous, one has
     \begin{align} \label{convolution-app}
       \Lim_{k \to \infty}
        \int_0^T
        \int_{\R^n}
        \Big|
        \int_{\R^n} \varphi(t,x,y) 
        \frac{G_k(t,x-y)}{(\nb_t)^{(k)}(x)}\nb_t(\mathrm{d}y)
        -
        \varphi(t,x,x)
        \Big|\nb_t(\mathrm{d}x)
        \mathrm{d}t=0 
    \end{align}
    and
    $$
        \Lim_{k \to \infty}
        \int_0^T
        \bigg|
        \E^\P
        \bigg[
        \int_{\R^n} \varphi(t,X^k_t,y) 
        \frac{G_l(t,X^k_t-y)}{(\nb_t)^{(k)}(X^k_t)}\nb_t(\mathrm{d}y)\bigg]
        -
        \int_{\R^n}
        \varphi(t,x,x)
        \nb_t(\mathrm{d}x) \bigg|
        \mathrm{d}t=0.
    $$

\end{proposition}

\begin{proof}
Mention that, as $\nb_t(\mathrm{d}x)\mathrm{d}t$ is equivalent to the Lebesgue measure on $[0,T] \x \R^n,$ there exists Borel measurable function $c: [0,T] \x \R^n \to \R$ such that $c(s,z)>0$ $\mathrm{d}t \otimes \mathrm{d}x$ a.e. $(s,z) \in [0,T] \x \R^n,$ and $\nb_t(\mathrm{d}x)\mathrm{d}t=c(t,x)\mathrm{d}x\mathrm{d}t.$

First, let us prove the result \eqref{convolution-app}.
If 
\begin{align*} 
    A_{k}:=\int_0^T \int_{\R^n}
    \Big|
    \int_{\R^n}\varphi(t,x,y)\frac{G_{k}(x-y)}{(\nb_t)^{(k)}(x)} \nb_t(\mathrm{d}y)
    -
    \varphi(t,x,x)
    \Big|
    \nb_t(\mathrm{d}x) \mathrm{d}t,
\end{align*}
one finds that
\begin{align} \label{eq:first-limit}
    &\Big| A_{k}-\int_0^T\int_{\R^n} \Big|\int_{\R^n} \big\{\varphi(t,x,y)-\varphi(t,x,x) \big\}G_{k}(x-y)c(t,y)\mathrm{d}y \Big| \mathrm{d}x \mathrm{d}t \Big| \nonumber
    \\
    &=
    \Big| \int_0^T \int_{\R^n} \Big|\int_{\R^n} \big\{\varphi(t,x,y)-\varphi(t,x,x) \big\}G_{k}(x-y)c(t,y)\mathrm{d}y \Big| \Big(\frac{c(t,x)}{(\nb_t)^{(k)}(x)}-1 \Big) \mathrm{d}x \mathrm{d}t \Big| \nonumber
    \\
    &\le K 
    \Big| \int_0^T \int_{\R^n} \int_{\R^n} G_{k}(x-y)c(t,y)\mathrm{d}y \Big|\frac{c(t,x)}{(\nb_t)^{(k)}(x)}-1 \Big| \mathrm{d}x \mathrm{d}t \Big| \nonumber
    = K 
    \Big| \int_0^T \int_{\R^n} (\nb_t)^{(k)}(x) \Big|\frac{c(t,x)}{(\nb_t)^{(k)}(x)}-1 \Big| \mathrm{d}x \mathrm{d}t \Big| \nonumber
    \\
    & \le K 
    \Big| \int_0^T \int_{\R^n} \Big|c(t,x)-(\nb_t)^{(k)}(x) \Big| \mathrm{d}x \mathrm{d}t \Big|
    =
    K\Big| \int_0^T\int_{\R^n} \Big|c(t,x)-\int_{\R^n} G_{k}(x-y)c(t,y)\mathrm{d}y \Big| \mathrm{d}x \mathrm{d}t \Big| \to_{k \to \infty}=0,
\end{align}
where the first inequality is true because $\varphi$ is bounded and the last result is obtained by the classical result of approximation by convolution.

\medskip
Now, for all $(t,y,\delta) \in [0,T] \x \R^n \x \R^{*}_{+},$ $v(t,y,\delta):= \sup_{z | |y-z| \le \delta}|\varphi(t,y,y)-\varphi(t,z,y)|,$ notice that $\lim_{\delta \to 0} v(t,y,\delta)=0.$  Observe that
\begin{align*}
    &\int_0^T\int_{\R^n} \Big|\int_{\R^n} \big\{\varphi(t,x,y)-\varphi(t,y,y) \big\}G_{k}(x-y)c(t,y)\mathrm{d}y \Big| \mathrm{d}x \mathrm{d}t
    \\
    &=
    \int_0^T\int_{\R^n} \Big|\int_{\R^n} \big\{\varphi(t,x,y)-\varphi(t,y,y) \big\}\big(1_{|x-y| \le \delta} + 1_{|x-y| > \delta} \big)G_{k}(x-y)c(t,y)\mathrm{d}y \Big| \mathrm{d}x \mathrm{d}t
    \\
    & \le
    \int_0^T\int_{\R^n} v(t,y,\delta) \int_{\R^n} 1_{|x-y| \le \delta}G_{k}(x-y)c(t,y)\mathrm{d}y \mathrm{d}x \mathrm{d}t
    +
    K \int_0^T\int_{\R^n} \int_{\R^n} 1_{|x-y| > \delta}G_{k}(x-y)c(t,y)\mathrm{d}y \mathrm{d}x \mathrm{d}t
    \\
    & \le 
    \int_0^T\int_{\R^n} v(t,y,\delta) \int_{\R^n}G_{k}(x-y)c(t,y)\mathrm{d}y \mathrm{d}x \mathrm{d}t
    +K~T\int_{\R^n} 1_{|z| > \delta}G_{k}(z)\mathrm{d}z
    \\
    &\le 
    \int_0^T\int_{\R^n} v(t,y,\delta) c(t,y)\mathrm{d}y \mathrm{d}t
    +
    K~T\int_{\R^n} 1_{|z| > \delta}G_{k}(z)\mathrm{d}z
    ,
\end{align*}
it is well known, for each $\delta > 0,$ $\lim_{k \to \infty}\int_{\R^n} 1_{|z| > \delta}G_{k}(z)\mathrm{d}z=0,$ one gets that
\begin{align} \label{eq_convolution}
    &\limsup_{k \to \infty}\int_0^T\int_{\R^n} \Big|\int_{\R^n} \{\varphi(t,x,y)-\varphi(t,y,y)\}G_{k}(x-y)c(t,y)\mathrm{d}y \Big| \mathrm{d}x \mathrm{d}t \le \lim_{\delta \to 0} \int_0^T\int_{\R^n} v(t,x,\delta) c(t,x) \mathrm{d}x \mathrm{d}t=0,
\end{align}
the last inequality is true because of  Lebesgue's dominated convergence theorem. Finally, one has that
\begin{align*}
    \limsup_{k \to \infty}A_k
    =
    &\limsup_{k \to \infty}\int_0^T\int_{\R^n} \Big|\int_{\R^n} \{\varphi(t,x,y)-\varphi(t,x,x)\}G_{k}(x-y)c(t,y)\mathrm{d}y \Big| \mathrm{d}x \mathrm{d}t
    \\
    &=
    \limsup_{k \to \infty}\int_0^T\int_{\R^n} \Big|\int_{\R^n} \varphi(t,x,y)c(t,y)G_{k}(x-y)\mathrm{d}y- \int_{\R^n}\varphi(t,x,x)G_{k}(x-y)c(t,y)\mathrm{d}y \Big| \mathrm{d}x \mathrm{d}t
    \\
    &=
    \limsup_{k \to \infty}\int_0^T\int_{\R^n} \Big|\int_{\R^n} \varphi(t,y,y)c(t,y)G_{k}(x-y)\mathrm{d}y- \int_{\R^n}\varphi(t,x,x)G_{k}(x-y)c(t,y)\mathrm{d}y \Big| \mathrm{d}x \mathrm{d}t
    \\
    & \le \limsup_{k \to \infty}
    \int_0^T\int_{\R^n} \Big|\int_{\R^n} \varphi(t,y,y)c(t,y)G_{k}(x-y)\mathrm{d}y- \varphi(t,x,x)c(t,x) \Big| \mathrm{d}x \mathrm{d}t
    \\
    &~~~~~~~~~~~~+\limsup_{k \to \infty} \int_0^T\int_{\R^n} \Big|\varphi(t,x,x)c(t,x)- \int_{\R^n}\varphi(t,x,x)G_{k}(x-y)c(t,y)\mathrm{d}y \Big| \mathrm{d}x \mathrm{d}t
    \\
    & \le \limsup_{k \to \infty}
    \int_0^T\int_{\R^n} \Big|\int_{\R^n} \varphi(t,y,y)c(t,y)G_{k}(x-y)\mathrm{d}y- \varphi(t,x,x)c(t,x) \Big| \mathrm{d}x \mathrm{d}t
    \\
    &~~~~~~~~~~~~+ \limsup_{k \to \infty} K\int_0^T\int_{\R^n} \Big|c(t,x)- \int_{\R^n}G_{k}(x-y)c(t,y)\mathrm{d}y \Big| \mathrm{d}x \mathrm{d}t =0,
\end{align*}
where the first equality derived from \eqref{eq:first-limit},  the third equality follows from \eqref{eq_convolution} and we find $0$ because of approximation by convolution result. Therefore $\lim_{k \to \infty} A_{k}=0,$ then the first assertion is proved.

\medskip
For the second point, let $k_0 \in \N^*,$ one has that
\begin{align*}
    &S^k(\varphi):=
    \int_0^T
    \bigg|  \E^\P \bigg[
    \int_{\R^n} \varphi(t,X^k_t,y) 
    \frac{G_k(t,X^k_t-y)}{(\nb_t)^{(k)}(X^k_t)}\nb_t(\mathrm{d}y)\bigg]
    -
    \int_{\R^n}
    \varphi(t,x,x)
    \nb_t(\mathrm{d}x)  \bigg| \mathrm{d}t
    \\
    &\le
    \int_0^T
    \bigg| \E^\P\bigg[\int_{\R^n} \varphi(t,X^k_t,y) 
    \frac{G_k(t,X^k_t-y)}{(\nb_t)^{(k)}(X^k_t)}\nb_t(\mathrm{d}y)
    -
    \int_{\R^n} \varphi(t,X^k_t,y) 
    \frac{G_{k_0}(t,X^k_t-y)}{(\nb_t)^{(k_0)}(X^k_t)}\nb_t(\mathrm{d}y) \bigg]\bigg|\mathrm{d}t
    \\
    &\;\;+
    \int_0^T  \bigg| \E^\P\bigg[ \int_{\R^n} \varphi(t,X^k_t,y) 
    \frac{G_{k_0}(t,X^k_t-y)}{(\nb_t)^{(k_0)}(X^k_t)}\nb_t(\mathrm{d}y) \bigg]
    -
    \int_{\R^n}
    \varphi(t,x,x)
    \nb_t(\mathrm{d}x) \bigg|\mathrm{d}t.
\end{align*}

By \cite[Chapter 2 Section 3 Theorem 4]{KrylovControlledDiffusion} and Markov inequality, for each $R>0,$ there exists a constant $C>0$ depending only on $(\mathrm{D},\theta,T,R)$ satisfying
\begin{align*}
    S^k(\varphi)
    &\le
    C\;\int_0^T\int_{\R^n}
    \Big|
    \int_{\R^n} \varphi(t,x,y) 
    \frac{G_k(t,x-y)}{(\nb_t)^{(k})(x)}\nb_t(\mathrm{d}y)
    -
    \int_{\R^n} \varphi(t,x,y) 
    \frac{G_{k_0}(t,x-y)}{(\nb_t)^{(k_0)}(x)}\nb_t(\mathrm{d}y)
    \Big|^{n}1_{|x| \le R}
    \mathrm{d}x \mathrm{d}t
    \\
    &+
    T\frac{\E^\P[\sup_{t \in [0,T]}|X^k_t|^p]}{R^p}
    +
    \int_0^T  \bigg| \E^\P\bigg[ \int_{\R^n} \varphi(t,X^k_t,y) 
    \frac{G_{k_0}(t,X^k_t-y)}{(\nb_t)^{(k_0)}(X^k_t)}\nb_t(\mathrm{d}y) \bigg]
    -
    \int_{\R^n}
    \varphi(t,x,x)
    \nb_t(\mathrm{d}x) \bigg|\mathrm{d}t.
\end{align*}
By using the first statement of the proposition (see proof above), there exists $(k_j)_{j \in \N^*} \subset \N^*$ a sub--sequence such that:
\begin{align*}
    \Lim_{j \to \infty}\bigg|
        \int_{\R^n} \varphi(s,z,y) 
        \frac{G_{k_j}(t,x-y)}{(\nb_t)^{(k_j)}(x)}\nb_t(\mathrm{d}y)
        -
        \varphi(s,z,z)
        \bigg|
        =
        0,\;\nb_t(\mathrm{d}x)\mathrm{d}t\;\;\mbox{a.e.}\;\; (s,z) \in [0,T] \x \R^n.
\end{align*}
As $\nb_t(\mathrm{d}x)\mathrm{d}t$ is equivalent to the Lebesgue measure on $[0,T] \x \R^n,$ $ \Lim_{j \to \infty}\bigg|\int_{\R^n} \varphi(s,z,y) \frac{G_{k_j}(t,x-y)}{(\nb_t)^{(k_j)}(x)}\nb_t(\mathrm{d}y)-\varphi(s,z,z)\bigg|=0,$ $\mathrm{d}t \otimes \mathrm{d}x$ a.e. $(s,z) \in [0,T] \x \R^n.$ All these observations allow us to say, by Lebesgue's dominated convergence theorem
\begin{align*}
    \Limsup_{k_0 \to \infty} \Limsup_{k \to \infty}\int_0^T\int_{\R^n}
    \Big|
    \int_{\R^n} \varphi(t,x,y) 
    \frac{G_k(t,x-y)}{(\nb_t)^{(k)}(x)}\nb_t(\mathrm{d}y)
    -
    \int_{\R^n} \varphi(t,x,y) 
    \frac{G_{k_0}(t,x-y)}{(\nb_t)^{(k_0)}(x)}\nb_t(\mathrm{d}y)
    \Big|^{n}1_{|x| \le R}
    \mathrm{d}x \mathrm{d}t=0.
\end{align*}

Finally, combining the previous result with the weak convergence, $\Lim_{k \to \infty} \Lc^\P(X^k_t)=\nb_t$ for each $t \in [0,T],$ and an obvious application of the first statement of the proposition, one gets
\begin{align*}
    \Limsup_{k \to \infty} S^k(\varphi)
    &\le
    \Limsup_{k_0, k \to \infty}\; C\;\int_0^T\int_{\R^n}
    \Big|
    \int_{\R^n} h(t,x,y) 
    \frac{G_k(t,x-y)}{(\nb_t)^{(k)}(x)}\nb_t(\mathrm{d}y)
    -
    \int_{\R^n} h(t,x,y) 
    \frac{G_{k_0}(t,x-y)}{(\nb_t)^{(k_0)}(x)}\nb_t(\mathrm{d}y)
    \Big|^{n}1_{|x| \le R}
    \mathrm{d}x \mathrm{d}t
    \\
    &\;\;\;\;\;+
    \limsup_{l_0 \to \infty}
    \int_0^T \Big| \int_{\R^n}
    \int_{\R^n} h(t,x,y) 
    \frac{G_{l_0}(t,x-y)}{(\nb_t)^{(l_0)}(x)}\nb_t(\mathrm{d}y)
    \nb_t(\mathrm{d}x) \mathrm{d}t
    -
    \int_0^T \int_{\R^n}
    h(t,x,x)
    \nb_t(\mathrm{d}x) \Big| \mathrm{d}t
    \\
    &~~~~~~~~~~~~~~~~~~+
    T\frac{\sup_{k>0}\E^\P[\sup_{t \in [0,T]}|X^k_t|^p]}{R^p}
    \le\;\;
    T\frac{\sup_{k>0}\E^\P[\sup_{t \in [0,T]}|X^k_t|^p]}{R^p},
\end{align*}
as $\sup_{k > 0} \E^\P[\sup_{t \in [0,T]}|X^k_t|^p] < \infty,$ by taking $R \to \infty$, we deduce the result.

\end{proof}

\medskip
The next result is essentially an application of \Cref{prop:conservation_to_the_limit}. It states the result of \Cref{prop:conservation_to_the_limit} under a form usually used in the paper. Let us consider the map $(\hat b, \hat \sigma): [0,T] \x \R^n \x \Pc^n_U \x \Pc^n_U \x U \to \R^n \x \S^n$ and $\hat \qb \in \M((\Pc^n_U)^2)$ s.t. $\hat \qb_t(\Z_{\nb_t} \x \Pc^n_U)=1$ $\mathrm{d}t$--for almost every $t \in [0,T].$ Recall that $\Lim_{k \to \infty} \Lc(X^k_t)=\nb_t$ in $\Wc_p$ for all $t \in [0,T].$ We pose $\nb^k_t:=\Lc(X^k_t).$


\begin{corollary}
\label{lemm:appr_coef}
    One has that
    \begin{align*}
        \Lim_{k \to \infty}
        \Bigg[
        \int_0^T \int_{(\Pc^n_U)^2} \bigg[ \int_{\R^n} \big|
        K^k(r,x,m,m')\big|^{p}
        \nb^{k}_{r}(\mathrm{d}x) + \Wc_p \Big(H^k(z,m)(\mathrm{d}u)\nb^{k}_{r}(\mathrm{d}z),m(\mathrm{d}u,\mathrm{d}z)\Big)^{p} \bigg]\hat \qb_r(\mathrm{d}m,\mathrm{d}m')
        \mathrm{d}r
       \Bigg]
        =0,
    \end{align*}
where
\begin{align*}
    K^k(s,x,m,\nub)
    :=
    &\int_{\R^n \x U} \big[\hat b,\hat \sigma \hat \sigma^\top \big]\big(s,y,m,\nub,u\big) \overline{H}^{k}(x,m)(\mathrm{d}y,\mathrm{d}u)
    -
    \int_U \big[\hat b,\hat \sigma \hat \sigma^\top \big]\big(s,x,m,\nub,u\big) H^k(x,m)(\mathrm{d}u),
\end{align*}
with $\overline{H}^{k}(x,m)(\mathrm{d}y,\mathrm{d}u):=m(\mathrm{d}y,\mathrm{d}u)\frac{G_k(x-y)}{(m(U,\mathrm{d}z))^{(k)}(x)}$ and $H^k(x,m)(\mathrm{d}u):=\int_{\R^n}\overline{H}^{k}(x,m)(\mathrm{d}u,\mathrm{d}y).$
    
\end{corollary}

\begin{proof}
    As $\hat \qb_t(\Z_{\nb_t} \x \Pc^n_U)=1$ $\mathrm{d}t$--almost surely $t \in [0,T],$ using convex inequality and Proposition \ref{prop:conservation_to_the_limit},  
\begin{align*}
    &\Lim_{k \to \infty}
    \int_0^T \int_{(\Pc^n_U)^2} \int_{\R^n} \big|
        K^k(r,x,m,\nub)\big|^{p}
        \nb^{k}_{r}(\mathrm{d}x)\hat \qb_r(\mathrm{d}m,\mathrm{d}\nub)
        \mathrm{d}r
    \\
    &\le
    \Limsup_{ k \to \infty} 
    \int_0^T \int_{(\Pc^n_U)^2} \int_{\R^n}
    \int_{\R^n \x U}
    \Big |  \big[\hat b,\hat \sigma \hat \sigma^\top \big](r,x,m,\nub,u)
    \\
    &~~~~~~~~~~~~~~~~~~~~~~~~~~~~~~~~~~-
     \big[\hat b,\hat \sigma \hat \sigma^\top \big](r,y,m,\nub,u) \Big |^p\frac{G_{k}(x-y) }{(\nb_r)^{(k)}(x)} m^y(\mathrm{d}u)\nb_r(\mathrm{d}y) \nb^{k}_r(\mathrm{d}x)\hat \qb_r(\mathrm{d}m,\mathrm{d}\nub)
    \mathrm{d}r=0.
\end{align*}

For all bounded continuous function $h:\R^n \x U \to \R$, using Proposition \ref{prop:conservation_to_the_limit} again,
\begin{align*}
    &\lim_{k \to \infty} \int_0^T \int_{(\Pc^n_U)^2} \Big| \int_{\R^n \x U} h(x,u)
    H^k(x,m)(\mathrm{d}u)\nb^k_r(\mathrm{d}x)
    -
    \int_{\R^n \x U} h(z,u) m(\mathrm{d}z,\mathrm{d}u)\Big|\hat \qb_r(\mathrm{d}m,\mathrm{d}\nub) \mathrm{d}r
    \\
    &\le
    \lim_{k \to \infty} \int_0^T \int_{(\Pc^n_U)^2} \int_{\R^n}  \Big|  \int_{\R^n \x U} h(x,u)
    m^{y}_r(\mathrm{d}u)\frac{G_k(y-x)}{(\nb_r)^{(k)}(x)}\nb_r(\mathrm{d}y)
    -
    \int_{\R^n \x U} h(z,u) m(\mathrm{d}z,\mathrm{d}u)\Big|  \nb^{k}_r(\mathrm{d}x) \hat \qb_r(\mathrm{d}m,\mathrm{d}\nub) \mathrm{d}r=0,
\end{align*}
similarly to \cite[Theorem 1.1.2.]{stroock2007multidimensional}, one finds a countable family of bounded continuous functions $(h^k)_{k \in \N^*}$ characterizing the weak convergence, therefore by Lebesgue's dominated convergence,
\begin{align*}
    \lim_{k \to \infty} \sum_{q \ge 0} \int_0^T \int_{(\Pc^n_U)^2} \frac{1}{2^q}\Big| \int_{\R^n \x U} h^q(x,u)
    H^k(x,m)(\mathrm{d}u)\nb^k_r(\mathrm{d}x)
    -
    \int_{\R^n \x U} h^q(z,u) m(\mathrm{d}z,\mathrm{d}u)\Big|\hat \qb_r(\mathrm{d}m,\mathrm{d}\nub) \mathrm{d}r=0,
\end{align*}
then $\lim_{k \to \infty} \int_0^T \int_{(\Pc^n_U)^2} \Delta \Big(H^k(z,m)(\mathrm{d}u)\nb^{k}_{r}(\mathrm{d}z),m(\mathrm{d}u,\mathrm{d}z)\Big) \hat \qb_r(\mathrm{d}m,\mathrm{d}\nub)\mathrm{d}r=0,$ where $\Delta$ is the metric characterizing the weak convergence on $\Pc^n_U.$ As $[\hat b,\hat \sigma]$ are bounded and $\nu \in \Pc_{p'}(\R^n),$ for $(r,m) \in [0,T] \x \Pc^n_U$,
\[
    \lim_{K \to \infty} \sup_{k \in \N^*} \int_{|z|+ \rho(u_0,u) \ge K} |z|^p + \rho(u_0,u)^p \;\;H^k(z,m)(\mathrm{d}u)\nb^k_r(\mathrm{d}z)=0.
\]
This is enough to conclude that, $\Lim_{k \to \infty}\;\;\int_0^T \int_{(\Pc^n_U)^2} \Wc_p \Big(H^k(z,m)(\mathrm{d}u)\nb^{k}_{r}(\mathrm{d}z),m(\mathrm{d}u,\mathrm{d}z)\Big) \hat \qb_r(\mathrm{d}m,\mathrm{d}\nub)\mathrm{d}r=0.$

\end{proof}

\medskip
\paragraph*{Consequence of the regularization: a continuity property}
Now, we want to provide some properties satisfying by a regularized map.
Let $\psi: [0,T] \x \R^n \x \Cc^\ell \x (\Cc^n_{\Wc})^2 \x (\Pc^n_U)^2 \x U \longrightarrow \R^j$ be a Borel function, with $j \in \N^*.$ 
For each $\varepsilon>0,$ one defines the function $\psi^\varepsilon: \Cc^\ell \x (\Cc^n_{\Wc})^2 \x {\color{black}\Pc((\Pc^n_U)^2)} \x [0,T] \x \R^n \longrightarrow \R^j$ as follows: for every $(t,x,\bb,\pi,\beta,q) \in [0,T] \x \R^n \x \Cc^\ell \x (\Cc^n_{\Wc})^2 \x \Pc((\Pc^n_U)^2)$
\begin{align*} 
    \psi^\varepsilon[\bb,\pi,\beta,q](t,x)
    :=
    \int_{(\Pc^n_U)^2} \int_{\R^n} \int_U \psi(t,y,\bb_{t \wedge \cdot},\pi_{t \wedge \cdot},\beta_{t \wedge \cdot},m,\nub,u)\frac{G_\varepsilon(x-y)}{(m(\mathrm{d}z,U))^{(\varepsilon)}(x)}m(\mathrm{d}u,\mathrm{d}y)q(\mathrm{d}m,\mathrm{d}\nub),
\end{align*}
where for every $m \in \Pc^n_U,$ $(m(\mathrm{d}z,U))^{(\varepsilon)}(x):= \int_{\R^n} G_{\varepsilon}(x-y) m(\mathrm{d}y,U).$

\medskip
Notice that $\big| \psi^\varepsilon[\bb,\pi,\beta,q](t,x) \big| \le \sup_{z',\bb',\zeta',m',\nu',u'}\big|\psi(t,z',\bb',\zeta',m',\nub',u') \big|,$ for all $(\bb,\pi,\beta,q,t,x).$ Then if $\psi$ is bounded, $\psi^\varepsilon$ is  bounded uniformly in $\varepsilon > 0.$
Also, given $(t,\bb,\pi,\beta,q),$ for each $\varepsilon>0,$ the function $\R^n \ni x \to \psi^{\varepsilon}[\bb,\pi,\beta,q](t,x) \in \R^j$ belongs to $C^{\infty}_b(\R^n),$ hence the name of $regularization$. 

\medskip
Under additional conditions, we have shown in the previous \Cref{prop:conservation_to_the_limit}, in some sense, $``\lim_{\epsilon \to 0} \psi^{\epsilon}=\psi"$  (see Proposition \ref{prop:conservation_to_the_limit} for more details). The next result checks that given $\varepsilon > 0,$ the map $\psi^\varepsilon$ satisfies a general continuity property.

\begin{proposition}\label{lemma:continuity}
For any $\psi:[0,T] \x \R^n \x \Cc^\ell \x (\Cc^n_{\Wc})^2 \x (\Pc^n_U)^2 \x U \longrightarrow \R$ and $\phi: [0,T] \x \R^n \to \R$ two bounded continuous functions. For each $\varepsilon>0,$  the function
\begin{align*}
    \big( \bb,\vartheta,\pi,\beta,q \big) \in \Cc^\ell \x (\Cc^n_{\Wc})^3 \x \M \big((\Pc^n_U)^2 \big) \longrightarrow \int_0^T \int_{\R^n}  \psi^{\varepsilon}[\bb,\pi,\beta,q_t](t,x)\phi(t,x) \vartheta_t(\mathrm{d}x)\mathrm{d}t \in \R
\end{align*}
is continuous.
    
\end{proposition}

\begin{proof}
    Let $(\bb^k,\vartheta^k,\pi^k,\beta^k,q^k)_{k \in \N} \subset \Cc^\ell \x (\Cc^n_{\Wc})^3 \x \M((\Pc^n_U)^2)$ and $(\bb,\vartheta,\pi,\beta,q) \in \Cc^\ell \x (\Cc^n_{\Wc})^3 \x \M((\Pc^n_U)^2)$  verifying $\Lim_k (\bb^k,\vartheta^k,\pi^k,\beta^k,q^k)=(\bb,\vartheta,\pi,\beta,q).$ 
    Notice that,
    \begin{align*}
        &\int_0^T \int_{\R^n}  \psi^\varepsilon[\bb,\pi,\beta,q_t](t,x)\phi(t,x) \vartheta_t(\mathrm{d}x)\mathrm{d}t
        \\
        &=
        \int_0^T \int_{\R^n}  \int_{(\Pc^n_U)^2} \int_{\R^n \x U} \psi(t,y,\bb_{t \wedge \cdot},\pi_{t \wedge \cdot},\beta_{t \wedge \cdot},m,\nub,u)\frac{G_\varepsilon(x-y)}{(m(\mathrm{d}z,U))^{(\varepsilon)}(x)}m(\mathrm{d}u,\mathrm{d}y)q_t(\mathrm{d}m,\mathrm{d}\nub)\phi(t,x) \vartheta_t(\mathrm{d}x)\mathrm{d}t
        \\
        &=
        \int_0^T \int_{\R^n} \int_{\Cc^\ell \x (\Cc^n_{\Wc})^2}  \int_{(\Pc^n_U)^2} \int_{\R^n \x U} \psi(t,y,g,e,e',m,\nub,u)\phi(t,x)H^{\varepsilon}(x,m)(\mathrm{d}u,\mathrm{d}y)q_t(\mathrm{d}m,\mathrm{d}\nub) \vartheta_t(\mathrm{d}x)\Psi_t(\mathrm{d}g,\mathrm{d}e,\mathrm{d}e')\mathrm{d}t,
    \end{align*}
    where
    \begin{align*}
        H^{\varepsilon}(x,m)(\mathrm{d}u,\mathrm{d}y)
        :=
        \frac{G_\varepsilon(x-y)}{(m(\mathrm{d}z,U))^{(\varepsilon)}(x)}m(\mathrm{d}u,\mathrm{d}y)\;\mbox{and}\;\Psi_t(\mathrm{d}g,\mathrm{d}e,\mathrm{d}e')\mathrm{d}t:=\delta_{(\bb_{t \wedge \cdot},\pi_{t \wedge \cdot},\beta_{t \wedge \cdot})}(\mathrm{d}g,\mathrm{d}e,\mathrm{d}e')\mathrm{d}t.
    \end{align*}
    Next, we define
    \begin{align*}
        Z^k(\mathrm{d}u,\mathrm{d}y,\mathrm{d}m,\mathrm{d}\nub,\mathrm{d}g,\mathrm{d}e,\mathrm{d}e',\mathrm{d}x,\mathrm{d}t)
        :=
        \frac{1}{T}H^{\varepsilon}(x,m)(\mathrm{d}u,\mathrm{d}y)q^k_t(\mathrm{d}m,\mathrm{d}\nub)\vartheta^k_t(\mathrm{d}x) \delta_{(\bb^k_{t \wedge \cdot},\pi^k_{t \wedge \cdot},\beta^k_{t \wedge \cdot})}(\mathrm{d}g,\mathrm{d}e,\mathrm{d}e')\mathrm{d}t
    \end{align*}
    and
    \begin{align*}
        Z(\mathrm{d}u,\mathrm{d}y,\mathrm{d}m,\mathrm{d}\nub,\mathrm{d}g,\mathrm{d}e,\mathrm{d}e',\mathrm{d}x,\mathrm{d}t)
        :=
        \frac{1}{T}H^{\varepsilon}(x,m)(\mathrm{d}u,\mathrm{d}y)q_t(\mathrm{d}m,\mathrm{d}\nub) \vartheta_t(\mathrm{d}x)\Psi_t(\mathrm{d}g,\mathrm{d}e,\mathrm{d}e')\mathrm{d}t.
    \end{align*}
    Then $(Z^k)_{k \in \N}$ is a sequence of probability measures belonging to $\Pc \big(U \x \R^n \x (\Pc^n_U)^2 \x \Cc^\ell \x (\Cc^n_{\Wc})^2 \x \R^n \x [0,T] \big).$ As $\Lim_k (\bb^k,\vartheta^k,\pi^k,\beta^k,q^k)=(\bb,\vartheta,\pi,\beta,q),$ it is straightforward to see that $(Z^k)_{k \in \N}$ is relatively compact in $\Pc \big(U \x \R^n \x (\Pc^n_U)^2 \x \Cc^\ell \x (\Cc^n_{\Wc})^2 \x \R^n \x [0,T] \big)$ and each sub--sequence converges to $Z,$ therefore $(Z^k)_{k \in \N}$ converges to $Z$ in a weak sense. As the function $(t,y,\bb,e,e',m,\nub,u,x) \in [0,T] \x \R^n \x \Cc^\ell \x (\Cc^n_{\Wc})^2 \x (\Pc^n_U)^2 \x U \x \R^n \to \psi(t,y,\bb_{t \wedge \cdot},e,e',m,\nub,u)\phi(t,x) \in \R^n$ is bounded continuous, we can conclude.
\end{proof}

\subsection{Some properties of Fokker--Planck equation}

Let us recall a useful result on square root of matrices. Denote by $\S^+_n$ the set of symmetric positive definite matrices of dimension $n \in \N^*.$ The principal square root function is denoted by: $f: Q \in \S^+_n \mapsto f(Q):=Q^{1/2} \in \S^+_n.$

\begin{proposition}{\rm \cite[Theorem 6.2]{NJHigham08}}\label{prop:ineq_matrix}
    There exists a constant $C(n)$ depending only of the dimension $n \in \N^*$ such that for any $(A,B) \in \S^+_n \x \S^+_n$
    \begin{align*}
        |f(A)-f(B)| \le C(n) \big[\lambda_{\min}(A)^{1/2} + \lambda_{\min}(B)^{1/2} \big]^{-1}|A-B|,
    \end{align*}
    where $\lambda_{\min}(\cdot)$ is the smallest eigenvalue.
\end{proposition}

\medskip
Let $E$ and $E'$ be two Polish spaces and $[\overline{b},\overline{a}]:[0,T] \x \R^n \x C([0,T];E) \x \M(E') \to \R^n \x \S^{n \x n}$ be a bounded Borel functions s.t.:  for all $(t,\pi,\hat q) \in [0,T] \x  C([0,T];E) \x \M(E'),$
\begin{align} \label{assum_existence-FP}
    \mbox{the function}\;x \in \R^n \to [\overline{b},\overline{a}](t,x,\pi_{t \wedge \cdot},\hat q_{t \wedge \cdot}) \in \R^n \x \S^{n \x n}\;\;\mbox{belongs to}\;\; C^2_b(\R^n)\;\mbox{and}\;\overline{a} \ge \rho \mathrm{I}_{n},
\end{align}
for a certain $\rho>0.$

\medskip
Also, let us introduce, for all $\varphi \in C^2(\R^n),$ $\overline{\Lc}_t\varphi[\pi,\hat q](x) 
    := 
    \frac{1}{2}  \text{Tr}\big[\overline{a}(t,x,\pi,\hat q_{t \wedge \cdot}) \nabla^2 \varphi(x) 
    \big] 
    + \overline{b}(t,x,\pi,\hat q_{t \wedge \cdot})^\top \nabla \varphi(x).$

\begin{lemma} \label{lemma:existence-measurability_FP}
    Let $\nu \in \Pc_p(\R^n).$ There exists a Borel function $Z: C([0,T];E) \x \M(E') \to \Cc^{n}_{\Wc}$ s.t. if $(\Om, \F, \Fc, \P)$ is a filtered probability space supporting $(\mu_t)_{t \in [0,T]}$ a $E$--valued $\F$--adapted continuous process and $(\hat \Lambda_t)_{t \in [0,T]}$ a $\Pc(E')$--valued $\F$--predictable process, then, the unique $\Pc(\R^n)$--valued $(\sigma \{\mu_{t \wedge \cdot}, \hat \Lambda_{t \wedge \cdot} \})_{t \in [0,T]}$--adapted continuous process $(\vartheta_t)_{t \in [0,T]}$ solution of: $\vartheta \in \Cc^{n,p}_{\Wc},$ and for all $(t,f) \in [0,T] \x C^{2}_b(\R^n),$
    \begin{align} \label{eq-FP_general}
	    \langle f,\vartheta_t \rangle
	    ~=~
	    \int_{\R^n} f(y) \nu(\mathrm{d}y)
	    +
	    \int_0^t \int_{\R^n} \overline{\Lc}_r f[\mu,\hat \Lambda](x) \vartheta_r(\mathrm{d}x) \mathrm{d}r,\;\P\mbox{--a.e.}
    \end{align} satisfies
    \begin{align*}
        \vartheta_t=Z_{t}(\mu_{t \wedge \cdot},\hat \Lambda_{t \wedge \cdot}),\;\mbox{for all}\;t \in [0,T],\;\P\mbox{--a.e.}
    \end{align*}
    
\end{lemma}

\begin{proof}
    For the uniqueness of \eqref{eq-FP_general}, as the coefficients $[\overline{b},\overline{a}]$ verify \eqref{assum_existence-FP}, by a slight extension of (proof of) \Cref{lemm:reguralization_FP}, one gets that equation \eqref{eq-FP_general} has at most one solution.
    
    Let $W$ be a $\R^n$--valued $(\P,\F)$ Brownian motion and $\xi$ be a $\Fc_0$--random variable of law $\nu,$ in addition, $(\xi,W)$ are $\P$--independent of $(\mu,\hat \Lambda).$ Next, let us show the existence and find the function $Z.$ Combining \eqref{assum_existence-FP} and \Cref{prop:ineq_matrix}, for any $(t,\pi,\hat q),$ the application $x \in \R^n \to \big(\overline{a}(t,x,\pi_{t \wedge \cdot},\hat q_{t \wedge \cdot}) \big)^{1/2} \in \S^{n \x n}$ is Lipshitz, with a Lipschitz constant depends only on $\overline{a}.$ Therefore, there exists the $\R^n$--valued $\F$--adapted process $X$ unique strong solution of
\begin{align*}
    X_s
    =
    \xi
    +
    \int_0^s
    \overline{b}(r,X_r,\mu,\hat \Lambda) \mathrm{d}r
    +
    \int_0^s
    \big( \overline{a}(r,X_r,\mu,\hat \Lambda) \big)^{1/2} \mathrm{d}W_r
    ~
    \mbox{for all}~
    s \in [0,T].
\end{align*}

    It is well known that $X_t=H_t(\xi,W_{t \wedge \cdot}, \mu_{t \wedge \cdot},\hat \Lambda_{t \wedge \cdot}),$ for all $t \in [0,T],$ $\P$--a.e. where $H: \R^n \x \Cc^n  \x C([0,T];E) \x \M(E') \to \Cc^n$ is a Borel function (independent of $\P$).
    
    Denote by $\G:=(\Gc_t)_{t \in [0,T]}$ the filtration defined by $\Gc_t:=\sigma \{\mu_{t \wedge \cdot}, \hat \Lambda_{t \wedge \cdot} \},$ for all $t \in [0,T].$ As $(\xi,W)$ are $\P$--independent of $(\mu,\hat \Lambda),$ one has that: for all $t \in [0,T],$ $\Lc^{\P}(X_{t \wedge \cdot}|\Gc_t)=\Lc^{\P}(X_{t \wedge \cdot}|\Gc_T),$ $\P$--a.e. then by \cite[Lemma A.1]{djete2019mckean}, the process $(\beta_t)_{t \in [0,T]}$ is a $\Pc(\R^n)$--valued $\G$--adapted continuous process where $\beta: (t,\om) \in [0,T] \x \Om \to \Lc^{\P}(X_{t}|\Gc_t)(\om) \in \Pc(\R^n),$ and by It\^o's formula $(\beta_t)_{t \in [0,T]}$ is solution of equation  \eqref{eq-FP_general}. In addition, there exists a Borel function (independent of $\P$) $Z: C([0,T];E) \x \M(E') \to \Cc^{n}_{\Wc}$ such that: $\P$--a.e., for all $t \in [0,T],$ $\beta_t=Z_{t}(\mu_{t \wedge \cdot},\hat \Lambda_{t \wedge \cdot}).$ 
\end{proof}

\end{appendix}

\end{document}